\documentclass[11pt, a4paper]{amsart}

\usepackage[margin=2.5cm]{geometry}

\usepackage[all, cmtip]{xy}
\usepackage{enumerate}
\usepackage[backend=bibtex, style=alphabetic, maxbibnames=99, maxcitenames=2, giveninits=true]{biblatex}
\usepackage{graphicx}
\usepackage[dvipsnames]{xcolor}

\usepackage{mathtools}
\usepackage{soul}
\usepackage{mathrsfs}
\usepackage{caption}
\usepackage{wrapfig}
\usepackage{picinpar}
\usepackage{amsfonts,amsmath,amssymb, mathrsfs}
\usepackage{upgreek}
\usepackage{relsize}
\usepackage{bm}
\usepackage{tikz-cd}
\usepackage[bookmarks=true]{hyperref}
\usepackage{pdfcomment}
\usepackage{cancel}

\hypersetup{
colorlinks=true,
linkcolor=MidnightBlue,
citecolor=OliveGreen,
urlcolor=RoyalBlue
}

\providecommand{\shorttitle}[1]{}
\providecommand{\shortauthors}[1]{}

\DeclareMathOperator{\Ugrp}{U}
\DeclareMathOperator{\aut}{Aut}

\DeclareMathOperator{\id}{Id}

\DeclareMathOperator{\GL}{GL}

\newcommand{\zz}{\textbf{Z}}

\newcommand{\cc}{\textbf{C}}
\newcommand{\pp}{\textbf{P}}
\newcommand{\dd}{\textbf{D}}
\newcommand{\mdd}{\mathbb D}

\newcommand{\zd}{\zz/d\zz}
\newcommand{\ov}{\overline}
\newcommand{\wt}{\widetilde}

\newcommand{\lab}{\left\langle}
\newcommand{\rab}{\right\rangle}

\newcommand{\artindisk}{\emph{Artin disk}}
\newcommand{\artindisks}{\emph{Artin disks}}
\newcommand{\loopdisk}{\emph{loop disk}}
\newcommand{\loopdisks}{\emph{loop disks}}

\DeclareMathOperator{\pt}{\mathsf{pt}}

\newcommand{\bnp}{B_{n,\mathcal{P}}}
\newcommand{\forgmap}{\mathcal{F}}

\numberwithin{equation}{section}

\theoremstyle{plain}
\newtheorem{theorem}{Theorem}[section]
\newtheorem{proposition}[theorem]{Proposition}
\newtheorem{lemma}[theorem]{Lemma}
\newtheorem{corollary}[theorem]{Corollary}

\newtheorem*{thmA}{Theorem A}
\newtheorem*{thmB}{Theorem B}
\newtheorem*{theorem*}{Theorem}

\theoremstyle{definition}

\newtheorem{remark}[theorem]{Remark}

\newenvironment{proof}[1][\proofname]{\par\vspace{0.2cm}\noindent\textit{#1.}\hspace{0.5em}}{\hfill$\square$\vspace{0.2cm}\par}

\author{Athira E V}
\address{Kerala School of Mathematics, Kunnamangalam, Kozhikode, Kerala 673571, India}
\email{athiraev@ksom.res.in}
\author{Pranav Haridas}
\address{Kerala School of Mathematics, Kunnamangalam, Kozhikode, Kerala 673571, India}
\email{pranav@ksom.res.in}


\title[Geometric Monodromy and the Multivariate Burau Representation]{Geometric Monodromy of Mixed Braid Groups and the Multivariate Burau Representation}

\begin{document}
\begin{abstract}
We study the monodromy action of the mixed braid
group $\bnp$ on the first cohomology of cyclic
branched covers of $\pp^1$, which are mutually
determined by a partition of branch points by equal
ramification. The monodromy representation splits
into irreducible representations on the $t$-eigenspaces
of the deck transformation. For each, we construct
an explicit spanning set using lifts of Pochhammer
contours and figure-eight curves, and compute the
Hermitian intersection form. The
representation factors through a reduced mixed braid
group by dropping \emph{$t$-invisible} parts of the
partition (those with trivial local monodromy).
In this reduced representation, each generator acts
by a complex reflection when the corresponding
spanning class is non-isotropic, and by a unitary
transvection when it is isotropic.
Provided $\infty$ has non-trivial local monodromy,
the factored representation is isomorphic to the
reduced multivariate Burau representation evaluated
at $t$.
\end{abstract}

\maketitle

\setcounter{tocdepth}{1}
\tableofcontents

\section{Introduction}

Let $p \colon X \to \pp^1$ be a cyclic branched cover
with deck transformation $T$ of order $d$,
and let $\mathcal{P}$ be the partition of
$\{1,\dots,n\}$ grouping its branch points
by equal exponent (every partition, by suitably choosing
exponents, yields a cyclic cover as well).  
The mixed braid group $\bnp$ (defined in
Section~\ref{subsection:mixed_braid_group_definition})
lifts to $X$ and acts on $H^1(X)$ via
pullback.  Since the
action commutes with $T$, it preserves each
$t$-eigenspace $H^1(X)_t$ for
$t \in \mu(d)$ (the $d$-th roots of unity),
yielding finite-dimensional
unitary representations
$\rho_t \colon \bnp \to
\Ugrp(H^1(X)_t)$.

For each non-trivial eigenvalue $t$, we
construct an explicit spanning set for
$H^1(X)_t$ using lifts of figure-eight curves
and Pochhammer contours, and compute the full
Hermitian intersection matrix of these
classes.  Our two main results are:

\begin{enumerate}
	\item For each generator $\sigma$ of $\bnp$, the
representation $\rho_t$ acts by
$\rho_t(\sigma)(\omega)
	      = \omega + \Lambda_\sigma
	      \langle \omega, \omega_\sigma \rangle\,\omega_\sigma$,
where $\Lambda_\sigma$ is an explicit rational
polynomial in $t$
	      (Theorem~A).
	      Once the $t$-invisible blocks are dropped, this is
	      a complex reflection when $\omega_\sigma$ is
	      non-isotropic, and a unitary transvection when it
	      is isotropic.
	\item The representation $\rho_t$ factors
through the forgetful homomorphism
dropping all $t$-invisible blocks,
and the factored representation is
isomorphic to the reduced multivariate
Burau representation evaluated at $t$,
provided $t^{k_\infty} \neq 1$
	      (Theorem~B).
\end{enumerate}

McMullen \cite{MR3020148} established this
representation framework for the full
braid group using cyclic covers with
uniform ramification. Venkataramana
\cite{gassner_inventiones} studied
representations of the pure braid group
using cyclic covers with exponents
coprime to the degree. This paper extends
both: we generalize McMullen's framework
to arbitrary cyclic covers with
non-uniform ramification, and we
generalize Venkataramana's by studying
the mixed braid group without coprimality restrictions.
Chaudhuri and
Mukherjee \cite{ChaudhuriMukherjee2026}
proved a related result for specialized
Abelian covers.

Deligne and Mostow
\cite{DeligneMostow1986} studied the
configuration space of branch points with
fractional weights $k_i/d$; the mixed
braid group appears as its fundamental
group when points with equal weights are
grouped. They used the resulting
monodromy to construct lattices in
hyperbolic space.

Donovan and Segal
\cite{DonovanSegal2015} established
that $\bnp$ acts on derived categories
of deformed surface singularities;
Bezrukavnikov and Riche
\cite{BezrukavnikovRiche2012}
constructed analogous actions for the
affine braid group on the Springer
resolution.

\medskip\noindent\textbf{Setup and main results.}
A cyclic branched cover $X \to \pp^1$ of
degree $d$ is defined by
$w^d = \prod_{i=1}^n (z-b_i)^{k_i}$,
with $0 < k_i < d$ and
$\gcd(d,k_1,\dots,k_n)=1$.
The branch locus is
$B = \{b_1,\dots,b_n\}$ together with
$\infty$ when
$k_\infty \not\equiv 0 \pmod{d}$, where
$k_\infty \equiv -\sum_{i=1}^n k_i \pmod{d}$.

Let $B_n$ be the braid group on $n$ strands,
generated by half-twists
$\sigma_1,\dots,\sigma_{n-1}$.
Let $\Psi \colon B_n \to S_n$ be the canonical
permutation map assigning to each braid its induced
endpoint permutation $\tau=\Psi(\sigma)$. A braid
lifts to a homeomorphism of $X$ iff
$k_{\tau(i)} \equiv u k_i \pmod{d}$ for some
$u \in \zd^*$ \cite{Ghaswala}. Such a lift
pulls back to a linear automorphism of
$H^1(X)$, yielding a representation of the
liftable subgroup of $B_n$. When the $k_i$ are not all equal,
a braid that permutes points with different
exponents may fail to lift. The pure braid
group $P_n$ always lifts, but excludes braids
that permute points with equal exponent.

The natural intermediate object is the
\emph{mixed braid group}. Partition
$\{1,\dots,n\}$ into $m$ blocks
$\mathcal{P} = \{P_1,\dots,P_m\}$ by
grouping indices with equal exponent
($k_i = k_j$ iff $i,j$ belong to the same
block). Assign to each block $P_r$ an
exponent
$0 < k_1 < k_2 < \cdots < k_m < d$, and let
$\pt \colon \{1,\dots,n\} \to \{1,\dots,m\}$
send a branch point index to its block;
thus the exponent of $b_i$ is
$k_{\pt(i)}$.
The \textbf{mixed braid group} $\bnp \subset B_n$
is the preimage under $\Psi$ of the subgroup of $S_n$
that preserves every partition block. Consequently,
for $\sigma\in\bnp$ with $\tau=\Psi(\sigma)$,
$k_{\pt(\tau(i))}=k_{\pt(i)}$, and the lifting
criterion holds with $u=1$.

By \cite[Theorem~4]{MR1465028}, $\bnp$ is generated
by two types of mapping class elements. An Artin generator
$\sigma_i$ operates within a single partition block. Its
support, called an \artindisk, contains the adjacent branch
points $b_i$ and $b_{i+1}$ of that block, and it acts by a
right, counterclockwise half-twist interchanging them.
A loop generator $A_{i,j}$ operates between distinct blocks.
Its support, called a \loopdisk, contains the last branch
point of $P_i$ and the first branch point of $P_j$, and it
acts by a positive, right full twist around the two points.
See Section~\ref{section:mixed_braid_group_lifts} for details.

\medskip\noindent
Every $\sigma \in \bnp$ lifts canonically
to $\wt\sigma \colon X \to X$ commuting with $T$
(Propositions~\ref{prop:lift-commutes-T}
and~\ref{prop:unique-lift-near-infty}). The pullback
$(\wt\sigma^{-1})^*$ preserves each $t$-eigenspace
$H^1(X)_t$, yielding a unitary representation
$\rho_t \colon \bnp \to
\Ugrp(H^1(X)_t)$ for each
$t \in \mu'(d) = \mu(d)\setminus\{1\}$.
We call a block $P_i$ \emph{$t$-invisible}
if the $t$-eigenspace of its local
cohomology vanishes
(Proposition~\ref{Prop:Dim_Disk_2_branch_points};
see also Remark~\ref{rem:t_invisible_blocks}).
The Chevalley--Weil theorem
\cite{ChevalleyWeil} gives
$\dim H^1(X)_t = n-1-r_t$
(or $n-2-r_t$ if $t^{k_\infty}=1$),
where $r_t$ is the sum of $|P_i|$
over all $t$-invisible blocks
(Proposition~\ref{prop:dim_cyclic_cover}).

We construct $n-1-r_t$ explicit classes
$\omega_1,\dots,\omega_{n-1-r_t}$ spanning $H^1(X)_t$,
one for each adjacent pair of consecutive visible branch
points. Each class is compactly supported in the corresponding
\artindisk\ or \loopdisk. For an \artindisk, we take a
figure-eight base contour enclosing its two branch points; for
a \loopdisk, we take a Pochhammer base contour enclosing its
two branch points. We lift the contour to the cyclic cover $X$;
the lift has several components. An $\ov t$-polynomial linear
combination of these components produces a $1$-cycle in $X$,
whose Poincar\'e dual is the class $\omega_k$ for that pair.
When a contiguous run of $t$-invisible blocks
$P_j,\dots,P_l$ occurs, we omit all classes whose support
intersects those blocks and instead insert a single general
loop class, supported on a disk stretching from the branch
point just before $P_j$ to the branch point just after $P_l$.
These classes span $H^1(X)_t$; when $t^{k_\infty}=1$, there
is one linear relation among them
(Proposition~\ref{thm:spanning_via_intersection}). Their
Hermitian intersection matrix is tridiagonal
(Proposition~\ref{prop:spanning_intersections}).

To study the representation, fix a
generator $\sigma$ of $\bnp$ and let
$\omega_\sigma$ be the class obtained by
the same contour-lift construction described
above, applied to the support disk of
$\sigma$. Since
$\sigma$ acts as the identity outside
its support disk, it fixes every
spanning class whose support is disjoint
from that disk.
The action of $\sigma$ on $H^1(X)_t$ is therefore determined
entirely by its effect on the few spanning classes whose
supports intersect its support disk. These effects are computed
by evaluating the geometric monodromy action on the lifted
contours via the intersection pairing
(see Lemmas~\ref{lem:artin_self_action},
\ref{lem:mixed_self_action}, and
\ref{lem:adjacent_spanning_class_action}).

\begin{thmA}[The Irreducible Representation of $\bnp$]
Let $t \in \mu'(d)$ be a non-trivial $d$-th root of unity, and let
$\rho_t \colon \bnp \to \Ugrp(H^1(X)_t, \langle\cdot,\cdot\rangle)$
be the monodromy representation. For a generator $\sigma$ of $\bnp$,
let $\omega_\sigma$ be the associated class supported in the support
disk of $\sigma$. Then
	\begin{align*}
		\rho_t(\sigma)(\omega)
		= \omega + \Lambda_\sigma \langle \omega, \omega_\sigma \rangle\,\omega_\sigma
	\end{align*}
for all $\omega \in H^1(X)_t$, where the coefficient
	\begin{align*}
		\Lambda_\sigma
		=\frac{\lambda_\sigma-1}
		{\langle\omega_\sigma,\omega_\sigma\rangle}
	\end{align*}
evaluates to the polynomial
	\begin{align*}
		\Lambda_\sigma = \begin{cases}
			           \sqrt{-1}\,(1-t^{k_{\pt(i)}})
			           & \text{if } \sigma = \sigma_i \text{ is an Artin generator}, \\[4pt]
			           \sqrt{-1}\,(1-t^{k_i})(1-t^{k_j})
			           & \text{if } \sigma = A_{i,j} \text{ is a loop generator}.
		           \end{cases}
	\end{align*}
Moreover, $\rho_t$ is irreducible.
\end{thmA}

After cancellation, $\Lambda_\sigma$ extends
polynomially across $t^{k_i}=-1$. Hence the
reflection formula remains defined when
$\langle\omega_\sigma,\omega_\sigma\rangle=0$.

A distinguishing feature of our setting is the presence of
\emph{$t$-invisible blocks}: blocks $P_i$ whose local
monodromy evaluates to the identity on the $t$-eigenspace
($t^{k_i} = 1$). When $t^{k_i}=1$, the deck
transformation acts as the identity on the
$t$-eigenspace, so every cohomology class
supported near $P_i$ vanishes.
The representation $\rho_t$ factors through
the forgetful map that drops all strands
indexed by $t$-invisible blocks, inducing
a representation $\bar\rho_t$ of the reduced
mixed braid group on the remaining
distinguished points.
A generator whose support disk meets a $t$-invisible
block acts as the identity. For $\bar\rho_t$, where no
block is $t$-invisible, every generator acts by a complex
reflection when $\omega_\sigma$ is non-isotropic, and by a
unitary transvection when $\omega_\sigma$ is isotropic.

This factorization is the key to identifying $\rho_t$ with
the classical multivariate Burau representation, defined
concretely via Fox calculus.

\begin{thmB}[Isomorphism with the Multivariate Burau Representation]
Let $t \in \mu'(d)$ with $t^{k_\infty} \neq 1$.
The geometric representation $\rho_t$ factors through the
forgetful homomorphism dropping all $t$-invisible blocks,
and the factored representation $\bar{\rho}_t$ is isomorphic
to the reduced multivariate Burau representation $\wt\beta_t$
of the reduced mixed braid group.
\end{thmB}

The isomorphism is established by an explicit change-of-basis
matrix. For a general loop generator $A_{i,j}$, the class
$\omega_{i,j}$ expands as a linear combination of spanning
classes (Proposition~\ref{prop:mixed_linear_combination}).
Theorem~A is proved in Section~\ref{section:action_mbg}, while
Theorem~B is proved in
Section~\ref{section:multivariate_burau_identification}.

Chaudhuri and Mukherjee \cite{ChaudhuriMukherjee2026}
independently study mixed braid group representations
via special Abelian branched covers with deck group
$\zz/d_1\zz \times \cdots \times \zz/d_m\zz$,
generalizing McMullen's cyclic covers (recovered when $m=1$)
but not encompassing arbitrary cyclic covers.
Their construction also decomposes cohomology into
eigenspaces and uses figure-eight and Pochhammer
contours to construct spanning sets. The present
cyclic-cover setting additionally allows
$t$-invisible blocks and the associated
factorization through a smaller mixed braid group.
Our matrices are obtained directly from the action
on lifted contours.

\medskip\noindent\textbf{Organization.}
Section~\ref{section:mixed_braid_group_lifts} defines the
cyclic cover $X$, realizes $\bnp$ as a mapping class subgroup,
and constructs canonical lifts of its generators to $X$.
Section~\ref{section:cohomology_cyclic_cover} builds the
cohomological representation on the $t$-eigenspaces of
$H^1(X)$ and computes their dimensions and Hodge signatures.
Section~\ref{section:generators} studies the local geometry
of support disks and their model cyclic covers.
Section~\ref{section:intersection_pairings} lifts these to
global classes on $X$, computes their Hermitian intersection
matrix, and proves they span $H^1(X)_t$.
Section~\ref{section:action_mbg} computes the monodromy
action of each generator and proves Theorem~A.
Section~\ref{section:multivariate_burau_identification}
handles the factorization for $t$-invisible blocks and
establishes the isomorphism with the multivariate Burau
representation (Theorem~B).

\section{The Mixed Braid Group and Lifting to Cyclic Covers}\label{section:mixed_braid_group_lifts}

We construct a degree-$d$ cyclic branched cover
$X \to \pp^1$ whose branch locus is $B$ and whose
monodromy encodes the partition data
$k_1,\dots,k_m$.  The mixed braid group $\bnp$ is
then realized as a subgroup of the mapping class
group of the plane relative to the finite branch
locus.  We prove that every mixed braid lifts to
$X$, that each lift commutes with the deck
transformation $T$
(Proposition~\ref{prop:lift-commutes-T}), and that
a unique canonical lift is fixed by requiring the
identity action near the fiber at infinity
(Proposition~\ref{prop:unique-lift-near-infty}).

\subsection{The cyclic cover and its
monodromy}\label{subsection:cyclic_cover_description}
Our ultimate goal is to study the representations of the
mixed braid group $\bnp$ that arise from its action on the
cohomology of certain covering spaces.  A cyclic cover
determines a partition of its branch points by grouping
those with equal exponent. Conversely, every partition,
by suitably choosing exponents, yields such a cover.
We therefore work with the cover defined by the
partition $\mathcal{P}$ and the exponents assigned to
its blocks.

We construct a cyclic branched cover
$p \colon X \to \pp^1$ of degree $d$ as the
compact Riemann surface associated to
\begin{equation*}
w^d = f(z),
	\qquad
f(z) = \prod_{i=1}^n (z - b_i)^{k_i},
\end{equation*}
where $0 < k_i < d$ and
$\gcd(d,k_1,\dots,k_n) = 1$.  Set
$k_\infty \equiv -\sum_{i=1}^n k_i \pmod{d}$.
The finite branch locus is
$B' := \{b_1, \dots, b_n\}$; the total branch
locus is $B := B' \cup \{\infty\}$ if
$k_\infty \not\equiv 0 \pmod{d}$, and $B := B'$
otherwise.  The affine curve may be singular
where $z = b_i$ and $w = 0$; $X$ is its
normalization and compactification. Away from
the preimage $\wt B := p^{-1}(B)$, the projection
$p(z,w) = z$ is an unramified covering map. It is
a folklore result that every cyclic cover of $\pp^1$
arises in this manner.

Let $0 < k_1 < k_2 < \cdots < k_m < d$ be the
distinct exponents appearing among the branch points.
Group the branch points by equal exponent. Let
$\mathcal{P} = \{P_1,\dots,P_m\}$ be the partition
of $\{1,\dots,n\}$ where $i,j$ belong to the same
block iff the exponent of $b_i$ equals that of $b_j$.
Let $\pt \colon \{1,\dots,n\} \to \{1,\dots,m\}$ map
each index to its block.

\begin{remark}
From now on we index exponents by the
partition: the exponent of a branch point $b_i$ is
denoted $k_{\pt(i)}$.  The raw list
$k_1,\dots,k_n$ is no longer used.
\end{remark}

The defining equation then takes the form
\begin{equation}\label{eqn:GeneralCyclicCover}
w^d = \prod_{i=1}^n (z - b_i)^{k_{\pt(i)}}.
\end{equation}
Conversely, every partition, by suitably choosing
exponents, yields a cyclic cover of $\pp^1$.

Remove the preimage of the branch locus to obtain the
punctured Riemann surface
$X' := X \setminus \wt B$.  The restriction
$p \colon X' \to \cc \setminus B'$ is an
unbranched, regular $d$-sheeted covering map.  Its
deck transformation group is cyclic of order $d$:
\begin{equation}\label{eqn:Deck_group}
G := \operatorname{Deck}(X/\pp^1)
	\cong \zd,
\end{equation}
generated by $T(z,w) = (z, \zeta w)$ where
$\zeta := e^{2\pi i/d}$.

The unbranched cover $X'$ is uniquely determined by a
monodromy homomorphism $\Phi$ from the fundamental group
$\pi_1(\cc \setminus B', z_0)$ of the base space to the
deck group $G \cong \zd$. Let $\ell_i$ be a standard
simple counterclockwise loop based at $z_0$ that winds
once around the branch point $b_i$. The elements
$[\ell_1], \dots, [\ell_n]$ generate the fundamental
group, and $\Phi$ maps each generator $[\ell_i]$ to the
local monodromy exponent $k_{\pt(i)} \pmod{d}$.
Because the target
group $\zd$ is abelian, $\Phi$ necessarily factors through
the abelianization of the fundamental group, which is the
first homology group $H_1(\cc \setminus B')$. By abuse of
notation, we use $\Phi$ to denote both this fundamental
group homomorphism and the induced map on homology,
$\Phi \colon H_1(\cc \setminus B') \to \zd$. It evaluates
on the basis elements as:
\begin{equation}\label{eqn:monodromy_map}
	\Phi([\ell_i]) = k_{\pt(i)} \pmod{d}.
\end{equation}
This evaluation determines the monodromy homomorphism and
fixes the topological structure of the cover.
Let $\Phi_\zz \colon H_1(\cc\setminus B')\to\zz$ be given by
$\Phi_\zz([\ell_i])=k_{\pt(i)}$. Then $\Phi$ is the reduction
of $\Phi_\zz$ modulo $d$.

\subsection{The Mixed Braid Group}
\label{subsection:mixed_braid_group_definition}

To define the generators geometrically, fix a reference
configuration of branch points
\begin{align*}
b_1 < b_2 < \cdots < b_n
\end{align*}
on the real axis. We treat braids as homeomorphisms of the
punctured disk and define group multiplication by functional
composition: $(fg)(z)=f(g(z))$. Thus, in a product of
braids, the rightmost factor is applied first.

The \textbf{classical braid group} $B_n$ on $n$ strands
admits the standard presentation
\begin{equation}\label{eqn:ClassicalBraidGroup}
B_n = \bigl\langle \sigma_1,\dots,\sigma_{n-1} \;\big|\;
  \sigma_i\sigma_j = \sigma_j\sigma_i \;\; (|i-j|\geq 2),\;\;
  \sigma_i\sigma_{i+1}\sigma_i = \sigma_{i+1}\sigma_i\sigma_{i+1}
  \bigr\rangle.
\end{equation}
Each generator $\sigma_i$ is the right, counterclockwise
half-twist exchanging the $i$-th and $(i+1)$-st strands.
Thus the left point moves to the right through the lower
half-plane, while the right point moves to the left through
the upper half-plane.

As introduced in
Subsection~\ref{subsection:cyclic_cover_description},
the partition $\mathcal{P} = \{P_1,\dots,P_m\}$ groups the
$n$ branch points according to their local monodromy exponents
$k_{\pt(i)}$, where $0 < k_1 < \cdots < k_m < d$ are the
distinct exponent values. After suitably renumbering the
points, we assume the blocks are contiguous:
$P_r = \{h_{r-1}+1,\dots,h_r\}$, where
$1 \leq h_1 < h_2 < \cdots < h_m = n$ and $h_0 = 0$.
The exponent of a branch point $b_i$ is $k_{\pt(i)}$.
Applying an orientation-preserving homeomorphism if
necessary, we assume without loss of generality that
$b_i = i$ for $1 \leq i \leq n$.

Given such a partition $\mathcal{P}$, the \textbf{mixed
braid group} with respect to $\mathcal{P}$, denoted $\bnp$,
is the subgroup of the classical braid group $B_n$
consisting of braids whose induced permutation $\tau \in S_n$
preserves each block of the partition setwise (equivalently,
$\pt(\tau(i)) = \pt(i)$ for all $i$).

For this partition, the mixed braid group $\bnp$ is
generated by (see \cite[Theorem~4]{MR1465028}):
\begin{equation}\label{eqn:GeneratorofMixedBraidGroup}
	\sigma_i \quad \text{for } i \notin \{h_1, \dots, h_{m-1}\}, \quad \text{and} \quad
A_{i,j} \quad \text{for } 1 \leq i < j \leq m,
\end{equation}
where $\sigma_i$ are the standard generators of $B_n$. Set
$\beta_{i,j} = \sigma_{h_{j-1}}\cdots\sigma_{h_i+1}.$
For $j=i+1$, this is the empty product, so
$\beta_{i,i+1}=1$.
Then $A_{i,j} = \beta_{i,j}\sigma_{h_i}^2\beta_{i,j}^{-1}.$
Manfredini also determined a complete set of relations for
these generators. We do not reproduce them here. The
representation $\rho_t$ is defined geometrically by
lifting braids and applying inverse pullback, which
directly yields a homomorphism independently of any
presentation of $\bnp$.

For every braid in the mixed braid group $\bnp$, there
exists a lift to the cyclic cover that commutes with the
deck generator $T$ (Proposition~\ref{prop:lift-commutes-T}).
Because the partition $\mathcal{P}$ determines the cyclic
cover of $\pp^1$ through \eqref{eqn:GeneralCyclicCover},
as detailed in Subsection~\ref{subsection:liftability},
$\bnp$ serves as a natural group associated to this cover
and its chosen deck transformation.

Every braid in $B_n$ determines an element of
$\operatorname{Mod}_c(\cc,B')$; see
\cite[Section~9.1.3]{FarbMargalit}. Hence $\bnp$ is a
mapping class subgroup. An \textbf{Artin generator}
$\sigma_i$ swaps the adjacent branch points $b_i$ and
$b_{i+1}$ within one block by a right, counterclockwise
half-twist. On homology it transposes $[\ell_i]$ and
$[\ell_{i+1}]$.

Let $D_{h_i}$ be the standard disk supporting
$\sigma_{h_i}^2$. The \textbf{loop generator} $A_{i,j}$
is the positive, right full twist supported in
$D_{i,j}:=\beta_{i,j}(D_{h_i}).$
The disk $D_{i,j}$ passes below every intermediate
branch point. It contains exactly the branch points
$b_{h_i}$ and $b_{h_{j-1}+1}$.
In general, a mixed braid acts on
$H_1(\cc\setminus B')$ by permuting the classes
$[\ell_1],\dots,[\ell_n]$ according to its induced
permutation.

For any block $P_i$, removing its $n_i := |P_i|$ strands
yields a well-defined mixed braid on $n - n_i$ strands.
The forgetful homomorphism
\begin{align*}
	\forgmap_i \colon \bnp
	\longrightarrow
B_{n-n_i, \,\mathcal{P}\setminus\{P_i\}}
\end{align*}
erases the strands of $P_i$.
Its kernel is the normal subgroup generated by
the internal Artin generators $\sigma_k$
($h_{i-1} < k < h_i$) and all loop generators
involving $P_i$ (namely $A_{i,j}$ for $j > i$ and
$A_{\ell,i}$ for $\ell < i$); these generators
represent the motions that become trivial upon
removing the strands of $P_i$.
For a root of unity $t \in \mu'(d)$, let
$\mathcal{I}_t = \{i : t^{k_i} = 1\}$ be the
set of block indices with $t^{k_i}=1$, and let
$r_t = \sum_{i \in \mathcal{I}_t} |P_i|$ be
the total number of branch points in
those blocks.  Define the
\textbf{reduced partition}
\begin{align*}
	\mathcal{P}_t
	= \mathcal{P} \setminus \{P_i : t^{k_i} = 1\},
\end{align*}
set $n_t = n - r_t$, and let
$m_t = |\mathcal{P}_t| = m - |\mathcal{I}_t|$.
Then $\mathcal{P}_t$ is a partition of
$n_t$ strands into $m_t$ blocks, each
satisfying $t^{k_j} \neq 1$.  When
$t^{k_i} \neq 1$ for all $i$, one has
$\mathcal{P}_t = \mathcal{P}$, $n_t = n$, and
$B_{n_t,\mathcal{P}_t} = \bnp$.
The forgetful homomorphism
\begin{equation}\label{eqn:forgmap_t}
	\forgmap_t \colon \bnp
	\longrightarrow
B_{n_t, \mathcal{P}_t}
\end{equation}
drops all blocks $P_i$ with $t^{k_i}=1$.
It is obtained by composing the relevant
maps $\forgmap_i$, and will be employed in
Section~\ref{section:action_mbg}
(see Remark~\ref{rem:factorization_invisible}).
Its kernel is the normal subgroup generated
by the internal Artin generators of each
$P_i$ with $t^{k_i}=1$ and all loop
generators having such a block as an
endpoint.

\subsection{Liftability and Canonical Lifts}
\label{subsection:liftability}
A homeomorphism of $\cc \setminus B'$ lifts to the
unbranched cover $X'$ if and only if it preserves the
monodromy homomorphism $\Phi$ up to a unit in $\zd$.
By \cite[Lemma~2.1]{Ghaswala}, a braid $\sigma \in B_n$
is liftable if and only if there exists
$u \in \zd^*$ with
\begin{align*}
k_{\pt(\tau(i))} = u \, k_{\pt(i)} \quad \text{for all } 1 \leq i \leq n,
\end{align*}
where $\tau \in S_n$ is the permutation induced by $\sigma$.

For any mixed braid $\sigma \in \bnp$, its induced
permutation satisfies $\pt(\tau(i)) = \pt(i)$ for all $i$.
Hence $k_{\pt(\tau(i))} = k_{\pt(i)}$. Setting $u=1$
in the lifting criterion guarantees that every element of
the mixed braid group lifts to the cyclic cover $X$.
Since $\sigma_*[\ell_i]=[\ell_{\tau(i)}]$, we get
$\Phi_\zz(\sigma_*[\ell_i])=k_{\pt(\tau(i))}=k_{\pt(i)}$, and
hence $\Phi_\zz\circ\sigma_*=\Phi_\zz$.

We will now prove that for any mixed braid, its lifts
commute with the deck transformation $T$.

\begin{proposition}\label{prop:lift-commutes-T}
Let $\sigma \in \bnp$ be a mixed braid. Then any lift
$\wt\sigma$ of $\sigma$ to the cover $X$ commutes
with the deck transformation $T$.
\end{proposition}
\begin{proof}
Let $\tau \in S_n$ be the permutation induced by
$\sigma$.  Since $\sigma$ preserves the partition
blocks, $\pt(\tau(i)) = \pt(i)$ and therefore
$k_{\pt(\tau(i))} = k_{\pt(i)}$.  Put $j := \tau(i)$.

The conjugate $\wt\sigma \circ T \circ \wt\sigma^{-1}$
is a lift of the identity on $\cc \setminus B'$, hence
lies in $\langle T \rangle$.  There exists an integer
$\alpha$ with
	\begin{align*}
		\wt\sigma \circ T^{k_{\pt(i)}}
		\circ \wt\sigma^{-1} = T^{\alpha}.
	\end{align*}
We determine $\alpha$ by evaluating at a point.

Fix $\wt z_0 \in p^{-1}(z_0)$ and set
$\wt z_1 := \wt\sigma^{-1}(\wt z_0)$, with images
$z_0$ and $z_1 = \sigma^{-1}(z_0)$ under $p$.
Choose a path $\wt\delta$ in $X'$ from $\wt z_0$
to $\wt z_1$ and put $\delta := p(\wt\delta)$.
Define the loop
	\begin{align*}
		\ell'_i := \overline{\delta} \cdot \ell_i
		\cdot \delta
	\end{align*}
based at $z_1$.  The loop $\ell'_i$ is freely
homotopic to $\ell_i$, so
$\Phi([\ell'_i]) = k_{\pt(i)}$.  Let $\wt\ell'_i$ be the
lift of $\ell'_i$ starting at $\wt z_1$.  By
definition of the monodromy representation,
$\wt\ell'_i(1) = T^{k_{\pt(i)}}(\wt z_1)$.

Apply $\wt\sigma$ to $\wt\ell'_i$.  The resulting
path $\wt\sigma \circ \wt\ell'_i$ lifts
$\sigma \circ \ell'_i$ and starts at
$\wt\sigma(\wt z_1) = \wt z_0$.  The base loop
$\sigma \circ \ell'_i$ is based at
$\sigma(z_1) = z_0$ and winds around
$\sigma(b_i) = b_j$; it is homologous to $\ell_j$.
Therefore $\Phi([\sigma \circ \ell'_i]) = k_{\pt(j)}$, and
$(\wt\sigma \circ \wt\ell'_i)(1)
	= T^{k_{\pt(j)}}(\wt z_0)$.

Combine the above equalities:
	\begin{align*}
		\bigl(\wt\sigma \circ T^{k_{\pt(i)}}
			\circ \wt\sigma^{-1}\bigr)(\wt z_0)
		&= \wt\sigma\bigl(T^{k_{\pt(i)}}(\wt z_1)\bigr) \\
		&= \wt\sigma\bigl(\wt\ell'_i(1)\bigr) \\
		&= (\wt\sigma \circ \wt\ell'_i)(1) \\
		&= T^{k_{\pt(j)}}(\wt z_0).
	\end{align*}
Hence $\alpha = k_{\pt(j)}$.  Because $k_{\pt(i)} = k_{\pt(j)}$,
$\wt\sigma \circ T^{k_{\pt(i)}}
	= T^{k_{\pt(i)}} \circ \wt\sigma$ for all $i$.

The hypothesis $\gcd(d, k_1, \dots, k_m) = 1$ implies
that $\{T^{k_{\pt(i)}}\}_{i=1}^n$ generates $\langle T
	\rangle \cong \zd$.  Since $\wt\sigma$ commutes with
each generator, it commutes with $T$.
\end{proof}

A liftable braid generally admits $d$ distinct lifts,
differing by deck transformations. To obtain a
well-defined representation, we must uniquely specify a
``canonical'' lift. Since every element of $\bnp$ is
compactly supported in $\cc$, we may fix a disk
neighborhood $\dd_\infty$ of $\infty$ in $\pp^1$ on which
all braid group elements act as the identity. 
% This disk
% will be called the \emph{disk at infinity}.

\begin{proposition}\label{prop:unique-lift-near-infty}
Let $\sigma \in \bnp$ be a mixed braid, viewed as an
element of $\operatorname{Mod}_c(\cc, B')$. Then
$\sigma$ has a unique canonical lift $\wt\sigma : X
	\to X$ that acts as the identity on a neighborhood of
$\wt\infty := p^{-1}(\infty)$.
\end{proposition}
\begin{proof}
The preimage $\wt\dd_\infty := p^{-1}(\dd_\infty)$
has $e_\infty := \gcd(d, k_\infty)$ connected
components, each a disk.  The deck group acts
transitively with stabilizer $\langle
T^{e_\infty} \rangle$, so the components are
indexed by $G / \langle T^{e_\infty} \rangle$.
Write them as $\wt\dd_\infty^{\,g}$ for
$g \in \{0, \dots, e_\infty-1\}$, with
$T\bigl(\wt\dd_\infty^{\,g}\bigr)
	= \wt\dd_\infty^{\,g+1}$.

Pick a lift of $\sigma$ and postcompose with a
deck transformation to obtain a lift $\wt\sigma$
that fixes a point
$x \in \wt\dd_\infty^{\,0}$.  Since
$\sigma|_{\dd_\infty} = \id$, the restriction
$\wt\sigma|_{\wt\dd_\infty^{\,0}}$ is a lift of
the identity fixing $x$ on a simply connected
domain, hence
$\wt\sigma|_{\wt\dd_\infty^{\,0}} = \id$.

By Proposition~\ref{prop:lift-commutes-T},
$\wt\sigma$ commutes with $T$.  For any $g$,
$\wt\dd_\infty^{\,g}
	= T^{\,g}(\wt\dd_\infty^{\,0})$.  Then
	\begin{align*}
		\wt\sigma|_{\wt\dd_\infty^{\,g}}
		= \wt\sigma
			\circ T^{\,g}|_{\wt\dd_\infty^{\,0}}
		= T^{\,g}
			\circ \wt\sigma|_{\wt\dd_\infty^{\,0}}
		= \id.
	\end{align*}
Thus $\wt\sigma|_{\wt\dd_\infty} = \id$.  Two
lifts differ by a deck transformation. A nontrivial
one cannot fix $\wt\dd_\infty$ pointwise, so
$\wt\sigma$ is unique.
\end{proof}

\begin{remark}\label{rem:canonical_lift}
For the remainder of this paper, any reference to the
lift $\wt\sigma$ of a mixed braid $\sigma \in \bnp$
implicitly assumes this unique canonical lift fixed
near the fiber at infinity.
\end{remark}

\section{The Cohomological Representation}\label{section:cohomology_cyclic_cover}

We equip $H^1(X)$ with the Hermitian intersection form
and use Poincar\'e duality to reduce pairing
computations to curve intersections.  The canonical
lifts of mixed braids produce unitary representations
\begin{align*}
	\rho_t \colon \bnp
	\longrightarrow
	\Ugrp\!\bigl(H^1(X)_t,
		\langle\cdot,\cdot\rangle\bigr)
\end{align*}
for $t \in \mu'(d)$.  The construction is independent
of $d$ for fixed exponents $k_i$
(Section~\ref{subsection:representation_on_cohomology}).
The Chevalley--Weil theorem determines the dimension
and Hodge signature of each $t$-eigenspace. The
explicit formulas are given in
Proposition~\ref{prop:dim_cyclic_cover} below.

\subsection{Intersection Pairings and Poincar\'e Duality}

Let $X$ be the compact Riemann surface defined above.
Write $H^1(X)$ for $H^1(X; \cc)$.  Endow $H^1(X)$
with the Hermitian intersection form
\begin{align*}
	\langle \omega, \omega' \rangle
	= \frac{\sqrt{-1}}{2}
	\int_X \omega \wedge \overline{\omega'},
\end{align*}
satisfying
$\langle \omega', \omega \rangle
= \overline{\langle \omega, \omega' \rangle}$.

Global Poincar\'e duality provides an isomorphism
\begin{equation}\label{eqn:global_poincare_isomorphism}
	\eta \colon H_1(X)
	\;\xrightarrow{\;\sim\;}\; H^1(X).
\end{equation}
For every orientation-preserving diffeomorphism
$f\colon X\to X$, Poincar\'e duality satisfies
\begin{equation}\label{eqn:poincare_naturality}
	(f^{-1})^*\eta(\gamma)=\eta(f_*\gamma).
\end{equation}
In particular,
\begin{align*}
	\rho_t(\sigma)\eta(\gamma)
	=\eta(\wt\sigma_*\gamma).
\end{align*}
Thus the action on a cohomology class may be computed by
pushing forward its Poincar\'e-dual cycle.

For an open subsurface $Y \subset X$, relative
Poincar\'e duality gives
$H_1(Y) \cong H^1_c(Y)$.  Extension by zero defines
a map
$\iota_! \colon H^1_c(Y) \longrightarrow H^1(X)$.

For curves $\gamma, \gamma'$ in $X$, the Hermitian
pairing of their Poincar\'e duals satisfies
\begin{equation}\label{eqn:poincare_dual_intersection_form}
	\langle \eta(\gamma), \eta(\gamma') \rangle
	= \frac{\sqrt{-1}}{2}\,
	\langle \gamma, \gamma' \rangle,
\end{equation}
where $\langle \gamma, \gamma' \rangle$ is the
algebraic intersection number.

The deck transformation $T$ acts on homology via
$T_*$ and on cohomology via $(T^{-1})^*$.  The
$t$-eigenspaces are
$H_1(X)_t = \ker(T_* - tI)$ and
$H^1(X)_t = \ker((T^{-1})^* - tI)$.  If
$T_* \gamma = t \gamma$, then
\begin{align*}
	(T^{-1})^*\eta(\gamma)
	= \eta(T_* \gamma)
	= t\,\eta(\gamma).
\end{align*}
Thus $\eta$ restricts to an isomorphism
\begin{align*}
	\eta \colon H_1(X)_t
	\;\xrightarrow{\;\sim\;}\; H^1(X)_t.
\end{align*}

For $\omega \in H^1(X)_t$ and $\omega' \in H^1(X)_s$,
$(T^{-1})^*(\omega \wedge \overline{\omega'})
= t \overline{s}(\omega \wedge \overline{\omega'})$.
Since $T^{-1}$ preserves the integral,
$\langle \omega, \omega' \rangle
= t \overline{s} \langle \omega, \omega' \rangle$,
which vanishes when $t \neq s$.  
Distinct eigenspaces are hence orthogonal.  
Together with
non-degeneracy of the global form, this implies the
restriction to each $H^1(X)_t$ is non-degenerate.

\subsection{The Representation on Cohomology}\label{subsection:representation_on_cohomology}

By Remark~\ref{rem:canonical_lift}, each $\sigma\in\bnp$
has a canonical lift $\wt\sigma\colon X\to X$. The
assignment $\sigma\mapsto\wt\sigma$ defines a homomorphism
from $\bnp$ to $\operatorname{Mod}(X)$. Its image commutes
with the deck transformation $T$.

The group $\operatorname{Mod}(X)$ acts on $H^1(X)$
by $f \mapsto (f^{-1})^*$. Restricting this action to the
image of $\bnp$ gives a linear representation on
$H^1(X)$. By
Proposition~\ref{prop:lift-commutes-T}, the induced
pullback preserves each eigenspace $H^1(X)_t$, yielding a
family of representations
\begin{align*}
	\rho_t\colon\bnp
	\longrightarrow
	\GL(H^1(X)_t),
	\quad
	\rho_t(\sigma) = (\wt\sigma^{-1})^*.
\end{align*}
Since $(\wt\sigma^{-1})^*$ is an isometry of
$\lab\cdot,\cdot\rab$, the image lies in the unitary
group:
\begin{align*}
	\rho_t\colon\bnp\longrightarrow
	\Ugrp
	\bigl(H^1(X)_t,\lab\cdot,\cdot\rab\bigr).
\end{align*}

\begin{remark}
$H^1(X)_1 \cong p^*H^1(\pp^1) = 0$, so
$\rho_1$ is trivial.  We restrict to
$t \in \mu'(d) := \mu(d) \setminus \{1\}$.
\end{remark}

The construction of the cover $X$ depends a priori on
the integer degree $d$ used in the construction.

Let $X_d:=X$. Suppose $d\mid d'$, and let $X_{d'}$
be the cyclic branched cover of degree $d'$ defined
by
$w^{d'} = \prod_{i=1}^n (z - b_i)^{k_{\pt(i)}}$
with the same exponents $k_{\pt(i)}$ but degree $d'$.  Then
$X_{d'}$ factors through $X_d$ via
\begin{align*}
	\pi \colon X_{d'} \to X_d,
	\qquad
	\pi(z,w) = (z, w^{\,d'/d}).
\end{align*}
Let $T'$ generate the deck group of $X_{d'}$.  Then
$\pi \circ T' = T \circ \pi$, so the pullback
$\pi^*$ carries $H^1(X_d)_t$ into $H^1(X_{d'})_t$.
Since the canonical lifts respect $\pi$,
\begin{align*}
	\pi^* \colon H^1(X_d)_t
	\;\xrightarrow{\;\sim\;}\; H^1(X_{d'})_t
\end{align*}
is a $\bnp$-equivariant isomorphism.  Hence
$\rho_t$ is independent of the degree $d$ for fixed
exponents $k_i$.

\begin{remark}
The map $\pi^*$ is not an isometry. Pulling back
and integrating over the degree-$(d'/d)$ cover
scales the intersection form:
	\begin{align*}
		\lab\pi^*(\omega),\pi^*(\omega')\rab_{X_{d'}}
		= \frac{d'}{d}\lab\omega,\omega'\rab_{X_d}.
	\end{align*}
To obtain a unitary isomorphism one must rescale
$\pi^*$ by the factor $\sqrt{d/d'}$.
\end{remark}

\subsection{Dimension of the Eigenspaces via
Chevalley--Weil}
\label{subsection:dimension_eigenspaces}

Let $p \colon X \to Y$ be a finite Galois
cover of compact Riemann surfaces with deck
group $G$.  For a branch point
$b \in B$, pick $\tilde{b} \in p^{-1}(b)$
and let $\gamma_b \in G$ generate the
stabilizer of $\tilde{b}$; its order is
the ramification index
$\nu_b := \operatorname{ord}(\gamma_b)$.  For a
character $\chi$ of $G$, define
\begin{align*}
N_b(\chi) :=
		\sum_{j=1}^{\nu_b-1}
		\Bigl(1 - \frac{j}{\nu_b}\Bigr)
		\dim\ker\!\bigl(
			\chi(\gamma_b)
			- \zeta_{\nu_b}^{\,j} I
		\bigr),
\end{align*}
where $\zeta_m := e^{2\pi i / m}$.  The
Chevalley--Weil theorem
\cite{ChevalleyWeil} (see also
\cite{Candelori} for a modern exposition)
gives the multiplicity of $\chi$ in
$H^{1,0}(X)$ as
\begin{align*}
N = \epsilon + (g_Y - 1)\dim\chi
		+ \sum_{b \in B} N_b(\chi),
\end{align*}
with $\epsilon = 1$ for the trivial
character and $\epsilon = 0$ otherwise.

Now specialize to the cyclic cover
$p \colon X \to \pp^1$ of degree $d$.
Here $G = \langle T \rangle \cong \zd$,
$g_Y = 0$, and the irreducible characters
are
$\chi_\alpha(T) = \zeta^\alpha = t$
($0 \le \alpha \le d-1$).  Take
$\alpha \neq 0$, so $\dim\chi = 1$
and $\epsilon = 0$.  At $b \in B$, the
preimage consists of
$e_b := \gcd(d, k_b)$ points, each
ramified with index
$\nu_b = d / e_b$.  The stabilizer is
generated by
$\gamma_b = T^{k_b}$ (of order
$\nu_b$). Since $\chi_\alpha$ is
$1$-dimensional,
$\chi_\alpha(\gamma_b)
= \zeta^{\alpha k_b}$ is a scalar.

Suppose first that $\alpha k_b/d \notin \zz$. There is a
unique $j \in \{1,\dots,\nu_b-1\}$ satisfying
\begin{align*}
	\zeta_{\nu_b}^{\,j}=\zeta^{\alpha k_b}.
\end{align*}
It is determined by
\begin{align*}
	\frac{j}{\nu_b}
	= \left\{\frac{\alpha k_b}{d}\right\}.
\end{align*}
Consequently,
\begin{align*}
N_b(\chi_\alpha)
	= 1-\left\{\frac{\alpha k_b}{d}\right\}
	= \left\{\frac{-\alpha k_b}{d}\right\}.
\end{align*}
If $\alpha k_b/d \in \zz$, no such $j$ occurs in
$\{1,\dots,\nu_b-1\}$. In this case
\begin{align*}
N_b(\chi_\alpha)=0
	=\left\{\frac{-\alpha k_b}{d}\right\}.
\end{align*}
Thus the multiplicity of the $T^*$-character
$\chi_\alpha$ in $H^{1,0}(X)$ is
\begin{align*}
	-1+\sum_{b\in B}
	\left\{\frac{-\alpha k_b}{d}\right\}.
\end{align*}

Our eigenspaces use $(T^{-1})^*$. Hence the
$t$-eigenspace of $(T^{-1})^*$ is the
$t^{-1}$-eigenspace of $T^*$. Replacing $\alpha$ by
$-\alpha$ gives
\begin{align*}
p_t:=\dim H^{1,0}(X)_t
	= -1+\sum_{i=1}^{n}
	\left\{\frac{\alpha k_{\pt(i)}}{d}\right\}
	+\left\{\frac{\alpha k_\infty}{d}\right\}.
\end{align*}
By Hodge symmetry,
$q_t:=\dim H^{0,1}(X)_t=p_{t^{-1}}$. Therefore
\begin{align*}
q_t
	= -1+\sum_{i=1}^{n}
	\left\{\frac{-\alpha k_{\pt(i)}}{d}\right\}
	+\left\{\frac{-\alpha k_\infty}{d}\right\}.
\end{align*}

Using $\{x\} + \{-x\} = 1$ for
$x \notin \zz$ and
$\{x\} + \{-x\} = 0$ for
$x \in \zz$, we obtain
\begin{align*}
p_t + q_t = n - 1 - r_t
	- \delta_\infty,
\end{align*}
where $\delta_\infty = 1$ if
$t^{k_\infty} = 1$ and $0$ otherwise,
and
\begin{align*}
	\mathcal{I}_t := \{\, i \in \{1,\dots,m\}
		\mid t^{k_i} = 1 \,\},
	\qquad
r_t := \sum_{i \in \mathcal{I}_t} |P_i|,
\end{align*}
where $\mathcal{I}_t$ is the set of
$t$-invisible block indices and $r_t$ is
the total number of branch points in
$t$-invisible blocks.

\begin{proposition}\label{prop:dim_cyclic_cover}
For $t \in \mu'(d)$,
	\begin{align*}
		\dim H^1(X)_t =
		\begin{cases}
n - 2 - r_t, & t^{k_\infty} = 1,\\
n - 1 - r_t, & t^{k_\infty} \neq 1.
		\end{cases}
	\end{align*}
The Hermitian form on $H^1(X)_t$ has
signature $(p_t, q_t)$.
\end{proposition}

\section{Local Cohomology of Support Disks}\label{section:generators}

To construct cohomology classes on $X$ we
first study the local geometry.  Each
generator of $\bnp$ has support in a disk
containing two branch points; we examine
the compactly supported cohomology of the
preimage of such a disk under the cyclic
cover.  Reducing to a model cover
over $\pp^1$,
Proposition~\ref{Prop:Dim_Disk_2_branch_points}
gives the dimension of the $t$-eigenspace.
When it is $1$-dimensional, a generator is
obtained by pushing forward a class from
one component of the preimage
(Remark~\ref{rem:generator_notation}).
These local classes are the input to the
global construction in
Section~\ref{section:intersection_pairings}.

\subsection{Support Disks and Their Preimages}
\label{subsection:local_geometry_supports}

Each generator of $\bnp$ is supported in a disk in $\cc$
containing exactly two branch points. These support disks
are classified by generator type:
\begin{figure}[ht]
	\centering
	\includegraphics[width=\textwidth]
		{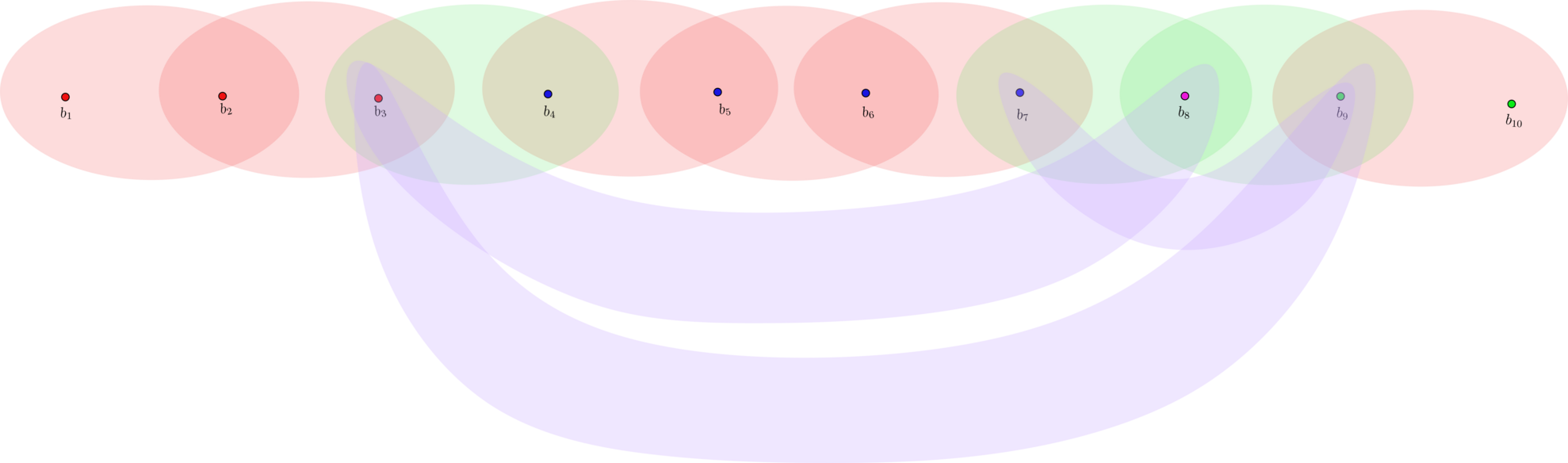}
	\caption{Support disks for $n=10$,
$\mathcal{P}=
		\{\{1,2,3\},\{4,5,6,7\},
		\{8\},\{9,10\}\}$.
Red: \artindisks.
Green: adjacent \loopdisks.
Purple: general \loopdisks.}
	\label{fig:support_disks}
\end{figure}

\begin{itemize}
	\item \textbf{\artindisks:} The support $D_i \subset \cc$
	of the Artin generator $\sigma_i$ contains only the
	branch points $b_i$ and $b_{i+1}$, which belong to the
	same partition block ($\pt(i) = \pt(i+1)$); hence
	$k_{\pt(i)} = k_{\pt(i+1)}$.
	\item \textbf{\loopdisks:} The support
	$D_{i,j}=\beta_{i,j}(D_{h_i})$ of $A_{i,j}$
	contains only $b_{h_i}$ and
	$b_{h_{j-1}+1}$, whose exponents are $k_i$
	and $k_j$. The disk $D_{i,j}$ passes below every
	intermediate branch point. For adjacent
	\loopdisks\ ($j=i+1$), we also denote the disk
	by $D_{h_i}$.						
\end{itemize}

\textbf{Indexing convention.}
For \artindisks\ or adjacent \loopdisks,
the single index ($D_i$, $\omega_i$) refers to a
branch point index; for general \loopdisks\ the
double index ($D_{i,j}$, $\omega_{i,j}$) refers to
partition indices.

Let $\dd$ be a support disk.  For an
\artindisk, both branch points lie in the same partition block
$P_i$.  For a
\loopdisk, the branch points lie in distinct blocks
$P_i$ and $P_j$.  Set
$e_i := \gcd(d, k_i)$ and
$e_{i,j} := \gcd(d, k_i, k_j)$.
These equal the number of components of
$p^{-1}(\dd)$ for an \artindisk\ and a \loopdisk\
respectively.

The monodromy image of
$\pi_1(\dd\setminus\{\text{two branch points}\})$
in $G \cong \zd$ is the subgroup generated by
$\{k_{\pt(i)}\}$ if $\dd$ is an \artindisk, and
$\{k_i, k_j\}$ if $\dd$ is a \loopdisk\ with
endpoints in $P_i$ and $P_j$.  Its
index is $e_i$ or $e_{i,j}$.  By covering space
theory, $\wt\dd := p^{-1}(\dd)$ splits into
that many components
$\wt\dd^{\,\alpha}$
($0 \le \alpha < e_i$ or $e_{i,j}$),
cyclically permuted by $T$.

By Riemann--Hurwitz applied to
$p\colon \wt D_{i,j}^{\,\alpha} \to D_{i,j}$,
the genus of each component of a \loopdisk\ is
\begin{align*}
	2g_{\wt D_{i,j}^\alpha}
	= \frac{d - e_i - e_j}{e_{i,j}}
	+ 2 -
	\frac{\gcd(k_i + k_j, d)}{e_{i,j}}.
\end{align*}
For an \artindisk\ $D_i$, the relation
$k_{\pt(i)} = k_{\pt(i+1)}$ forces both points into
$P_{\pt(i)}$, and
\begin{align*}
	2g_{\wt D_i^{\,\alpha}}
	= \frac{d - \gcd(2k_{\pt(i)}, d)}{e_{\pt(i)}}.
\end{align*}

The global Riemann--Hurwitz theorem, summing
over all branch points with
$e_\infty := \gcd(d, \sum_{i=1}^n k_{\pt(i)})$,
gives the genus of $X$:
\begin{align*}
	2g_X = (n-1)d
	- \sum_{i=1}^n e_{\pt(i)} - e_\infty + 2.
\end{align*}

\subsection{The Model Cover}\label{subsection:model_cover}

We analyze a standard model to compute the local cohomology
above a support disk. Consider the cyclic cover $Y$ over
$\pp^1$ defined by
\begin{align*}
w^{d} = (z-1)^{x}(z+1)^{y},
	\qquad \gcd(d, x, y) = 1.
\end{align*}
The deck group is $G := \zd$, generated by the transformation
$S(z,w) = (z, e^{2\pi i/d}w)$.

Let $\mdd := \mdd(0,2)$ be the open disk of radius $2$
containing the branch points $\pm 1$. We call $\mdd$ the
\emph{model disk}. Since $\gcd(d, x, y) = 1$, the monodromy
image is all of $\zd$, so the preimage
$\wt\mdd := p^{-1}(\mdd)$ is connected.

Let $\Omega := \pp^1 \setminus \mdd$ and
$\wt\Omega := p^{-1}(\Omega)$. Evaluating the monodromy
map $\Phi$ on the boundary loop $\partial\mdd$ yields
$\Phi([\partial\mdd])=x+y\pmod d$. Thus each
component of $\wt\Omega$ is stabilized by the
subgroup generated by
$x + y$.  Set
$e_x := \gcd(d, x)$, $e_y := \gcd(d, y)$, and
$e_\infty := \gcd(d, x + y)$.  The preimage
$\wt\Omega$ splits into $e_\infty$ disjoint disks
$\wt\Omega^\alpha := S^\alpha(\wt\Omega^{\,0})$ for
$0 \le \alpha < e_\infty$.

We compute $H^1_c(\wt\mdd)$ via the long exact
sequence of the pair $(Y, \wt\Omega)$.  Since
$\wt\Omega$ is a union of disjoint disks,
$H^1(\wt\Omega) = 0$, and the sequence reduces to
\begin{align*}
	0 \to H^0(Y, \wt\Omega)
	\to H^0(Y)
	\to H^0(\wt\Omega)
	\to H^1(Y, \wt\Omega)
	\to H^1(Y) \to 0.
\end{align*}
By excision, $H^1(Y, \wt\Omega) \cong
H^1_c(\wt\mdd)$.  The maps are $S$-equivariant.
Restrict to $s$-eigenspaces for $s \in \mu'(d)$.
Since $Y$ is connected, $H^0(Y)_s = 0$, yielding
\begin{align}\label{eqn:model_cover_short_exact}
	0 \to H^0(\wt\Omega)_{s}
	\to H^1_c(\wt\mdd)_{s}
	\to H^1(Y)_{s} \to 0.
\end{align}

\begin{proposition}\label{PROP:model_disk_cohomology_dimension}
For every $s \in \mu'(d)$,
	\begin{align*}
	\dim H^1_c(\wt\mdd)_{s} =
	\begin{cases}
		0 & s^{x}=1 \;\text{or}\; s^{y}=1,
			\;(\text{equivalently }
s\in\mu(e_x)\cup\mu(e_y)),
			\\[2pt]
		1 & \text{otherwise.}
	\end{cases}
	\end{align*}
\end{proposition}
\begin{proof}
The space $H^0(\wt\Omega)$ is spanned by the constant
functions on its $e_\infty$ components. These components
are cyclically permuted by $S$. Therefore:
	\begin{align*}
		\dim H^0(\wt\Omega)_{s} =
		\begin{cases}
			1 & \text{if } s \in \mu(e_{\infty}), \\
			0 & \text{otherwise.}
		\end{cases}
	\end{align*}
Apply Proposition~\ref{prop:dim_cyclic_cover} to
$Y$.  The cover has two finite branch points
and one possibly at infinity.  Since
$\gcd(d, x, y) = 1$, the integers
$e_x, e_y, e_\infty$ are pairwise coprime,
so $r_s \in \{0,1\}$ for $s \neq 1$.  Hence
	\begin{align*}
		\dim H^1(Y)_s = 1 - r_s =
		\begin{cases}
			0 & s \in \mu(e_\infty) \cup
			\mu(e_x) \cup \mu(e_y), \\
			1 & \text{otherwise.}
		\end{cases}
	\end{align*}

By dimension additivity on
	\eqref{eqn:model_cover_short_exact}, we have
$\dim H^1_c(\wt\mdd)_{s} = \dim H^0(\wt\Omega)_s
	+ \dim H^1(Y)_s$.
If $s \in \mu'(e_\infty)$, the first summand is $1$
and the second is $0$, yielding
$\dim H^1_c(\wt\mdd)_{s} = 1$. If
$s \notin \mu(e_\infty)$, the first summand vanishes,
and we find $\dim H^1_c(\wt\mdd)_{s} = \dim H^1(Y)_s$.
This dimension equals $1$ if and only if
$s \notin \mu(e_x) \cup \mu(e_y)$. Because the set
$\mu'(e_\infty)$ is contained in
$\mu'(d) \setminus (\mu(e_x) \cup \mu(e_y))$,
we have that
$\dim H^1_c(\wt\mdd)_s$ equals $1$ if
$s \notin \mu(e_x) \cup \mu(e_y)$, and $0$ otherwise.
\end{proof}

\subsection{Local Eigenspaces of the Support Disks}

Let $\dd \subset \cc$ be a support disk containing
exactly two branch points, as classified in
Subsection~\ref{subsection:local_geometry_supports}.
If $\dd$ is an \artindisk\ with branch points 
in the partition $P_i$, set
$e_i := \gcd(d, k_i)$; the preimage $\wt\dd$
consists of $e_i$ components
$\wt\dd^{\,\alpha}$
($0 \le \alpha < e_i$) cyclically permuted by
$T$.  If $\dd$ is a \loopdisk\ with branch points
in the partition $P_i$ and $P_j$, set
$e_{i,j} := \gcd(d, k_i, k_j)$; the preimage
consists of $e_{i,j}$ components.

After placing the branch points at $z = \pm 1$,
each component
$p|_{\wt\dd^{\,\alpha}}\colon
\wt\dd^{\,\alpha} \to \dd$
is isomorphic to the restriction over $\mdd$ of
the model cover
\begin{align*}
w^{d/e} = (z-1)^{\kappa_1/e}(z+1)^{\kappa_2/e},
\end{align*}
where $(\kappa_1, \kappa_2, e) = (k_i, k_i, e_i)$
for an \artindisk\ and
$(\kappa_1, \kappa_2, e) = (k_i, k_j, e_{i,j})$
for a \loopdisk.  The deck transformation $S$
of this model cover corresponds to $T^e$ on
the full cover.

\begin{remark}\label{remark:Relation_eigenvalues_s_and_t}
In this subsection, a subscript $s$
denotes the $S$-eigenspace and a
subscript $t$ the $T$-eigenspace.
Since $S = T^e$, the eigenvalues satisfy
$s = t^e$.
\end{remark}

Applying
Proposition~\ref{PROP:model_disk_cohomology_dimension}
with parameters $(d/e, \kappa_1/e, \kappa_2/e)$,
for $s \in \mu'(d/e)$ the dimension on a single
component is:
\begin{equation}
\label{eqn:dim_component}
	\dim H^1_c(\wt\dd^{\,\alpha})_s =
	\begin{cases}
		0 & s \in \mu(e_1/e)
			\cup \mu(e_2/e),\\[2pt]
		1 & \text{otherwise},
	\end{cases}
\end{equation}
where $e_1 = \gcd(d, \kappa_1)$,
$e_2 = \gcd(d, \kappa_2)$.
In either case (Artin and loop) 
the dimension is $0$ on the
trivial eigenspace $s = 1$ and $1$ on all
other eigenspaces whose eigenvalues avoid
$\mu(e_1/e) \cup \mu(e_2/e)$.

\begin{proposition}\label{Prop:Dim_Disk_2_branch_points}
For $t \in \mu'(d)$,
	\begin{align*}
		\dim H^1_c(\wt\dd)_t =
		\begin{cases}
			0 & t^{k_i}=1,
			\\[4pt]
			1 & \text{otherwise},
		\end{cases}
	\end{align*}
if $\dd$ is an \artindisk\ with both branch
points in $P_i$; and
	\begin{align*}
		\dim H^1_c(\wt\dd)_t =
		\begin{cases}
			0 & t^{k_i}=1 \;\text{or}\;
t^{k_j}=1,
			\\[4pt]
			1 & \text{otherwise},
		\end{cases}
	\end{align*}
if $\dd$ is a \loopdisk\ with endpoints in
$P_i$ and $P_j$.
\end{proposition}

\begin{proof}
We prove the \loopdisk\ case; the \artindisk\
case is analogous.  An eigenvector in
$H^1_c(\wt\dd)_t$ restricts to an
$s$-eigenvector of $S = T^{e_{i,j}}$ on
each component, with $s = t^{e_{i,j}}$.
The condition $t^{k_i}=1$ or $t^{k_j}=1$
is equivalent to
$s \in \mu(e_i/e_{i,j})
	\cup \mu(e_j/e_{i,j})$.
If $t^{k_i}=1$ or $t^{k_j}=1$,
	\eqref{eqn:dim_component} forces the
$s$-eigenspace on every component to
vanish, so $H^1_c(\wt\dd)_t = 0$.

If $t^{k_i} \neq 1$ and $t^{k_j} \neq 1$,
	\eqref{eqn:dim_component} gives a
$1$-dimensional $s$-eigenspace on a single
component $\wt\dd^{\,0}$.  Pick a non-zero
$\omega^{\,0} \in H^1_c(\wt\dd^{\,0})_s$
and define
$\omega^{\,\alpha} = (T^{-1})^*\omega^{\,\alpha-1}$
iteratively.  The sum
	\begin{equation}\label{eqn:generator_atoms}
		\omega :=
		\sum_{\alpha=0}^{e_{i,j}-1}
		\ov t^{\,\alpha}\,\omega^{\,\alpha}
	\end{equation}
forms a non-zero element of
$H^1_c(\wt\dd)_t$.  By Riemann--Hurwitz
applied to the branched cover
$\wt\dd \to \dd$, the total dimension is
$\dim H^1_c(\wt\dd)
	= d - e_i - e_j + e_{i,j}$.
This integer equals the number of
$t \in \mu'(d)$ with $t^{k_i} \neq 1$
and $t^{k_j} \neq 1$.  Since we have
constructed a non-zero element for each
such $t$, and the direct sum of these
eigenspaces cannot exceed the total
dimension, each non-trivial eigenspace
is exactly one-dimensional.
\end{proof}

\begin{remark}\label{rem:generator_notation}
We establish notation for a generator of
$H^1_c(\wt\dd)_t$ (non-zero precisely
when the stated exponents satisfy
$t^{k} \neq 1$):
	\begin{description}
		\item[$\omega_i$] for an \artindisk\
support $D_i$
			($t^{k_{\pt(i)}} \neq 1$),
		\item[$\omega_{h_i}$]
			($=\omega_{i,i+1}$) for an
adjacent \loopdisk\ support
			($t^{k_i}, t^{k_{i+1}} \neq 1$),
		\item[$\omega_{i,j}$] for a general
			\loopdisk\ support $D_{i,j}$
			($t^{k_i}, t^{k_j} \neq 1$),
		\item[$\omega_\sigma$] for a generic
support disk associated with
$\sigma \in \bnp$.
	\end{description}
Explicit geometric choices are
constructed in
Section~\ref{section:intersection_pairings}
	(see also Proposition~\ref{thm:spanning_via_intersection}).
\end{remark}

\begin{remark}[$t$-invisible blocks]
\label{rem:t_invisible_blocks}
Proposition~\ref{Prop:Dim_Disk_2_branch_points}
shows: if $t^{k_i}=1$, then
$\dim H^1_c(\wt\dd)_t = 0$ for every
support disk with a branch point in
$P_i$. We call such a
partition block \emph{$t$-invisible}.  These blocks
contribute nothing to $\rho_t$.
Consequently $\rho_t$ factors through
$\forgmap$
	(Subsection~\ref{subsection:mixed_braid_group_definition}),
proved in
Section~\ref{section:action_mbg}.
\end{remark}

\begin{lemma}\label{lem:global_injectivity}
Let $t \in \mu'(d)$.  If
$H^1(X)_t \neq 0$, then the
extension-by-zero map
$\iota_!\colon H^1_c(\wt\dd)_t
	\to H^1(X)_t$ is injective.
\end{lemma}
\begin{proof}
Let $\Omega := \pp^1 \setminus \dd$ and
$\wt\Omega := p^{-1}(\Omega)$.
The excision sequence for
$(X, \wt\Omega)$ reduces to:
	\begin{align*}
		0 \longrightarrow H^0(\wt\Omega)_t
		\longrightarrow H^1_c(\wt\dd)_t
		\xrightarrow{\;\iota_!\;}
H^1(X)_t,
	\end{align*}
since $H^0(X)_t = 0$ for $t \neq 1$.
Thus $\iota_!$ is injective exactly
when $H^0(\wt\Omega)_t = 0$.

Suppose $H^0(\wt\Omega)_t \neq 0$.  Take
any finite branch point
$b_\ell \in \Omega$.  The deck
transformation $T^{-k_{\pt(\ell)}}$ fixes
the preimages of $b_\ell$, hence acts
as the identity on $H^0(\wt\Omega)$.
Thus any $\omega \in H^0(\wt\Omega)_t$
satisfies
$\omega = (T^{-k_{\pt(\ell)}})^*\omega
	= t^{k_{\pt(\ell)}}\omega$, forcing
$t^{k_{\pt(\ell)}} = 1$.  The same
argument with $k_\infty$ applies to
$\infty$.  Therefore $r_t \ge n-1$
and $t^{k_\infty}=1$.
Proposition~\ref{prop:dim_cyclic_cover}
gives $\dim H^1(X)_t = 0$.
Hence $H^1(X)_t \neq 0$ implies
$\iota_!$ is injective.
\end{proof}

\section{Intersection Pairings and Generating Classes}%
\label{section:intersection_pairings}

We extend the local classes of
Section~\ref{section:generators} to global
classes on $X$ and compute their Hermitian
intersection pairings.  For each element of
$H^1(X)_t$, we select a base contour in
the corresponding support disk and form a
$\ov t$-polynomial combination of its
lifts.  This yields a cycle in $H_1(X)_t$
representing its Poincar\'e dual.  Signed
geometric intersections of these cycles determine the
intersection matrix. Its rank computation in
Proposition~\ref{thm:spanning_via_intersection} proves that
the classes span $H^1(X)_t$.

\subsection{Construction of Poincar\'e Duals}
\label{subsection:construction_poincare_duals}

Let $\ell_k$ denote a simple counterclockwise loop around
$b_k$.

\begin{figure}[ht]
	\centering
	\begin{minipage}{0.7\textwidth}
		\centering
		\begin{minipage}{0.48\linewidth}
			\centering
			\includegraphics[width=\linewidth]
				{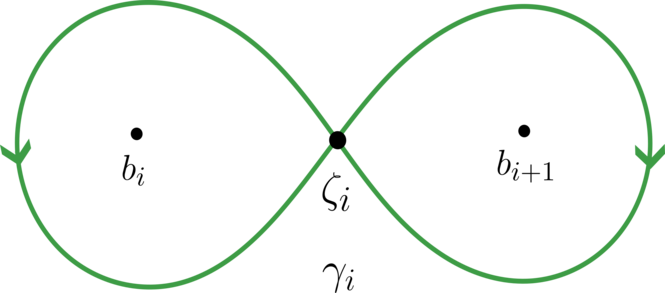}
			\caption*{Artin generator}
		\end{minipage}
		\hfill
		\begin{minipage}{0.48\linewidth}
			\centering
			\includegraphics[width=\linewidth]
				{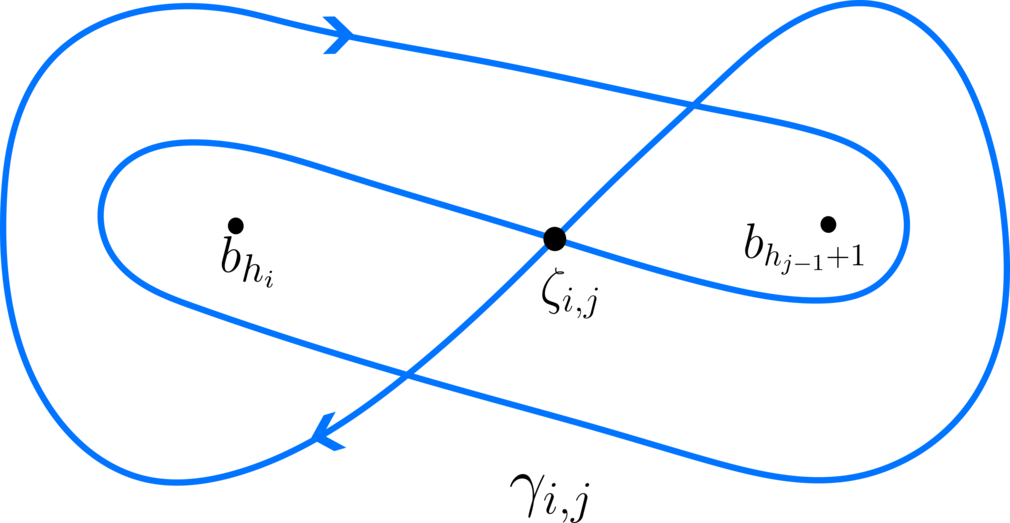}
			\caption*{Loop generator}
		\end{minipage}
		\caption{Contours for the Artin
and loop generators.}
		\label{fig:generators}
	\end{minipage}
\end{figure}

For an \artindisk\ $D_i$, define the figure-eight contour
$\gamma_i:=\ell_i\ov{\ell_{i+1}}$ based at $\zeta_i$.
Since $k_{\pt(i)}=k_{\pt(i+1)}$, its monodromy is
trivial. For a \loopdisk\ $D_{i,j}$, define the
Pochhammer contour $\gamma_{i,j} := 
[\ell_{h_i},\ell_{h_{j-1}+1}]$
based at $\zeta_{i,j}$. Here $[a,b]:=aba^{-1}b^{-1}$,
which fixes the orientation of the Pochhammer contour.
Its monodromy is trivial because it is a commutator.
When $j=i+1$, abbreviate
$D_{i,i+1}$, $\gamma_{i,i+1}$, and $\zeta_{i,i+1}$ by
$D_{h_i}$, $\gamma_{h_i}$, and $\zeta_{h_i}$. These
contours have zero total winding number.

Since each base contour $\gamma$ lies in $\ker\Phi$, its
$d$ lifts are closed immersed curves in $X$.
Fix a base point $\wt\zeta^{\,0}$ above $\zeta$
and write $\wt\zeta^{\,g} :=
T^g(\wt\zeta^{\,0})$. The family of lifts
$\wt\gamma^{\,g}$ ($g\in G$) yields the
following $t$-eigenspace cycles:
\begin{align}\label{eqn:artin_atom}
	\wt\gamma_i :=
	\sqrt{\frac{2}{d}} \frac{t^{k_{\pt(i)}}}{1-t^{k_{\pt(i)}}}
	\sum_{g\in G}
	\ov{t}^g\,\wt\gamma_i^{\,g},
\end{align}
\begin{align}\label{eqn:mixed_atom}
	\wt\gamma_{i,j} :=
	\sqrt{\frac{2}{d}} \frac{t^{k_i+k_j}}{(1-t^{k_i})(1-t^{k_j})}
	\sum_{g\in G}
	\ov{t}^g\,
	\wt\gamma_{i,j}^{\,g}.
\end{align}
By Proposition~\ref{Prop:Dim_Disk_2_branch_points}, the
lifts $\wt\gamma^{\,g}$ correspond precisely to the
local spanning elements $\omega^{\,g}$ established
in its proof.  $T_*$ sends $\wt\gamma^{\,g}$ to
$\wt\gamma^{\,g+1}$, so $T_*\wt\gamma
= t\,\wt\gamma$.
These cycles are non-zero exactly when the local
eigenspaces are non-trivial. The Poincar\'e duality
isomorphism $\eta$ from
\eqref{eqn:global_poincare_isomorphism} preserves
eigenspaces, thus restricting to an isomorphism
$\eta\colon H_1(X)_t \xrightarrow{\;\sim\;} H^1(X)_t$.
Using this restriction, we define the explicit cohomology
generators of Remark~\ref{rem:generator_notation} as the
duals of these specific $t$-eigenspace cycles:
\begin{align*}
	\omega_i &:= \eta(\wt\gamma_i), \\
	\omega_{i,j} &:= \eta(\wt\gamma_{i,j}).
\end{align*}
We work with these explicit classes $\omega$ throughout
the remainder of the paper to compute intersection numbers,
study the mixed braid group action, and establish the
identification with the Burau representation.

\begin{remark}[Twisted cycle normalizations]
\label{rem:normalization}
The Artin normalization satisfies
	\begin{align*}
		\frac{t^k}{1-t^k}
		= -\frac{1}{1-t^{-k}}.
	\end{align*}
The minus sign records the chosen contour
orientation. For a loop class,
	\begin{align*}
		\frac{t^{k_i+k_j}}
		{(1-t^{k_i})(1-t^{k_j})}
		=
		\frac{1}
		{(1-t^{-k_i})(1-t^{-k_j})}.
	\end{align*}
These expressions match the regularization factors
for twisted cycles in
	\cite[Section~3]{KitaYoshida1994}. Counting
intersections on the finite cover reproduces the
homological pairings obtained there from twisted
chains.
\end{remark}

The cycle $\wt\gamma$ in
\eqref{eqn:artin_atom}--\eqref{eqn:mixed_atom}
depends on the choice of initial lift
$\wt\gamma^{\,0}$ among the $d$ lifts of
$\gamma$.  A different choice replaces
$\wt\gamma^{\,0}$ by $T^a(\wt\gamma^{\,0})$
for some $a$, which multiplies $\wt\gamma$
by $t^{\,a}$.  The intersection pairing
$\langle\omega,\omega'\rangle$ therefore
depends on the \emph{relative} lift choices
for $\omega$ and $\omega'$.  We fix these
choices once and for all as follows.

\begin{figure}[ht]
	\centering
	\includegraphics[width=\textwidth]
			{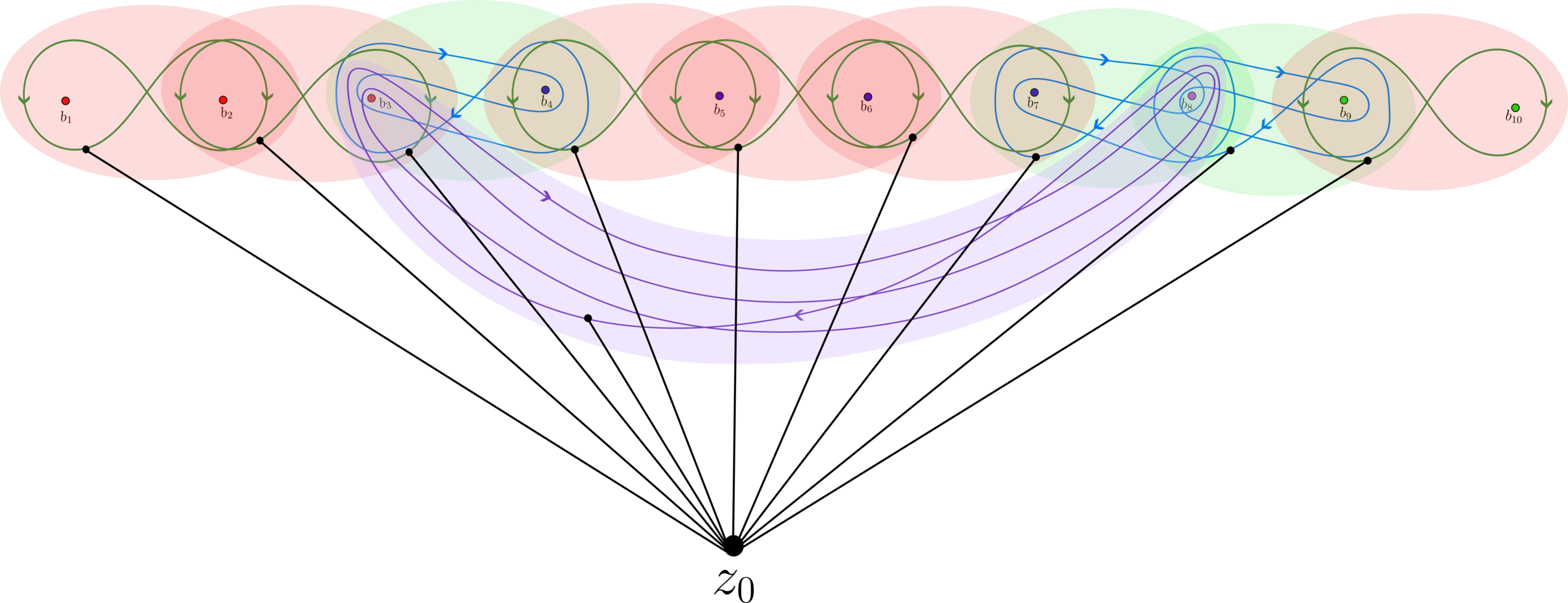}
	\caption{Base point convention
	for the example above.
	The black paths illustrate the
	convention described in
	Remark~\ref{rem:base_point_convention}.
	The label $\wt z_0$ records the
	chosen initial lift; the black
	paths are its projections from
	$z_0$.}
	\label{fig:base_point_convention}
\end{figure}

\begin{remark}[Base point and access-path conventions]
\label{rem:base_point_convention}
Choose $z_0\in\dd_\infty$ with $\operatorname{Im}(z_0)<0$ 
sufficiently far below every support disk. The
canonical lift is the identity on
$p^{-1}(\dd_\infty)$ and therefore fixes every point
above $z_0$.

For each $b_k$, use the upward branch cut
$b_k+\sqrt{-1}[0,\infty)$. The complement of these cuts is a simply-connected
slit plane containing $z_0$. For each $k$, let $r_k=[z_0,b_k]$ be the
straight line segment oriented from $z_0$ toward $b_k$. Its interior lies
in the slit plane. For a figure-eight contour $\gamma_i$, choose its base
point $\zeta_i$ to be the positive transverse intersection of
$r_i$ with the $\ell_i$ portion of $\gamma_i$. For a Pochhammer contour
$\gamma_{i,j}$, choose its base point $\zeta_{i,j}$ to be the 
positive transverse intersection of $r_{h_i}$ with the
$\ov\ell_{h_i}\,\ov\ell_{h_{j-1}+1}$ portion of $\gamma_{i,j}$. Use the
subsegments $[z_0,\zeta_i]\subset r_i$ and
$[z_0,\zeta_{i,j}]\subset r_{h_i}$, respectively, as the access paths.
	See Figure~\ref{fig:base_point_convention}.

Fix $\wt z_0^{\,0}\in p^{-1}(z_0)$. Lifting each
access path from $\wt z_0^{\,0}$ determines an
endpoint $\wt\zeta^{\,0}$ and hence an initial lift
$\wt\gamma^{\,0}$. These choices are used in every
sheet-label and intersection computation below.
\end{remark}

\begin{lemma}[Reduction formula]%
\label{lem:reduction_formula}
Let $\omega, \omega'$ be Poincaré duals of
$t$-eigenspace cycles
$\wt\gamma = C_{\gamma} \sqrt{2/d}
	\sum_{g\in G} \ov{t}^{\,g} \wt\gamma^{\,g}$
and $\wt\gamma' = C_{\gamma'} \sqrt{2/d}
	\sum_{g'\in G} \ov{t}^{\,g'} (\wt\gamma')^{\,g'}$.
Their intersection pairing is
	\begin{equation}
	\label{eqn:poincare_dual_intersection_form_expansion}
		\langle\omega,\omega'\rangle
		= C_{\gamma} \ov{C_{\gamma'}} \sqrt{-1}
		\sum_{g\in G}
t^{g}
		\langle\wt\gamma^{\,0},(\wt\gamma')^{\,g}\rangle.
	\end{equation}
\end{lemma}
\begin{proof}
	\eqref{eqn:poincare_dual_intersection_form}
gives $\langle\omega,\omega'\rangle =
	\tfrac{\sqrt{-1}}{2}
	\langle\wt\gamma,\wt\gamma'\rangle$.
Expanding the cycle definitions gives
	\begin{align*}
		\langle\wt\gamma,\wt\gamma'\rangle
		= C_{\gamma} \ov{C_{\gamma'}} \frac{2}{d}
	\sum_{g,g'\in G}
	\ov{t}^{\,g}t^{g'}
		\langle\wt\gamma^{g},(\wt\gamma')^{g'}\rangle.
	\end{align*}
By $G$-equivariance,
$\langle\wt\gamma^{g},(\wt\gamma')^{g'}\rangle
	= \langle\wt\gamma^0,(\wt\gamma')^{g'-g}\rangle$.
Substitute $g' \mapsto g'+g$.  The summand
is independent of $g$, so summing over $g$
cancels the factor $d$:
	\begin{align*}
		\langle\wt\gamma,\wt\gamma'\rangle
		= C_{\gamma} \ov{C_{\gamma'}}\, 2
		\sum_{g\in G}
t^{g}
		\langle\wt\gamma^0,(\wt\gamma')^{g}\rangle.
	\end{align*}
Multiplying by $\tfrac{\sqrt{-1}}{2}$ yields
	\eqref{eqn:poincare_dual_intersection_form_expansion}.
\end{proof}

\begin{remark}[Intersection lifting convention]
\label{rem:intersection_lifting}
For an intersection pairing
$\langle\omega,\omega'\rangle$, index the lifts
$\wt a_l^{\,g}$ of each intersection point $a_l$
on the corresponding lift $\wt\gamma^{\,g}$ of
the first contour $\gamma$.
\end{remark}
\subsection{Intersection Pairings of Adjacent Classes}
\label{subsection:spanning_intersections}

We evaluate the intersection pairings of the
adjacent classes $\{\omega_1, \dots, \omega_{n-1}\}$.
Applying the reduction formula, we compute these
pairings case by case by enumerating the geometric
intersections of their underlying lifts.

\begin{proposition}[Adjacent class intersections]%
\label{prop:spanning_intersections}
For $1\le i,j\le n-1$, use the convention in
Remark~\ref{rem:intersection_lifting}. Then
	\begin{align}\label{eqn:spanning_intersections}
		\langle\omega_i,\omega_j\rangle =
		\begin{cases}
			\sqrt{-1} \dfrac{1+t^{k_r}}{1-t^{k_r}}
			& i=j,\; D_i \text{ an \artindisk\ in partition } P_r,
			\\[8pt]
			\dfrac{\sqrt{-1}}{2} \left( \dfrac{1+t^{k_r}}{1-t^{k_r}} + \dfrac{1+t^{k_{r+1}}}{1-t^{k_{r+1}}} \right)
			& i=j,\; D_i \text{ a \loopdisk\ from } P_r \text{ to } P_{r+1},
			\\[8pt]
			\dfrac{-\sqrt{-1}}{1 - t^{k_r}}
			& j=i+1, \text{ with shared branch point } b_{i+1} \text{ having exponent } k_r,
			\\[8pt]
			0 & j > i+1.
		\end{cases}
	\end{align}
For $i > j$: $\langle\omega_i,\omega_j\rangle
	= \ov{\langle\omega_j,\omega_i\rangle}$.
\end{proposition}

\begin{proof}
By Lemma~\ref{lem:reduction_formula}, the pairing is
$\langle\omega,\omega'\rangle =C_{\gamma} \ov{C_{\gamma'}}
	\sqrt{-1}
	\sum_{g\in G} t^g
	\langle\wt\gamma^{\,0},(\wt\gamma')^{\,g}\rangle$.
Let $\zeta$ and $\zeta'$ be the base points of $\gamma$
and $\gamma'$. To evaluate the geometric intersections
at a point $a \in \gamma \cap \gamma'$, construct a
closed loop based at $z_0$. Traverse the base path
from $z_0$ to $\zeta$, follow $\gamma$ to $a$, trace
$\gamma'$ from $a$ to $\zeta'$, and return to $z_0$
along the reversed base path of $\gamma'$. The total
monodromy of this loop identifies the exact lift
$(\wt\gamma')^{\,g}$ that intersects the base lift
$\wt\gamma^{\,0}$ at the lift of $a$. We compute these
intersections case by case.

\subsection*{Case 1: Self-intersections ($i=j$)}
The self-intersection of an adjacent class depends
on whether its support is an \artindisk\ (Subcase 1a)
or a \loopdisk\ (Subcase 1b).

\subsubsection*{Subcase 1a: Self-intersection of classes from \artindisks}
\leavevmode\par
\noindent
\begin{minipage}[c]{0.68\textwidth}
Suppose $D_i$ is an \artindisk\ in $P_r$.
The curve $\gamma_i$ has one self-intersection
at $a$.
Its base lift $\wt\gamma_i^{\,0}$ traverses two
preimages of $a$, intersecting $\wt\gamma_i^{\,k_r}$
at $\wt a^{\,k_r}$ (sign $+1$) and $\wt\gamma_i^{\,-k_r}$
at $\wt a^{\,0}$ (sign $-1$).
The weighted sum of these intersections yields the
raw polynomial $(t^{k_r} - t^{-k_r})$. 
	\end{minipage}\hfill
\begin{minipage}[c]{0.28\textwidth}
	\centering
	\includegraphics[width=\linewidth]{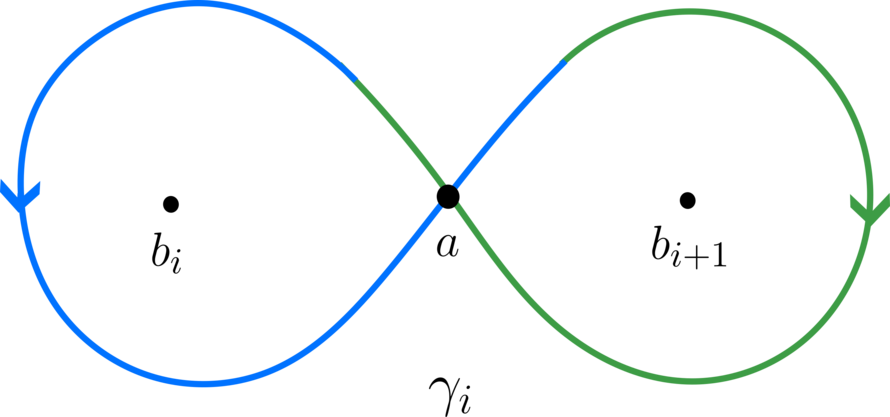}
\end{minipage}
This expression
factorizes as $(1-t^{-k_r})(1+t^{k_r})$. By the
reduction formula
\eqref{eqn:poincare_dual_intersection_form_expansion},
the cohomology pairing evaluates to:
\begin{align*}
	\langle\omega_i,\omega_i\rangle
	&= \sqrt{-1} \frac{(1-t^{-k_r})(1+t^{k_r})}{(1-t^{-k_r})(1-t^{k_r})}
	= \sqrt{-1} \frac{1 + t^{k_r}}{1 - t^{k_r}}.
\end{align*}

\vspace{1em}

\subsubsection*{Subcase 1b: Self-intersection of classes from \loopdisks}
Suppose $D_i$ is a \loopdisk\ connecting partition $P_r$ to $P_{r+1}$. 
The curve $\gamma_i$ self-intersects at $a_1$, $a_2$,
and $a_3$. 
\leavevmode\par
\noindent
\begin{minipage}[c]{0.68\textwidth}
	Tracking accumulated monodromy yields the
	intersections of the base lift $\wt\gamma_i^{\,0}$. At
	$a_1$, it meets $\wt\gamma_i^{\,-k_r-k_{r+1}}$ at $\wt a_1^{\,0}$ ($+1$)
	and $\wt\gamma_i^{\,k_r+k_{r+1}}$ at $\wt a_1^{\,k_r+k_{r+1}}$ ($-1$).
	At $a_2$,  it meets
	$\wt\gamma_i^{\,-k_{r+1}}$ at $\wt a_2^{\,0}$ ($-1$)
	and $\wt\gamma_i^{\,k_{r+1}}$ at
	$\wt a_2^{\,k_{r+1}}$ ($+1$).
	At $a_3$, it meets $\wt\gamma_i^{\,k_r}$ at $\wt a_3^{\,0}$
	($+1$) and $\wt\gamma_i^{\,-k_r}$ at $\wt a_3^{\,-k_r}$ ($-1$).
	\end{minipage}\hfill
\begin{minipage}[c]{0.28\textwidth}
	\centering
	\includegraphics[width=\linewidth]{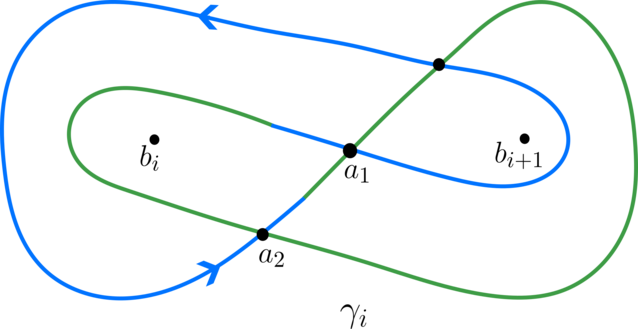}
\end{minipage}
The weighted sum of these intersections yields the raw
polynomial $(t^{-k_r-k_{r+1}} - t^{k_r+k_{r+1}}- t^{-k_{r+1}} + t^{k_{r+1}} 
 + t^{k_r} - t^{-k_r})$. This expression
factorizes as $-(1 - t^{k_r})(1 - t^{k_{r+1}})(1 - t^{-k_r-k_{r+1}})$.
By the reduction formula \eqref{eqn:poincare_dual_intersection_form_expansion},
the cohomology pairing evaluates to:
\begin{align*}
	\langle\omega_i,\omega_i\rangle
	&= \sqrt{-1} \frac{-(1 - t^{k_r})(1 - t^{k_{r+1}})(1 - t^{-k_r-k_{r+1}})}{(1-t^{k_r})(1-t^{-k_r})(1-t^{k_{r+1}})(1-t^{-k_{r+1}})}
	= \frac{\sqrt{-1}}{2} \left( \frac{1+t^{k_r}}{1-t^{k_r}} + \frac{1+t^{k_{r+1}}}{1-t^{k_{r+1}}} \right).
\end{align*}

	Observe that both cases yield strictly real values
	for the Hermitian self-intersections.

\subsection*{Case 2: Adjacent disks ($j=i+1$)}

Adjacent disks $D_i$ and $D_{i+1}$ can be either
\artindisks\ or \loopdisks. Four configurations
arise. When both are \artindisks\ (Subcase 2a),
when an \artindisk\ precedes a \loopdisk\ (Subcase 2b),
when a \loopdisk\ precedes an \artindisk\ (Subcase 2c),
and when both are \loopdisks\ (Subcase 2d).

\subsubsection*{Subcase 2a: Adjacent \artindisks\ ($D_i, D_{i+1}$ are \artindisks)}
Because $D_i$ and $D_{i+1}$ are adjacent \artindisks\ within
the same partition block $P_r$, their relevant local monodromy
exponent is $k_r$.

The curves $\gamma_i$ and $\gamma_{i+1}$ intersect twice.
Let $a_1$ be the first intersection encountered along
$\gamma_i$ from $\zeta_i$, and $a_2$ the second.
Due to our base point convention
(Remark~\ref{rem:base_point_convention}) and studying the path
$z_0 \to \zeta_i \to a_l \to \zeta_{i+1} \to z_0$, we evaluate
each intersection. 
\leavevmode\par
\noindent
\begin{minipage}[c]{0.56\textwidth}
At the first intersection $a_1$, the acquired
monodromy is $0$. Thus, $\wt\gamma_i^{\,0}$
intersects $\wt\gamma_{i+1}^{\,0}$ at
$\wt a_1^{\,0}$ with sign $-1$. At the second
intersection $a_2$, the acquired monodromy is $-k_r$.
Thus, $\wt\gamma_i^{\,0}$ intersects
$\wt\gamma_{i+1}^{\,-k_r}$ at $\wt a_2^{\,0}$ with sign
$+1$.
	\end{minipage}\hfill
\begin{minipage}[c]{0.40\textwidth}
	\centering
	\includegraphics[width=\linewidth]{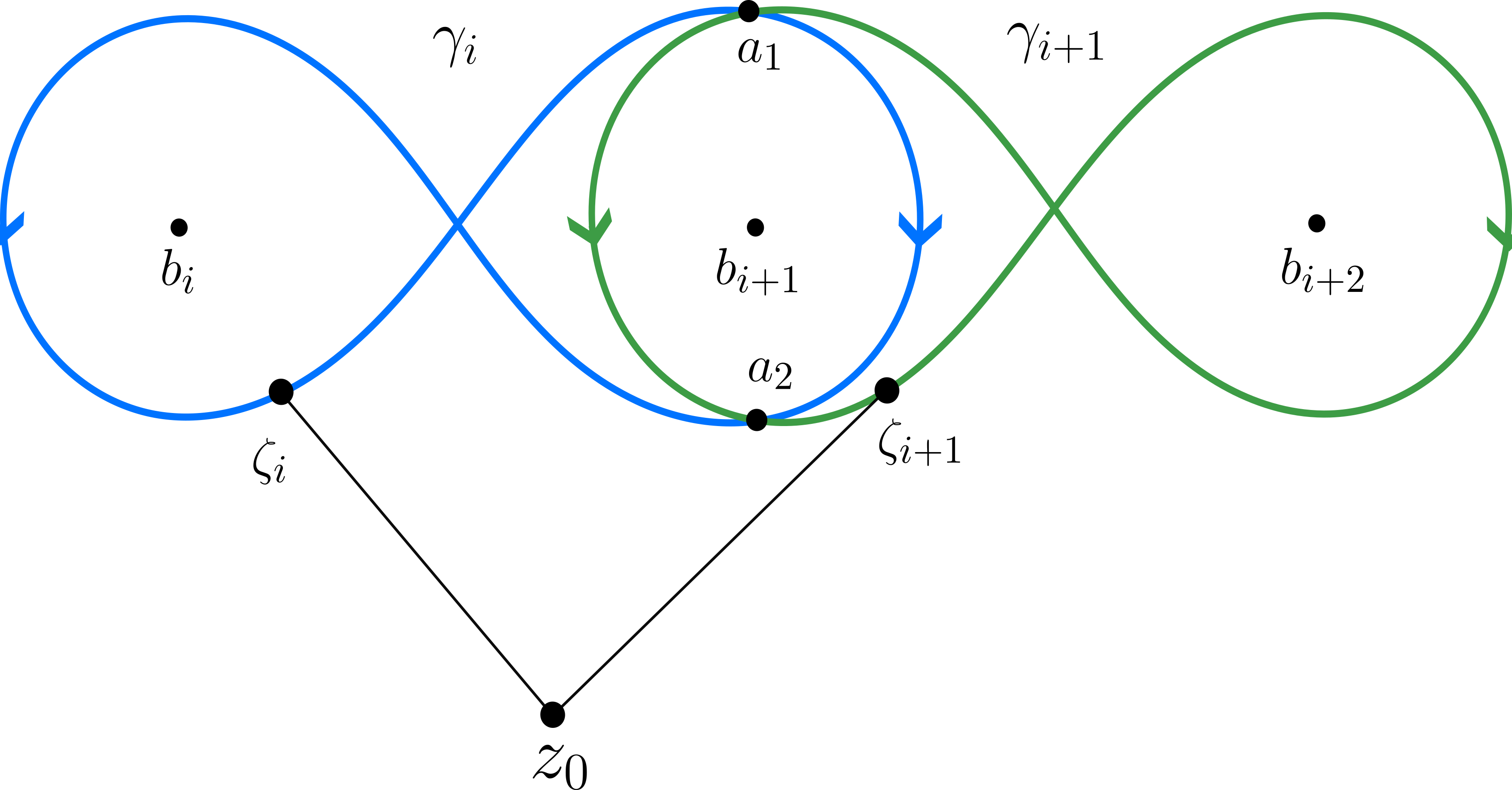}
\end{minipage}

The weighted sum of these intersections yields the
raw polynomial $-(1-t^{-k_r})$. By the reduction
formula \eqref{eqn:poincare_dual_intersection_form_expansion},
the cohomology pairing evaluates to:
\begin{align*}
	\langle \omega_i, \omega_{i+1} \rangle
	&= \sqrt{-1} \frac{-t^{k_r}t^{-k_r}(1-t^{-k_r})}{(1-t^{k_r})(1-t^{-k_r})}
	= \frac{-\sqrt{-1}}{1 - t^{k_r}}.
\end{align*}

\vspace{1em}

\subsubsection*{Subcase 2b: Adjacent Artin and loop disk}
The \artindisk\ $D_i$ belongs to the partition block $P_r$,
while the \loopdisk\ $D_{i+1}$ connects $P_r$ and $P_{r+1}$.
Consequently, the relevant local monodromy exponents are
$k_r$ (for the shared branch point $b_{i+1}$) and $k_{r+1}$
(for the subsequent branch point $b_{i+2}$).
\leavevmode\par
\noindent
\begin{minipage}[c]{0.56\textwidth}
The curves $\gamma_i$ and $\gamma_{i+1}$ intersect
four times near the shared branch point $b_{i+1}$.
Accounting for monodromies, $\wt\gamma_i^{\,0}$
intersects $\wt\gamma_{i+1}^{\,0}$ at $\wt a_1^{\,0}$
	(sign $+1$), $\wt\gamma_{i+1}^{\,k_{r+1}}$ at
$\wt a_2^{\,0}$ (sign $-1$), $\wt\gamma_{i+1}^{\,k_{r+1}-k_r}$
at $\wt a_3^{\,0}$ (sign $+1$), and $\wt\gamma_{i+1}^{\,-k_r}$
at $\wt a_4^{\,0}$ (sign $-1$).
	\end{minipage}\hfill
\begin{minipage}[c]{0.40\textwidth}
	\centering
	\includegraphics[width=\linewidth]{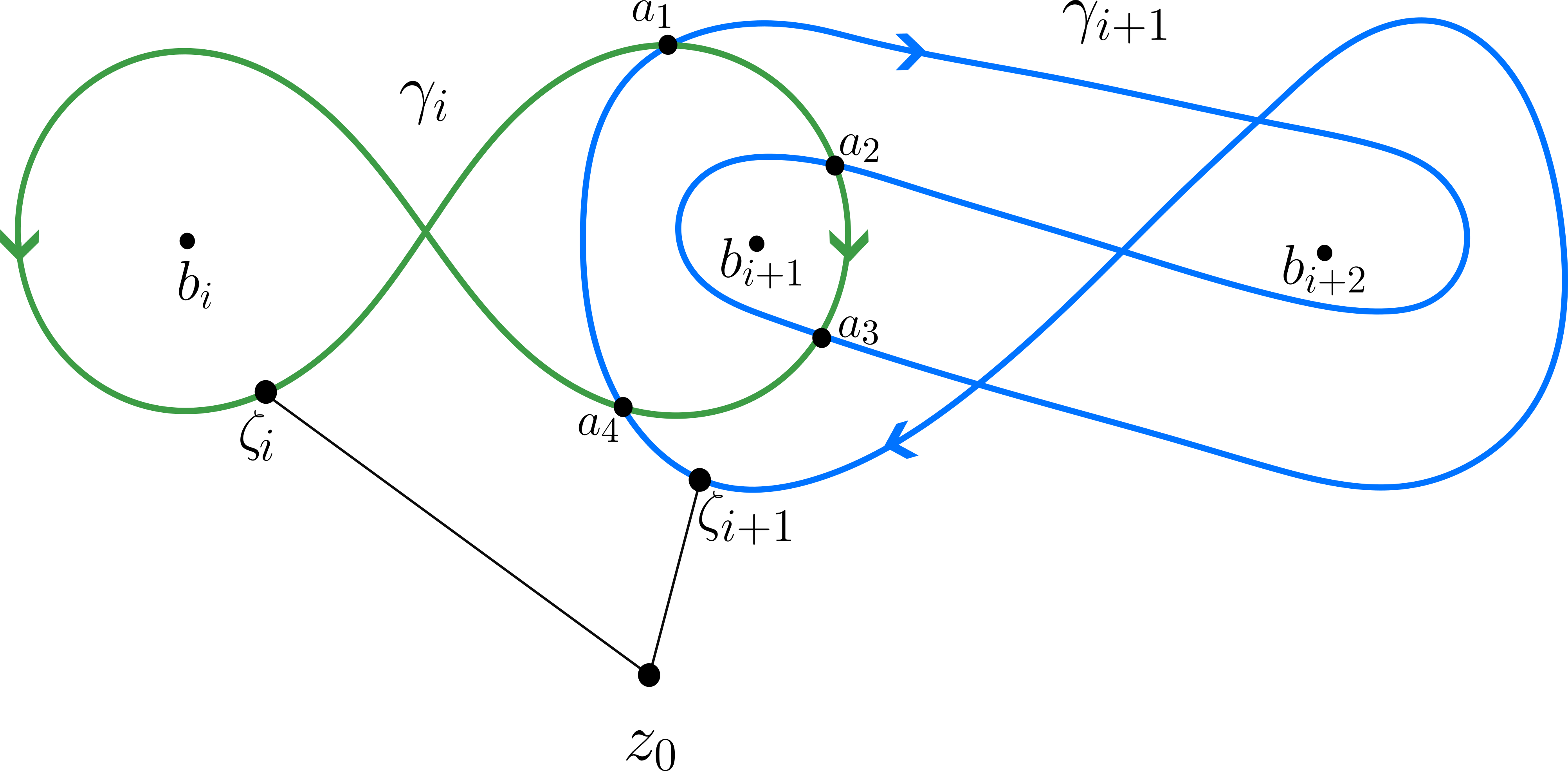}
\end{minipage}
The weighted sum of these intersections yields the raw
polynomial $(1 - t^{k_{r+1}} + t^{k_{r+1}-k_r} - t^{-k_r})$.
This expression factorizes as $(1-t^{-k_r})(1 - t^{k_{r+1}})$.
By the reduction formula
\eqref{eqn:poincare_dual_intersection_form_expansion},
the cohomology pairing evaluates to:
\begin{align*}
	\langle \omega_i, \omega_{i+1} \rangle
	&= \sqrt{-1} \frac{t^{k_r}t^{-k_r-k_{r+1}}(1-t^{-k_r})(1 - t^{k_{r+1}})}{(1-t^{k_r})(1-t^{-k_r})(1-t^{-k_{r+1}})}
	= \frac{-\sqrt{-1}}{1 - t^{k_r}}.
\end{align*}

\subsubsection*{Subcase 2c: Adjacent loop and Artin disk}
Suppose the \loopdisk\ $D_i$ connects $P_{r-1}$
to $P_r$, while the \artindisk\ $D_{i+1}$
belongs to $P_r$. The local monodromy exponents
are $k_{r-1}$ at $b_i$ and $k_r$ at the shared
point $b_{i+1}$.
\leavevmode\par
\noindent
\begin{minipage}[c]{0.56\textwidth}
The curves $\gamma_i$ and $\gamma_{i+1}$ intersect
four times near the shared branch point $b_{i+1}$.
Accounting for monodromies, $\wt\gamma_i^{\,0}$ 
intersects $\wt\gamma_{i+1}^{\,-k_{r-1}}$
at $\wt a_1^{\,0}$ (sign $-1$),
$\wt\gamma_{i+1}^{\,-k_{r-1}-k_r}$ at
$\wt a_2^{\,0}$ (sign $+1$), $\wt\gamma_{i+1}^{\,-k_{r}}$
at $\wt a_3^{\,0}$ (sign $-1$), and $\wt\gamma_{i+1}^{\,0}$
at $\wt a_4^{\,0}$ (sign $+1$).
	\end{minipage}\hfill
\begin{minipage}[c]{0.40\textwidth}
	\centering
	\includegraphics[width=\linewidth]{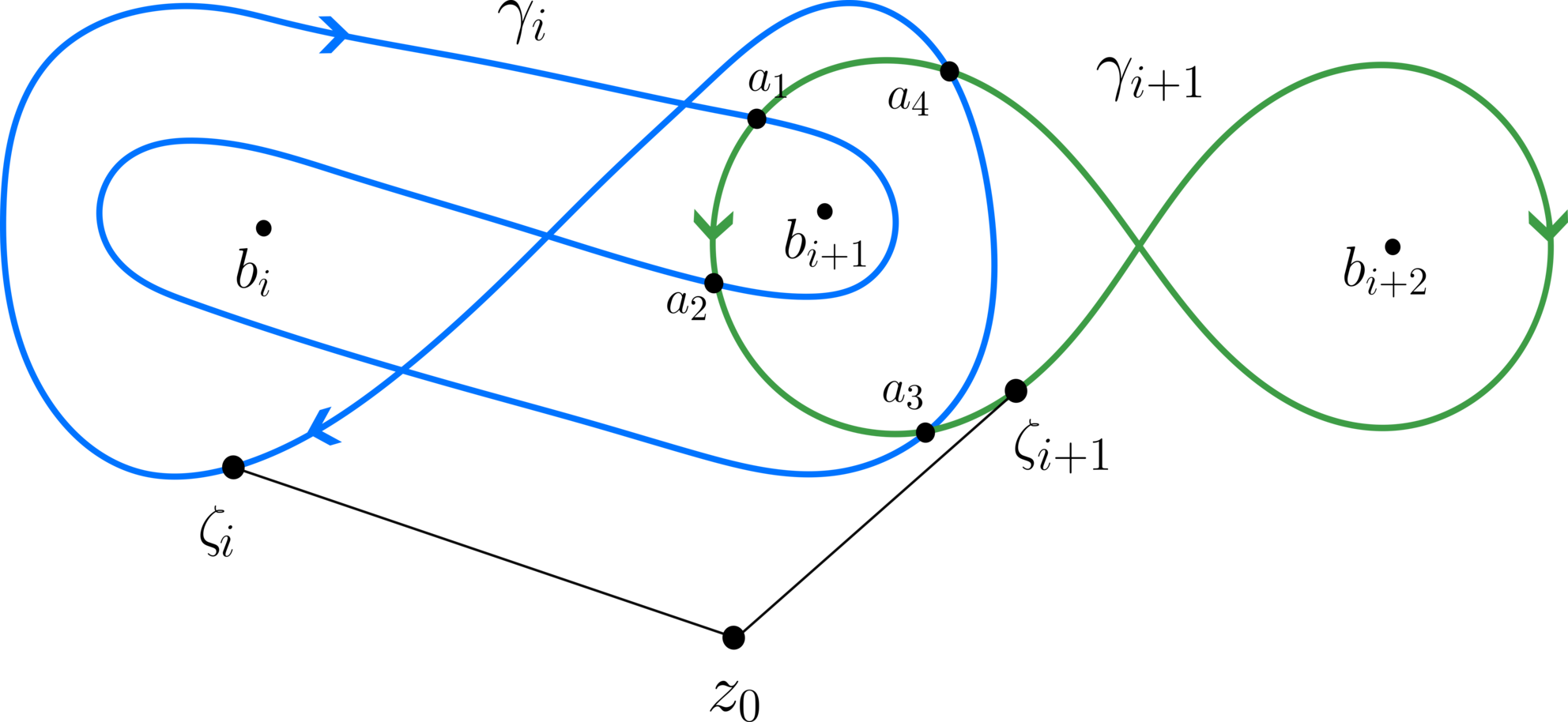}
\end{minipage}

The weighted sum of these intersections yields the raw
polynomial $(1 - t^{-k_{r-1}} - t^{-k_r} + t^{-k_{r-1}-k_r})$.
This expression factorizes as $(1 - t^{-k_{r-1}})(1 - t^{-k_r})$.
By the reduction formula
\eqref{eqn:poincare_dual_intersection_form_expansion},
the cohomology pairing evaluates to:
\begin{align*}
	\langle \omega_i, \omega_{i+1} \rangle
	&= \sqrt{-1} \frac{t^{k_{r-1}+k_r}t^{-k_r}(1 - t^{-k_{r-1}})(1 - t^{-k_r})}{(1-t^{k_{r-1}})(1-t^{k_r})(1-t^{-k_r})}
	= \frac{-\sqrt{-1}}{1 - t^{k_r}}.
\end{align*}

\vspace{1em}

\subsubsection*{Subcase 2d: Adjacent loop disks}
This situation arises when the partition block $P_r$
consists of a single branch point, making $D_i$ and
$D_{i+1}$ both \loopdisks. The disk $D_i$
connects $P_{r-1}$ to $P_r$, while $D_{i+1}$
connects $P_r$ to $P_{r+1}$. The local
monodromy exponents are $k_{r-1}$ at $b_i$,
$k_r$ at $b_{i+1}$, and $k_{r+1}$ at
$b_{i+2}$.

The curves $\gamma_i$ and $\gamma_{i+1}$
intersect eight times near the shared point
$b_{i+1}$.
\leavevmode\par
\noindent
\begin{minipage}[c]{0.56\textwidth}
Accounting for monodromies, 
$\wt\gamma_i^{\,0}$ intersects $\wt\gamma_{i+1}^{\,-k_{r-1}}$	
at $\wt a_1^{\,0}$ (sign $+1$), $\wt\gamma_{i+1}^{\,-k_{r-1}+k_{r+1}}$ 
at $\wt a_2^{\,0}$ (sign $-1$),
$\wt\gamma_{i+1}^{\,-k_{r-1}+k_{r+1}-k_r}$
at $\wt a_3^{\,0}$ (sign $+1$), $\wt\gamma_{i+1}^{\,-k_{r-1}-k_r}$
at $\wt a_4^{\,0}$ (sign $-1$), $\wt\gamma_{i+1}^{\,-k_r}$
at $\wt a_5^{\,0}$ (sign $+1$), $\wt\gamma_{i+1}^{\,-k_r+k_{r+1}}$
at $\wt a_6^{\,0}$ (sign $-1$), $\wt\gamma_{i+1}^{\,k_{r+1}}$
at $\wt a_7^{\,0}$ (sign $+1$), and $\wt\gamma_{i+1}^{\,0}$
at $\wt a_8^{\,0}$ (sign $-1$).
	\end{minipage}\hfill
\begin{minipage}[c]{0.40\textwidth}
	\centering
	\includegraphics[width=\linewidth]{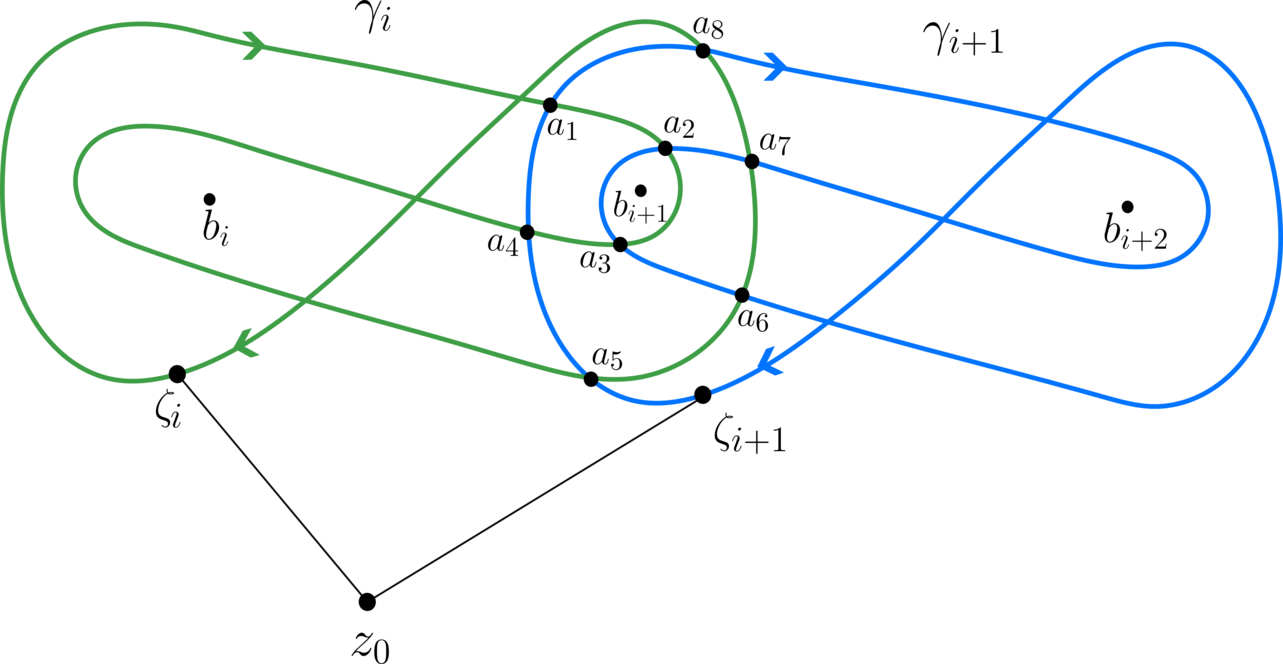}
\end{minipage}

The weighted intersection sum is
$
t^{-k_{r-1}}-t^{k_{r+1}-k_{r-1}}
	+t^{-k_{r-1}+k_{r+1}-k_r}
	-t^{-k_{r-1}-k_r}
	+t^{-k_r}-t^{-k_r+k_{r+1}}
	+t^{k_{r+1}}-1.
$
It factorizes as
$
	-(1-t^{-k_{r-1}})
	(1-t^{k_{r+1}})(1-t^{-k_r}).
$
By the reduction formula
\eqref{eqn:poincare_dual_intersection_form_expansion},
the cohomology pairing evaluates to:
\begin{align*}
	\langle \omega_i, \omega_{i+1} \rangle
	&= \sqrt{-1} \frac{t^{k_{r-1}+k_r}t^{-k_r-k_{r+1}}[-(1 - t^{-k_{r-1}})(1 - t^{k_{r+1}})(1 - t^{-k_r})]}{(1-t^{k_{r-1}})(1-t^{k_r})(1-t^{-k_r})(1-t^{-k_{r+1}})}
	= \frac{-\sqrt{-1}}{1 - t^{k_r}}.
\end{align*}

\smallskip
\noindent\textbf{Disjoint disks.} When $|i-j| \geq 2$, the
disks $D_i$ and $D_j$ are disjoint, so the lifted curves
have no intersection points and
$\langle \omega_i, \omega_j \rangle = 0$.
\end{proof}

\begin{remark}\label{rem:crossing_sum}
Let $\gamma$ and $\gamma'$ be as in Lemma~\ref{lem:reduction_formula}, with
base points $\zeta$ and $\zeta'$. For $a\in\gamma\cap\gamma'$ let
$\epsilon(a)=\pm1$ be the local intersection sign at $a$, and let $\ell_a$ be
the loop at $z_0$ used in the proof of
Proposition~\ref{prop:spanning_intersections}: the access path from $z_0$ to
$\zeta$, then $\gamma$ to $a$, then $\gamma'$ from $a$ to $\zeta'$, then the
reversed access path of $\gamma'$. Put $\phi(a)=\Phi_\zz([\ell_a])\in\zz$.
That proof shows that $\wt\gamma^{\,0}$ meets $(\wt\gamma')^{\,g}$ exactly at
the lifts of those $a$ with $\phi(a)\equiv g\pmod d$, and there with sign
$\epsilon(a)$. Since $t^d=1$, we have $t^{\,g}=t^{\phi(a)}$, so such an $a$
contributes $\epsilon(a)\,t^{\phi(a)}$. Hence
	\begin{equation}\label{eqn:crossing_sum}
		\sum_{g\in G}t^{\,g}
		\bigl\langle\wt\gamma^{\,0},(\wt\gamma')^{\,g}\bigr\rangle
		=\sum_{a\in\gamma\cap\gamma'}\epsilon(a)\,t^{\phi(a)}
		\in\zz[t^{\pm1}].
	\end{equation}
Neither $\epsilon(a)$ nor $\phi(a)$ involves $d$. The right-hand side
therefore depends only on $\gamma$, $\gamma'$ and the exponents $k_r$.
\end{remark}

\subsection{Pairings with General Loop Classes}
\label{subsection:general_loop_pairings}

For $j > i+1$, the class $\omega_{i,j}$ arises from a
non-adjacent \loopdisk\ that connects the last branch point
of $P_i$ to the first branch point of $P_j$. A general
\loopdisk\ with support $D_{i,j}$ intersects exactly four
other disks. These correspond to adjacent classes
$\omega_{h_i-1}$, $\omega_{h_i}$, $\omega_{h_{j-1}}$,
and $\omega_{h_{j-1}+1}$. The generators $\omega_{h_i-1}$
and $\omega_{h_{j-1}+1}$ are either Artin or loop type,
while the generators $\omega_{h_i}$ and $\omega_{h_{j-1}}$
are strictly loop type.

\begin{proposition}[General loop class pairings]%
\label{prop:general_mixed_intersections}
For $j>i+1$, use the convention in
Remark~\ref{rem:intersection_lifting}. The
self-pairing of $\omega_{i,j}$ is
	\begin{align}\label{eqn:mixed_self_intersection_general}
		\langle\omega_{i,j},\omega_{i,j}\rangle =
		\frac{\sqrt{-1}}{2} \left( \frac{1+t^{k_i}}{1-t^{k_i}} + \frac{1+t^{k_j}}{1-t^{k_j}} \right).
	\end{align}
For $\omega_r \neq \omega_{i,j}$ (non-self pairings), the
intersections with the adjacent classes
evaluate to:
	\begin{align}\label{eqn:mixed_intersections_general}
		\langle\omega_r,\omega_{i,j}\rangle =
		\begin{cases}
			\dfrac{-\sqrt{-1}}{1 - t^{k_i}}
			& r = h_i-1,
			\\[8pt]
			\dfrac{-\sqrt{-1}}{1 - t^{-k_i}}
			& r = h_i,
			\\[8pt]
			\dfrac{\sqrt{-1}}{1 - t^{k_j}}
			& r = h_{j-1},
			\\[8pt]
			\dfrac{\sqrt{-1}}{1 - t^{-k_j}}
			& r = h_{j-1}+1,
			\\[8pt]
			0 & \text{otherwise.}
		\end{cases}
	\end{align}
The reversed pairings
$\langle\omega_{i,j},\omega_r\rangle$ are determined
by Hermitian symmetry.
\end{proposition}

\begin{proof}
By the reduction formula
	\eqref{eqn:poincare_dual_intersection_form_expansion},
the pairing $\langle\omega_r,\omega_{i,j}\rangle$
is a weighted sum over lifted intersection points.
We evaluate each configuration separately.

\subsection*{Case 3: Self-intersection}
The self-pairing matches the geometry of the adjacent loop
classes from Subcase~1b of
Proposition~\ref{prop:spanning_intersections}.
By the reduction formula
\eqref{eqn:poincare_dual_intersection_form_expansion},
the cohomology pairing evaluates to:
\begin{align*}
	\langle \omega_{i,j}, \omega_{i,j} \rangle 
	&= \frac{\sqrt{-1}}{2} \left( \frac{1+t^{k_i}}{1-t^{k_i}} + \frac{1+t^{k_j}}{1-t^{k_j}} \right).
\end{align*}

\subsection*{Case 4: Intersection with $\omega_r$ ($r=h_i-1$)}

For $r=h_i-1$, the class $\omega_r$ is either
an Artin class or a loop class.
\leavevmode\par
\noindent
\begin{minipage}[c]{0.56\textwidth}
	\textbf{Subcase 4a: $\omega_r$ is an Artin class.}
The curves intersect four times near their shared branch point
$b_{r+1}$. The lifting geometry matches the
arrangement in Subcase~2b
of Proposition~\ref{prop:spanning_intersections}.
The cohomology pairing evaluates to:
	\begin{align*}
		\langle \omega_r, \omega_{i,j} \rangle
		&= \frac{-\sqrt{-1}}{1 - t^{k_i}}.
	\end{align*}
	\end{minipage}\hfill
\begin{minipage}[c]{0.40\textwidth}
	\centering
	\includegraphics[width=\linewidth]{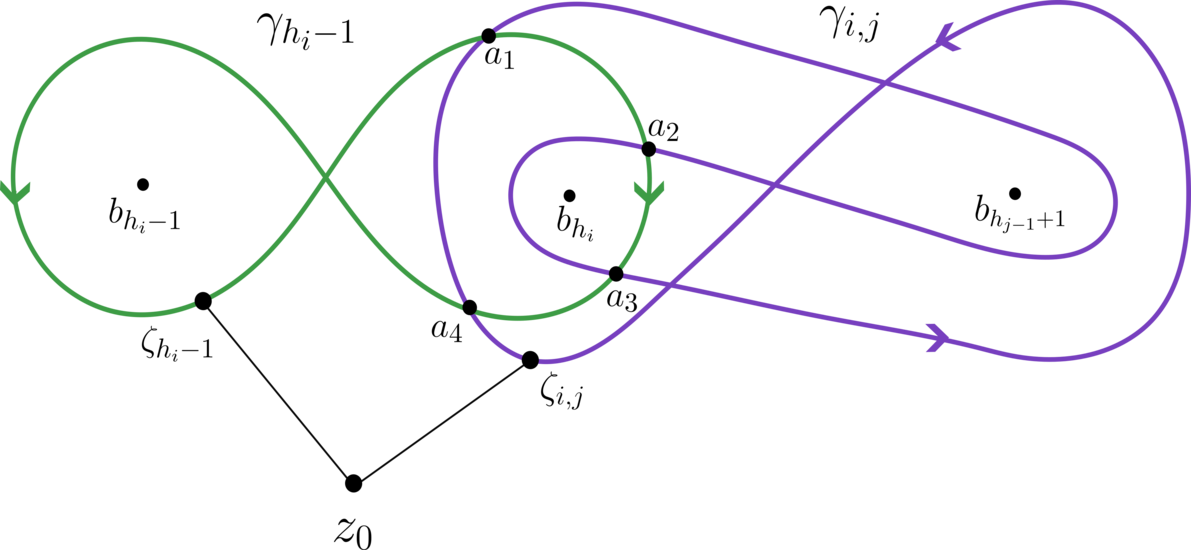}
\end{minipage}
\leavevmode\par
\noindent
\begin{minipage}[c]{0.56\textwidth}
	\textbf{Subcase 4b: $\omega_r$ is a loop class.}
The curves intersect eight times near their shared branch point
$b_{r+1}$. The lifting geometry matches the
arrangement in Subcase~2d
of Proposition~\ref{prop:spanning_intersections}.
The cohomology pairing evaluates to:
	\begin{align*}
		\langle \omega_r, \omega_{i,j} \rangle
		&= \frac{-\sqrt{-1}}{1 - t^{k_i}}.
	\end{align*}
	\end{minipage}\hfill
\begin{minipage}[c]{0.40\textwidth}
	\centering
	\includegraphics[width=\linewidth]{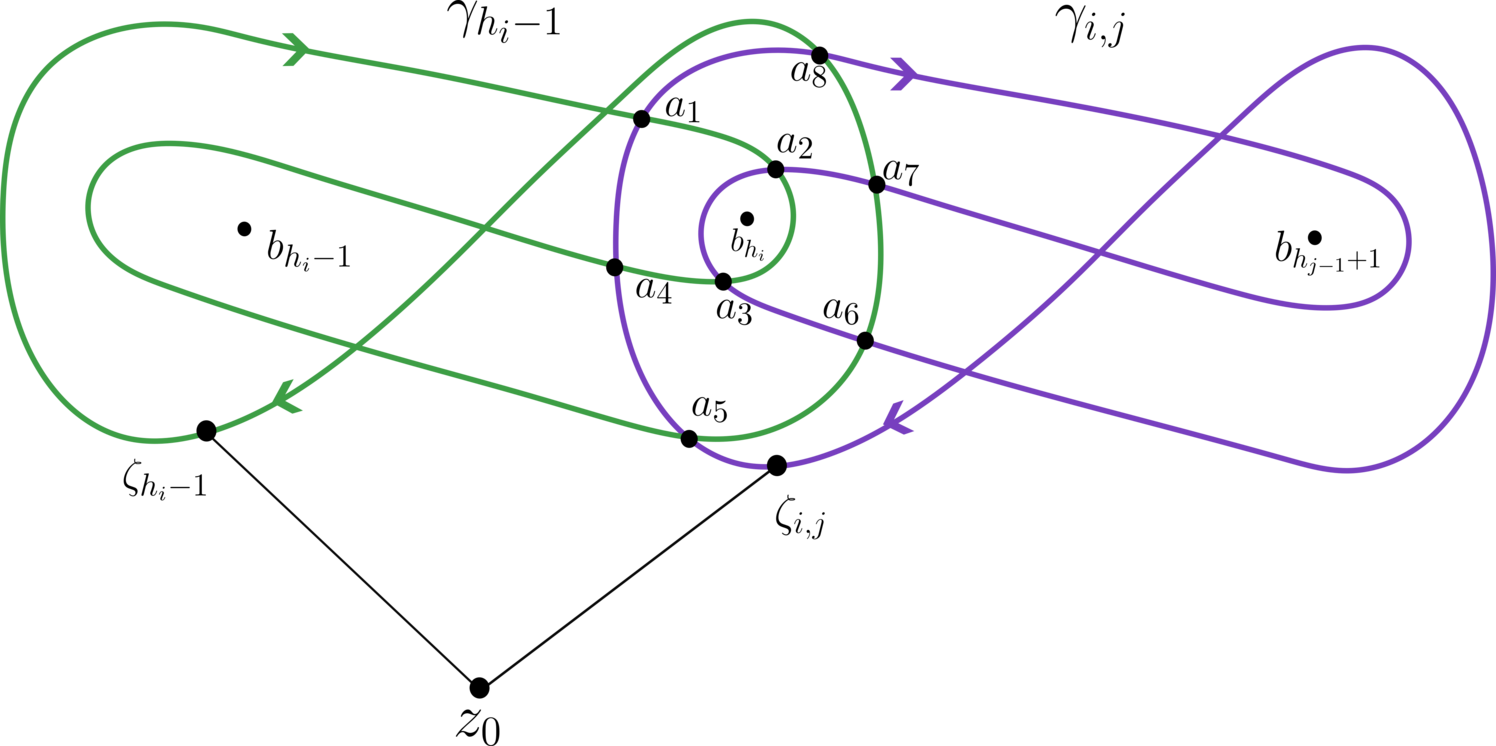}
\end{minipage}

\subsection*{Case 5: Intersection with $\omega_r$ ($r=h_i$)}
The curves $\gamma_r$ and $\gamma_{i,j}$
intersect eight times near their shared branch point
$b_r$.
\leavevmode\par
\noindent
\begin{minipage}[c]{0.56\textwidth}
Accounting for monodromies, $\wt\gamma_r^{\,0}$ intersects 
$\wt\gamma_{i,j}^{\,-k_{i+1}}$ at $\wt a_1^{\,0}$ (sign $+1$), 
$\wt\gamma_{i,j}^{\,-k_{i+1}+k_j}$ at $\wt a_2^{\,0}$ (sign $-1$), 
$\wt\gamma_{i,j}^{\,-k_{i+1}+k_j+k_i}$ at $\wt a_3^{\,0}$ (sign $+1$), 
$\wt\gamma_{i,j}^{\,-k_{i+1}+k_i}$ at $\wt a_4^{\,0}$ (sign $-1$), 
$\wt\gamma_{i,j}^{\,k_i}$ at $\wt a_5^{\,0}$ (sign $+1$), 
$\wt\gamma_{i,j}^{\,k_j+k_i}$ at $\wt a_6^{\,0}$ (sign $-1$), 
$\wt\gamma_{i,j}^{\,k_j}$ at $\wt a_7^{\,0}$ (sign $+1$), and 
$\wt\gamma_{i,j}^{\,0}$ at $\wt a_8^{\,0}$ (sign $-1$).
	\end{minipage}\hfill
\begin{minipage}[c]{0.40\textwidth}
	\centering
	\includegraphics[width=\linewidth]{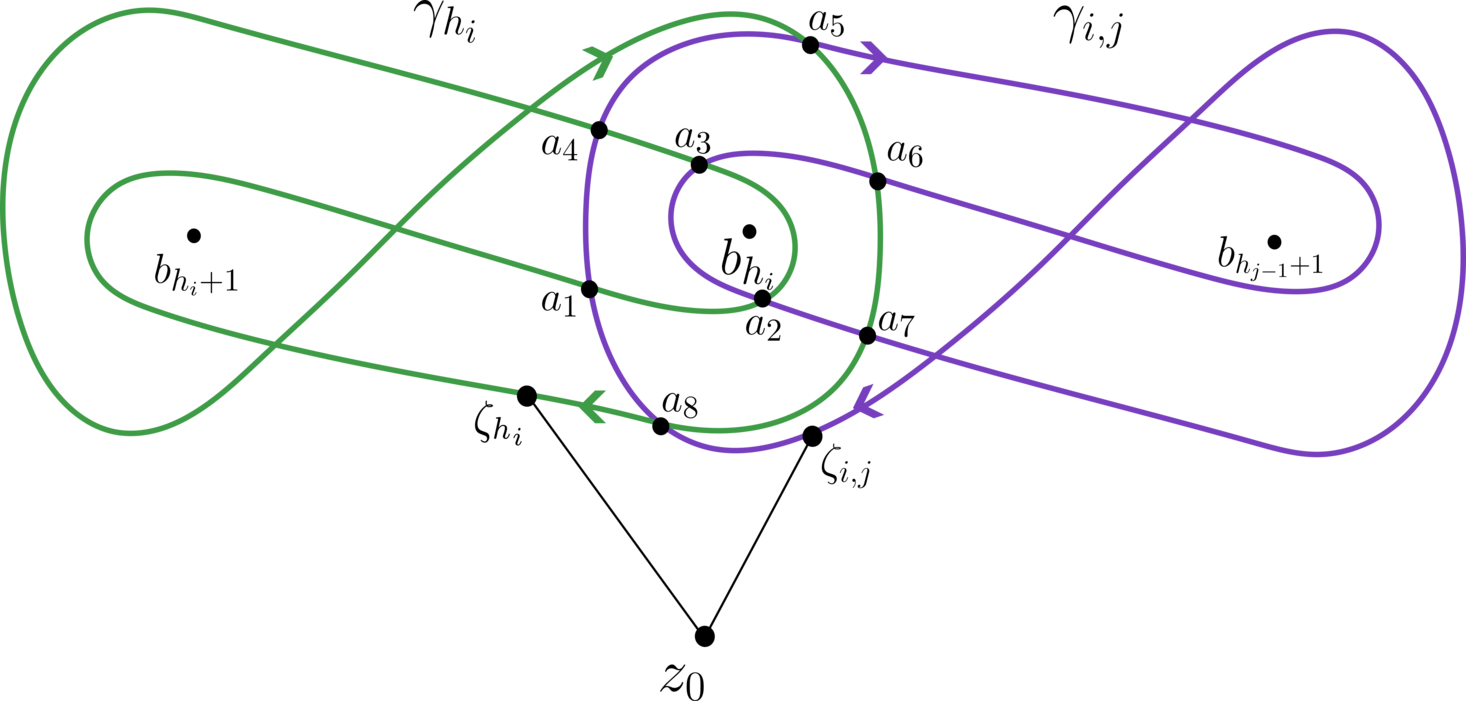}
\end{minipage}

The weighted intersection sum is
$
t^{-k_{i+1}}-t^{-k_{i+1}+k_j}
	+t^{k_i-k_{i+1}+k_j}
	-t^{k_i-k_{i+1}}
	+t^{k_i}-t^{k_i+k_j}+t^{k_j}-1.
$
It factorizes as
$
t^{-k_{i+1}}(1-t^{k_i})
	(1-t^{k_j})(1-t^{k_{i+1}}).
$
By the reduction formula
\eqref{eqn:poincare_dual_intersection_form_expansion},
the cohomology pairing evaluates to:
\begin{align*}
	\langle \omega_r, \omega_{i,j} \rangle
	= \sqrt{-1} \frac{t^{k_i+k_{i+1}} t^{-k_i - k_j} [t^{-k_{i+1}}(1 - t^{k_i})(1 - t^{k_j})(1 - t^{k_{i+1}})]}{(1-t^{k_i})(1-t^{k_{i+1}})(1-t^{-k_i})(1-t^{-k_j})} 
	= \frac{-\sqrt{-1}}{1 - t^{-k_i}}.
\end{align*}

\vspace{1em}

\subsection*{Case 6: Intersection with $\omega_r$ ($r=h_{j-1}$)}
The curves $\gamma_r$ and $\gamma_{i,j}$
intersect eight times near their shared branch point
$b_{r+1}$.
\leavevmode\par
\noindent
\begin{minipage}[c]{0.56\textwidth}
Accounting for monodromies, 
$\wt\gamma_r^{\,0}$ intersects 
$\wt\gamma_{i,j}^{\,-k_{j-1}}$ at $\wt a_1^{\,0}$ (sign $-1$), 
$\wt\gamma_{i,j}^{\,-k_{j-1}+k_i}$ at $\wt a_2^{\,0}$ (sign $+1$), 
$\wt\gamma_{i,j}^{\,-k_{j-1}+k_i-k_j}$ at $\wt a_3^{\,0}$ (sign $-1$), 
$\wt\gamma_{i,j}^{\,-k_{j-1}-k_j}$ at $\wt a_4^{\,0}$ (sign $+1$), 
$\wt\gamma_{i,j}^{\,-k_j}$ at $\wt a_5^{\,0}$ (sign $-1$), 
$\wt\gamma_{i,j}^{\,k_i-k_j}$ at $\wt a_6^{\,0}$ (sign $+1$), 
$\wt\gamma_{i,j}^{\,k_i}$	at $\wt a_7^{\,0}$ (sign $-1$), 
and $\wt\gamma_{i,j}^{\,0}$ at $\wt a_8^{\,0}$ (sign $+1$).
The weighted intersection sum is
\begin{align*}
\begin{aligned}
&-t^{-k_{j-1}}+t^{-k_{j-1}+k_i}\\
&-t^{-k_{j-1}+k_i-k_j}
+t^{-k_{j-1}-k_j}-t^{-k_j}\\
&+t^{k_i-k_j}-t^{k_i}+1.
\end{aligned}
\end{align*}
It factorizes as
\begin{align*}
t^{-k_{j-1}-k_j}(1-t^{k_i})
(1-t^{k_j})(1-t^{k_{j-1}}).
\end{align*}
	\end{minipage}\hfill
\begin{minipage}[c]{0.40\textwidth}
	\centering
	\includegraphics[width=\linewidth]{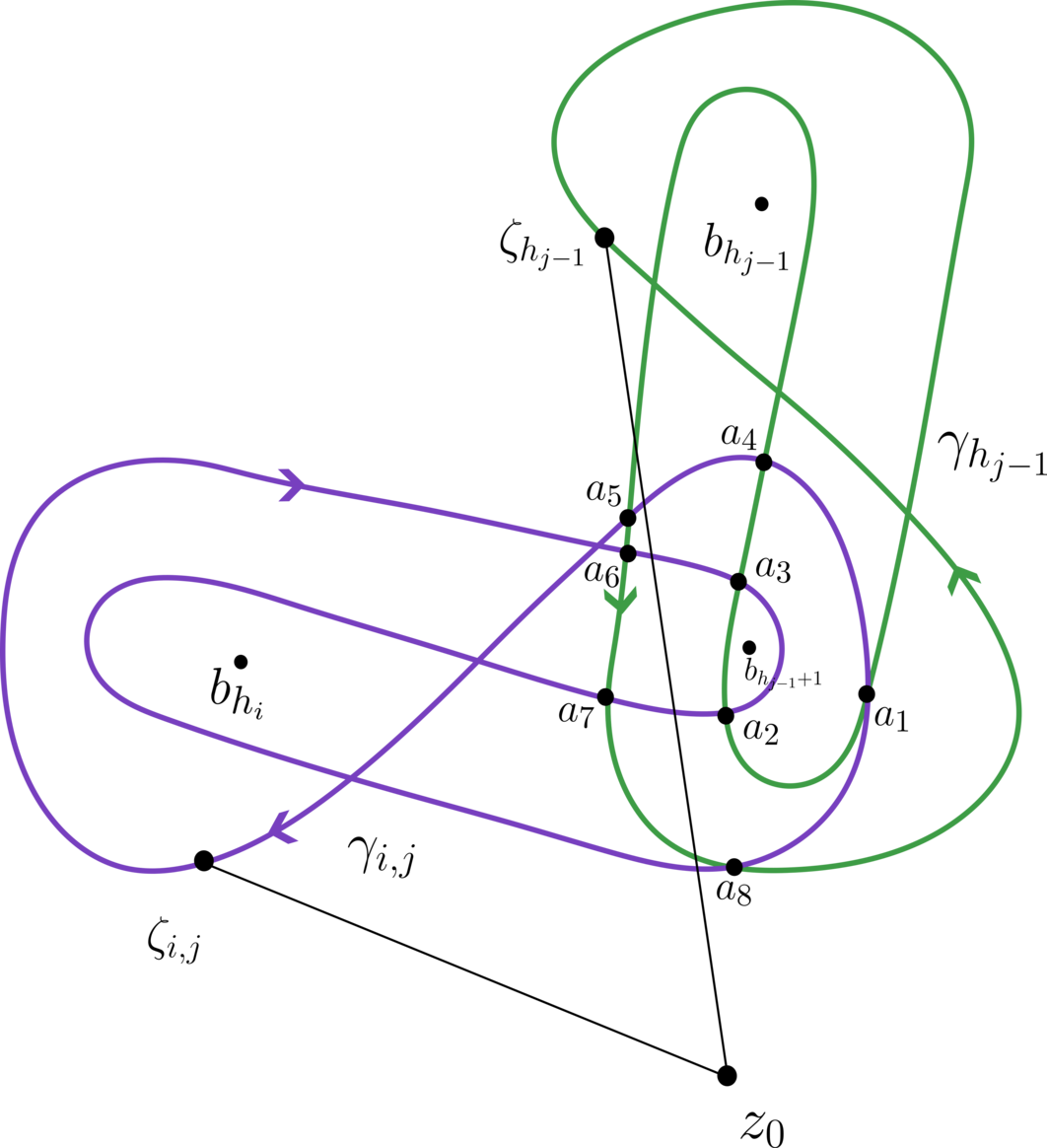}
\end{minipage}

By the reduction formula
\eqref{eqn:poincare_dual_intersection_form_expansion},
the cohomology pairing evaluates to:
\begin{align*}
	\langle \omega_r, \omega_{i,j} \rangle
	= \sqrt{-1} \frac{t^{k_{j-1}+k_j} t^{-k_i-k_j} [t^{-k_{j-1}-k_j}(1 - t^{k_i})(1 - t^{k_j})(1 - t^{k_{j-1}})]}{(1-t^{k_{j-1}})(1-t^{k_j})(1-t^{-k_i})(1-t^{-k_j})} 
	= \frac{\sqrt{-1}}{1 - t^{k_j}}.
\end{align*}

\subsection*{Case 7: Intersection with $\omega_r$ ($r=h_{j-1}+1$)}
\leavevmode\par
\noindent
\begin{minipage}[c]{0.56\textwidth}

\textbf{Subcase 7a: $\gamma_r$ is an Artin class.}
The curves $\gamma_r$ and $\gamma_{i,j}$ intersect four times.
Accounting for monodromies, $\wt\gamma_r^{\,0}$ intersects 
$\wt\gamma_{i,j}^{\,0}$ at $\wt a_1^{\,0}$ (sign $-1$), 
$\wt\gamma_{i,j}^{\,k_i}$ at 	$\wt a_2^{\,0}$ (sign $+1$), 
$\wt\gamma_{i,j}^{\,k_i+k_j}$ at $\wt a_3^{\,0}$ (sign $-1$), and 
$\wt\gamma_{i,j}^{\,k_j}$ at $\wt a_4^{\,0}$ (sign $+1$).
	\end{minipage}\hfill
\begin{minipage}[c]{0.40\textwidth}
	\centering
	\includegraphics[width=\linewidth]{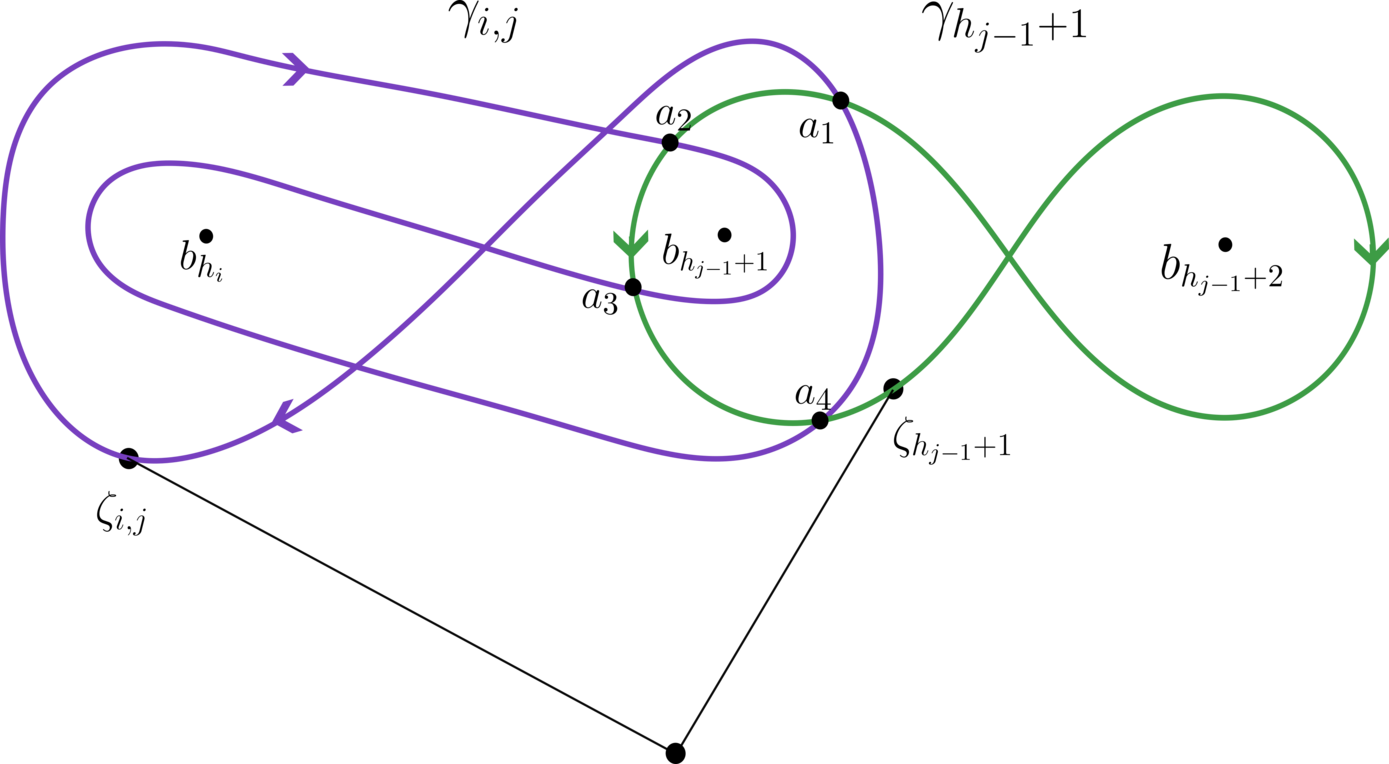}
\end{minipage}

The weighted sum of these intersections yields the raw
polynomial $(-1 + t^{k_i} - t^{k_i+k_j} + t^{k_j})$.
This expression factorizes as $-(1 - t^{k_i})(1 - t^{k_j})$.
By the reduction formula
\eqref{eqn:poincare_dual_intersection_form_expansion},
the cohomology pairing evaluates to:
\begin{align*}
	\langle \omega_r, \omega_{i,j} \rangle
	= \sqrt{-1} \frac{t^{k_j} t^{-k_i - k_j}[-(1 - t^{k_i})(1 - t^{k_j})]}{(1-t^{k_j})(1-t^{-k_i})(1-t^{-k_j})}
	= \frac{\sqrt{-1}}{1 - t^{-k_j}}.
\end{align*}

\vspace{1em}

\textbf{Subcase 7b: $\gamma_r$ is a loop class.}
The curves $\gamma_r$ and $\gamma_{i,j}$ intersect
eight times near their shared branch point.
\leavevmode\par
\noindent
\begin{minipage}[c]{0.56\textwidth}
Accounting for monodromies, $\wt\gamma_r^{\,0}$ intersects 
$\wt\gamma_{i,j}^{\,k_j}$ at $\wt a_1^{\,0}$ (sign $-1$), 
$\wt\gamma_{i,j}^{\,k_j+k_i}$ at $\wt a_2^{\,0}$ (sign $+1$), 
$\wt\gamma_{i,j}^{\,k_i}$ at $\wt a_3^{\,0}$ (sign $-1$), 
$\wt\gamma_{i,j}^{\,0}$ at $\wt a_4^{\,0}$ (sign $+1$), 
$\wt\gamma_{i,j}^{\,-k_{j+1}}$ at $\wt a_5^{\,0}$ (sign $-1$), 
$\wt\gamma_{i,j}^{\,-k_{j+1}+k_i}$ at $\wt a_6^{\,0}$ (sign $+1$), 
$\wt\gamma_{i,j}^{\,-k_{j+1}+k_i+k_j}$
at $\wt a_7^{\,0}$ (sign $-1$), and
$\wt\gamma_{i,j}^{\,-k_{j+1}+k_j}$ at $\wt a_8^{\,0}$ (sign $+1$).
	\end{minipage}\hfill
\begin{minipage}[c]{0.40\textwidth}
	\centering
	\includegraphics[width=\linewidth]{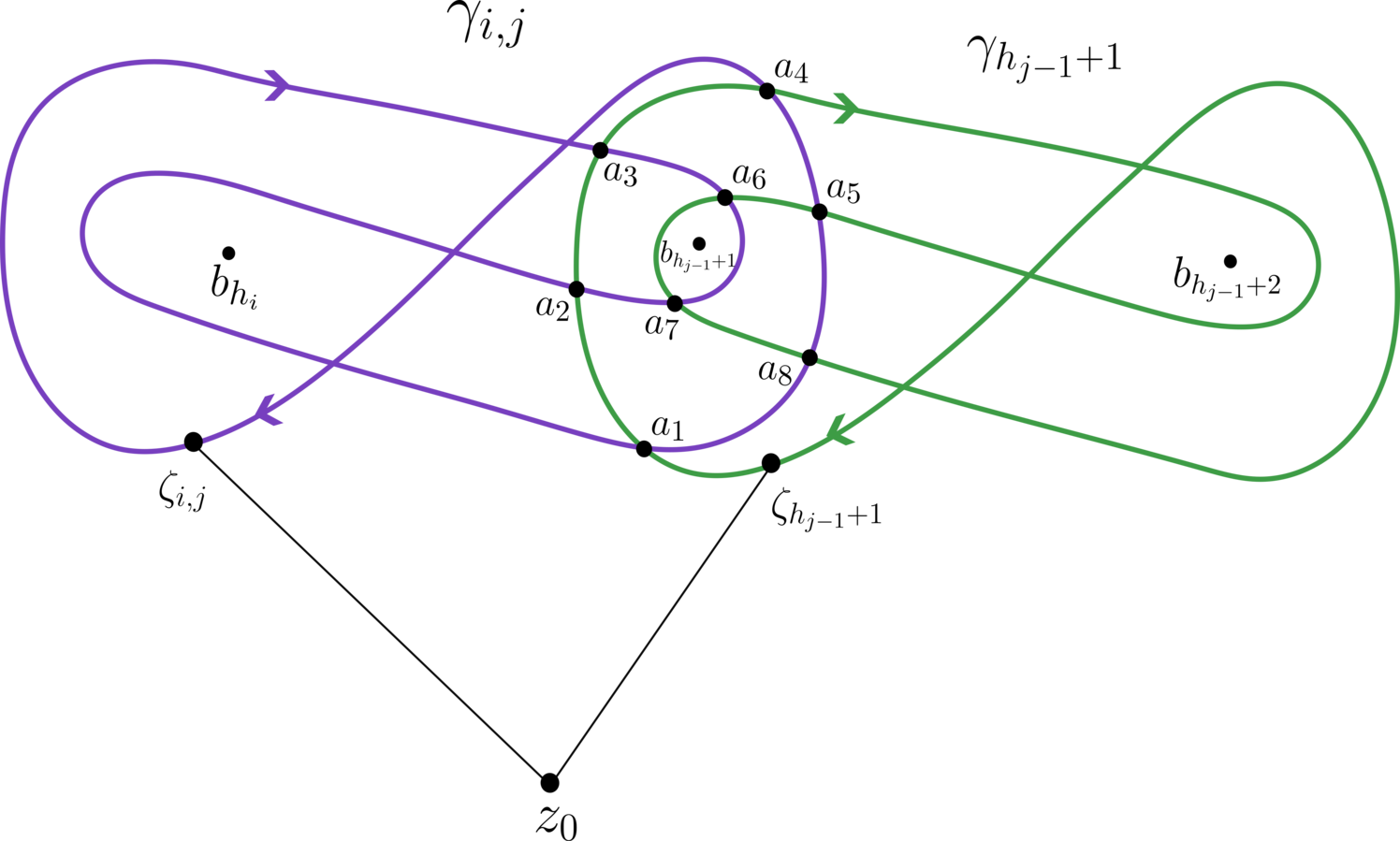}
\end{minipage}

The weighted intersection sum is
$
	1-t^{k_i}-t^{k_j}+t^{k_i+k_j}
	-t^{-k_{j+1}}+t^{-k_{j+1}+k_i}
	+t^{-k_{j+1}+k_j}
	-t^{-k_{j+1}+k_i+k_j}.
$
It factorizes as
$
	(1-t^{k_i})(1-t^{k_j})
	(1-t^{-k_{j+1}}).
$
By the reduction formula
\eqref{eqn:poincare_dual_intersection_form_expansion},
the cohomology pairing evaluates to:
\begin{align*}
	\langle \omega_r, \omega_{i,j} \rangle
	= \sqrt{-1} \frac{t^{k_j + k_{j+1}} t^{-k_i - k_j} [ (1 - t^{k_i})(1 - t^{k_j})(1 - t^{-k_{j+1}}) ]}{(1 - t^{k_j})(1 - t^{k_{j+1}})(1 - t^{-k_i})(1 - t^{-k_j})} 
	= \frac{\sqrt{-1}}{1 - t^{-k_j}}.
\end{align*}
This evaluates every generator pairing and proves the
proposition.
\end{proof}

\subsection{The Intersection Matrix and Spanning}
\label{subsection:intersection_matrix}

Without loss of generality, place the
branch points sequentially on the positive
integers $1,2,\dots,n$. When a block
$P_i$ is $t$-invisible, its branch points
are dropped: the corresponding spanning
classes vanish, and those integers are no
longer treated as branch points.

We begin with the $n-1$ adjacent classes
$\omega_1,\dots,\omega_{n-1}$ as a
candidate spanning set for $H^1(X)_t$.
By Remark~\ref{rem:t_invisible_blocks}, a
block $P_i$ is $t$-invisible when
$t^{k_i}=1$; then every local class with
a branch point in $P_i$ vanishes.  Hence
for $2 \le i \le m-1$, the
$|P_i|+1$ classes
$\omega_{h_{i-1}},\dots,\omega_{h_i}$
vanish; for a boundary block $P_1$ or
$P_m$, only the $|P_i|$ classes
$\omega_1,\dots,\omega_{h_1}$ or
$\omega_{h_{m-1}},\dots,\omega_{n-1}$
vanish.

\begin{figure}[ht]
\centering
\includegraphics[width=\textwidth]{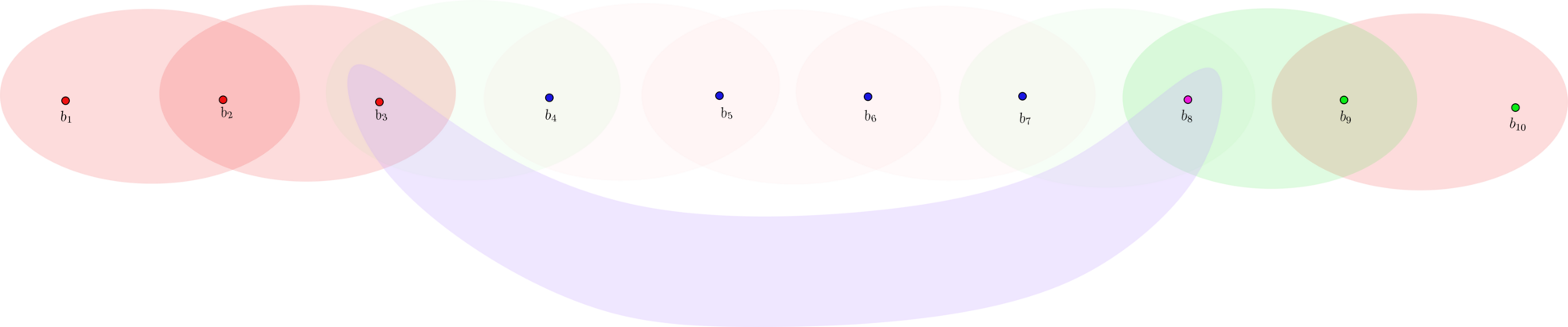}
\caption{Spanning set with a
$t$-invisible block.
All adjacent classes meeting the invisible
block vanish. They are collectively replaced
by one general loop disk (purple).}
\label{fig:disks_generating_set}
\end{figure}
Let $P_i,\dots,P_j$ be a maximal consecutive
sequence of $t$-invisible blocks.  The zero classes
$\omega_{h_{i-1}},\dots,\omega_{h_j}$ are
replaced by the single general loop class
$\omega_{i-1,j+1}$, supported on the disk from
$b_{h_{i-1}}$ to $b_{h_j+1}$.  The resulting
set has cardinality $n-1-r_t$, where
$r_t$ is the total number of branch
points in $t$-invisible blocks. We prove below that
these classes span $H^1(X)_t$. When
$t^{k_\infty}=1$, the
spanning
classes satisfy a single linear relation.

The pairings among the surviving classes are
given by
Propositions~\ref{prop:spanning_intersections}
and~\ref{prop:general_mixed_intersections}.

\begin{proposition}\label{thm:spanning_via_intersection}
Let $\omega_1,\dots,\omega_{\ell-1}$ be
the classes constructed above, where
$\ell = n - r_t$. Let $M$ be their
$(\ell-1)\times(\ell-1)$ intersection
matrix, with entries given by
Propositions~\ref{prop:spanning_intersections}
and~\ref{prop:general_mixed_intersections}.
If $t^{k_\infty} \neq 1$, then
$\det(M) \neq 0$. Thus, the classes are
linearly independent and, by
Proposition~\ref{prop:dim_cyclic_cover},
form a basis of $H^1(X)_t$.
If $t^{k_\infty} = 1$, then
$\det(M) = 0$ and
$\operatorname{rank}(M) = \ell-2$.
Consequently, the classes satisfy a
single linear relation.
\end{proposition}

\begin{proof}
The matrix $M$ is formed by collapsing
each maximal sequence of $t$-invisible
blocks into a single general loop class.
Let $b_{v_1},\dots,b_{v_\ell}$ be the
visible branch points in sequential
order. Its entries are
$M_{i,i} = \langle\omega_i,\omega_i\rangle$,
$M_{i,i+1} = \langle\omega_i,\omega_{i+1}\rangle$,
and $M_{i+1,i} = \langle\omega_{i+1},\omega_i\rangle$,
as given by
Propositions~\ref{prop:spanning_intersections}
and~\ref{prop:general_mixed_intersections}.
By induction on $\ell$, the determinant is
	\begin{align*}
		\det(M)
		= \frac{(\sqrt{-1})^{\ell-1}
		\left(
		1-\prod_{i=1}^{\ell}t^{k_{\pt(v_i)}}
		\right)}
		{\prod_{i=1}^{\ell}
		(1-t^{k_{\pt(v_i)}})}.
	\end{align*}
Every omitted, $t$-invisible branch point contributes
a factor equal to one. Since
	\begin{align*}
k_\infty
		\equiv-\sum_{r=1}^{n}k_{\pt(r)}\pmod d,
	\end{align*}
we have
	\begin{align*}
		\prod_{i=1}^{\ell}t^{k_{\pt(v_i)}}
		=t^{-k_\infty}.
	\end{align*}
Consequently,
	\begin{align*}
		\det(M)
		= \frac{(\sqrt{-1})^{\ell-1}
		(1-t^{-k_\infty})}
		{\prod_{i=1}^{\ell}
		(1-t^{k_{\pt(v_i)}})}.
	\end{align*}
Each visible factor satisfies
$t^{k_{\pt(v_i)}}\neq1$, so the denominator is
non-zero.

When $t^{k_\infty} \neq 1$, we have $\det(M) \neq 0$.
Proposition~\ref{prop:dim_cyclic_cover} states
$\dim H^1(X)_t = \ell-1$. Thus, the $\ell-1$ classes
are linearly independent and form a basis.

When $t^{k_\infty} = 1$, we have $\det(M) = 0$. Let
$M'$ be the principal submatrix obtained by deleting
the final row and column of $M$. Applying the identical
inductive logic to $M'$ yields
	\begin{align*}
		\det(M') = \frac{(\sqrt{-1})^{\ell-2}
		\bigl(1 - \prod_{i=1}^{\ell-1} t^{k_{\pt(v_i)}}\bigr)}
		{\prod_{i=1}^{\ell-1} (1-t^{k_{\pt(v_i)}})}.
	\end{align*}
Since each block is visible, the denominator is
non-zero. The hypothesis
$t^{k_\infty} = 1$ ensures
$\prod_{i=1}^{\ell}t^{k_{\pt(v_i)}}=1$.
Since $t^{k_{\pt(v_\ell)}}\neq1$, we deduce
$\prod_{i=1}^{\ell-1}t^{k_{\pt(v_i)}}\neq1$.
Consequently, the numerator of $\det(M')$ is non-zero,
which implies $\det(M') \neq 0$. This establishes
$\operatorname{rank}(M) = \ell-2$. By
Proposition~\ref{prop:dim_cyclic_cover},
$\dim H^1(X)_t = \ell-2$. The classes therefore span
$H^1(X)_t$ and possess exactly one linear relation.
Proposition~\ref{prop:spanning_relations} gives this
relation explicitly.
\end{proof}

\subsection{Linear Relations}
\label{subsection:linear_relations}

By Proposition~\ref{prop:dim_cyclic_cover},
$\dim H^1(X)_t = n - 1 - r_t$
(or $n - 2 - r_t$ if $t^{k_\infty}=1$).
We constructed $n-1-r_t$ spanning
classes $\omega_1,\dots,\omega_{n-1-r_t}$.
When $t^{k_\infty}=1$, there is exactly
one linear relation among them.

\begin{proposition}\label{prop:spanning_relations}
Let $\ell = n - r_t$, and let $b_{v_1},\dots,b_{v_\ell}$ be 
the visible branch points in sequential order.
The classes $\omega_1,\dots,\omega_{\ell-1}$ span $H^1(X)_t$. 
If $t^{k_\infty}=1$, the unique linear relation is
	\begin{align*}
		\sum_{j=1}^{\ell-1} \bigl(1 - t^{-\kappa_j}\bigr) \omega_j = 0,
	\end{align*}
where $\kappa_j = \sum_{m=1}^j k_{\pt(v_m)}$.
\end{proposition}

\begin{proof}
Assume $t^{k_\infty}=1$. Define the
candidate relation
$\omega = \sum_{j=1}^{\ell-1} \ov{c_j} \omega_j$,
where $c_j = 1 - t^{\kappa_j}$.
Setting $c_0=0$ and $c_\ell=0$ is consistent because
$\kappa_0=0$ and
	\begin{align*}
t^{\kappa_\ell}
		=\prod_{j=1}^{\ell}t^{k_{\pt(v_j)}}
		=t^{-k_\infty}=1.
	\end{align*}
We prove $\omega = 0$.

For any spanning class $\omega_i$,
expand $\langle\omega_i,\omega\rangle$
and let $M$ be the matrix from
Proposition~\ref{thm:spanning_via_intersection}.
Thus, $\langle\omega_i,\omega\rangle = (Mc)_i$.
Since $M$ is tridiagonal,
$(Mc)_i = M_{i,i-1}c_{i-1}
	+ M_{i,i}c_i + M_{i,i+1}c_{i+1}$.
Substituting the intersection values from
Propositions~\ref{prop:spanning_intersections}
and~\ref{prop:general_mixed_intersections},
this evaluates to zero. For the boundary
case $i = \ell - 1$, this evaluation
explicitly relies on the hypothesis
$t^{k_\infty} = 1$. By non-degeneracy
of the intersection form on $H^1(X)_t$,
we conclude $\omega=0$, and hence
the relation holds.
\end{proof}

\begin{proposition}[Non-spanning mixed classes]%
\label{prop:mixed_linear_combination}
Suppose that every block $P_i,\dots,P_j$ is
$t$-visible. For $j>i+1$, the non-spanning mixed
class decomposes as
	\begin{align*}
		\omega_{i,j} = \sum_{\alpha=h_i}^{h_{j-1}} \omega_\alpha.
	\end{align*}
\end{proposition}

\begin{proof}
Since the spanning classes $\{\omega_r\}$ span $H^1(X)_t$,
the non-degenerate intersection form determines a class
uniquely. It therefore suffices to verify that
	\begin{align*}
		\langle \omega_r, \omega_{i,j} \rangle
		= \Bigl\langle \omega_r,
		   \sum_{\alpha=h_i}^{h_{j-1}} \omega_\alpha \Bigr\rangle
	\end{align*}
holds for every spanning class $\omega_r$.
When $r < h_i-1$ or $r > h_{j-1}+1$, the support
of $\omega_r$ is disjoint from the support of
$\omega_{i,j}$, so both sides vanish.
For $h_i < r < h_{j-1}$, the tridiagonality of the
intersection pairing
	(Proposition~\ref{prop:spanning_intersections})
reduces the sum
$\sum_{\alpha=h_i}^{h_{j-1}}
	\langle \omega_r, \omega_\alpha \rangle$
to at most three terms with
$\alpha = r-1, r, r+1$. Substituting the
explicit formulas from that proposition,
these terms cancel, giving zero. The remaining
boundary cases
$r \in \{h_i-1, h_i, h_{j-1}, h_{j-1}+1\}$
are verified by direct evaluation using
Proposition~\ref{prop:spanning_intersections}.
Each reduces to one or two intersection
pairings and matches the left side exactly.
\end{proof}

\begin{remark}
If $t$-invisible blocks occur, first pass to the
visible-only configuration described in
Subsection~\ref{subsection:intersection_matrix}. A bridge
class across an invisible run is adjacent in that
configuration and belongs to the spanning set. The
proposition applies to non-adjacent loop classes after
reindexing the visible configuration.
\end{remark}

\begin{remark}[Isotropic classes]\label{rem:isotropic}
For a spanning class $\omega_\sigma$, its
self-intersection vanishes exactly when $t^{k_{\pt(i)}}=-1$ for
an \artindisk, or when $t^{k_i+k_{i+1}}=1$ with
$t^{k_i} \neq 1$ and $t^{k_{i+1}} \neq 1$ for a \loopdisk.
In such instances, the class $\omega_\sigma$ is isotropic.
\end{remark}

%%%%%%%%%%%%%%%%%%%%%%%%%%%%%%%%%%
\section{The Monodromy Representation}\label{section:action_mbg}

We compute the representation
$\rho_t \colon \bnp \to \operatorname{GL}(H^1(X)_t)$.
We determine its action on spanning classes and
establish Theorem~A. Fix a generator
$\sigma \in \bnp$ with support disk $\dd_\sigma$
and a spanning class $\omega$ with support disk
$\dd$. Neither branch point of $\dd$ belongs
to a $t$-invisible block, so $\omega \neq 0$
and $\dim H^1_c(\wt\dd)_t = 1$.

If $\dd$ and $\dd_\sigma$ are disjoint, their
preimages $\wt\dd$ and $\wt\dd_\sigma$ are disjoint,
and $\sigma$ acts as the identity on $\omega$.
If $\dd = \dd_\sigma$, the action of $\sigma$
on $\omega_\sigma$ restricts to the $1$-dimensional
space $H^1_c(\wt\dd_\sigma)_t$. This self-action
is computed in Subsection~\ref{subsection:action_own_support}.

Assume now that $\dd$ and $\dd_\sigma$ are distinct
but intersect. By the geometry of the chosen
supports, they intersect in a smaller disk
containing exactly one shared branch point.
The Mayer--Vietoris sequence for compactly
supported cohomology of $\wt\dd_\sigma \cup \wt\dd$ is
\begin{align*}
H^1_c(\wt\dd_\sigma \cap \wt\dd)
	\to H^1_c(\wt\dd_\sigma) \oplus
H^1_c(\wt\dd)
	\to H^1_c(\wt\dd_\sigma \cup \wt\dd)
	\to H^2_c(\wt\dd_\sigma \cap \wt\dd).
\end{align*}
The intersection $\wt\dd_\sigma \cap \wt\dd$
is a disjoint union of disks above the single
shared branch point, making
$H^1_c(\wt\dd_\sigma \cap \wt\dd) = 0$.
Specializing to $t$-eigenspaces, we find
$H^2_c(\wt\dd_\sigma \cap \wt\dd)_t = 0$
because the shared point is some $b_s$ with
$t^{k_{\pt(s)}}\neq1$. Otherwise $\dd$ would
contain a $t$-invisible branch point.
Consequently, we obtain the short exact sequence
\begin{equation*}
	0 \to H^1_c(\wt\dd_\sigma)_t \oplus
H^1_c(\wt\dd)_t
	\to H^1_c(\wt\dd_\sigma \cup \wt\dd)_t \to 0.
\end{equation*}
If the other branch point of $\dd_\sigma$
is $t$-invisible, then $H^1_c(\wt\dd_\sigma)_t = 0$
and $H^1_c(\wt\dd_\sigma \cup \wt\dd)_t
\cong H^1_c(\wt\dd)_t$. Otherwise, the space
$H^1_c(\wt\dd_\sigma \cup \wt\dd)_t$ is the direct sum
$H^1_c(\wt\dd_\sigma)_t \oplus H^1_c(\wt\dd)_t$,
which is two-dimensional. Therefore, to determine
the action of $\sigma$ on $\omega$, it suffices
to compute its action on
$H^1_c(\wt\dd_\sigma \cup \wt\dd)_t$.
This intersecting case, along with the synthesis
of the full representation, is the content of
Subsections~\ref{subsection:action_adjacent_classes}
and \ref{subsection:representation_theorem}.

\subsection{The action of a generator on its
own support}
\label{subsection:action_own_support}

We compute the action of a generator
$\sigma$ on $H^1_c(\wt\dd_\sigma)_t$.
If either branch point of
$\dd_\sigma$ lies in a $t$-invisible
block, then
$H^1_c(\wt\dd_\sigma)_t = 0$ and
$\sigma$ acts as the identity.  We
therefore restrict to generators
whose support disk has no branch
point in a $t$-invisible block.

\subsubsection*{Action of a generator on its own class}

The Artin and loop generators have the same local form.
Identify their support disk with $\mdd(0,2)$ so that the
two branch points correspond to $-1$ and $1$. Fix
$1<r<s<2$ and a smooth non-increasing function
$\chi\colon[0,2]\longrightarrow[0,1]$
such that $\chi=1$ on $[0,r]$ and $\chi=0$ on
$[s,2]$. For $\theta=\pi$ or $2\pi$, define
$R_\theta(z)
=\exp\!\bigl(\sqrt{-1}\theta\chi(|z|)\bigr)z.$
Extend $R_\theta$ by the identity outside $\mdd(0,2)$.
Its real-valued annular angle lift decreases from
$\theta$ at the inner boundary to zero at the outer
boundary. Thus $R_\pi$ is the right half-twist and
$R_{2\pi}$ is the right full twist.

By \eqref{eqn:poincare_naturality}, the action on a
cohomology class is computed by pushing forward its
Poincar\'e-dual cycle. We therefore evaluate
$(\wt{R_\theta})_*(\wt\gamma)$ for the corresponding
inner contour $\gamma$.

Assume the basepoint $z_0$ lies outside $\mdd(0,2)$.
Let $\gamma$ denote a representative of its homology
class supported entirely within the inner disk $\mdd(0,r)$.
Pick a curve $\delta$ in the lower half-plane from
$z_0$ to the local basepoint $\zeta$
(Remark~\ref{rem:base_point_convention}),
and let $\delta'$ be the path from $\zeta$ to $0$ along $\gamma$.
The homology class of $\gamma$ corresponds to the homotopy
class of the based loop $\delta \cdot \gamma \cdot \ov\delta$.

By convention, the homology class $\wt\gamma^0$ corresponds
to the lift based at $\wt z_0^0$. Because $R_\theta$ is the
identity outside $\mdd(0,s)$ and
the canonical lift
$\wt{R_\theta}$ is fixed near the fiber at infinity
(Remark~\ref{rem:canonical_lift}), it follows that
$\wt{R_\theta}$ fixes $\wt z_0^0$. The pushforward
$(\wt{R_\theta})_*(\wt\gamma^0)$ is therefore computed by
lifting the transformed loop
$R_\theta(\delta \cdot \gamma \cdot \ov\delta)$ from $\wt z_0^0$.

The composite path $\delta\cdot\delta'$ runs from $z_0$
to $0$ through the lower half-plane. Define
\begin{align*}
\alpha
=\delta\cdot\delta'
\cdot\ov{R_\theta(\delta')}
\cdot\ov{R_\theta(\delta)},
\end{align*}
which forms a loop based at $z_0$ whose monodromy $k_\alpha$
depends on the specific generator. Using the identity
$R_\theta(\delta\cdot\delta') = \ov\alpha \cdot \delta\cdot\delta'$
and denoting the transformed inner loop by $\gamma'$,
we expand the conjugate:
\begin{align*}
R_\theta(\delta \cdot \gamma \cdot \ov\delta)
	&= R_\theta(\delta\cdot\delta' \cdot \ov{\delta'}\cdot \gamma \cdot \delta' \cdot \ov{\delta'}\cdot\ov\delta) \\
	&= R_\theta(\delta\cdot\delta') \cdot R_\theta(\ov{\delta'}\cdot\gamma\cdot\delta') \cdot \ov{R_\theta(\delta\cdot\delta')} \\
	&= (\ov\alpha \cdot \delta\cdot\delta') \cdot (\ov{\delta'}\cdot\gamma'\cdot\delta') \cdot (\ov{\delta'}\cdot\ov\delta \cdot \alpha) \\
	&= \ov\alpha \cdot (\delta \cdot \gamma' \cdot \ov\delta) \cdot \alpha.
\end{align*}

Lifting this composite loop from $\wt z_0^0$ proceeds
in three segments. The initial segment $\ov\alpha$ has
monodromy $-k_\alpha$ and terminates at
$T_*^{-k_\alpha}(\wt z_0^0)$. The middle factor
$\delta \cdot \gamma' \cdot \ov\delta$ lifts to the translated
cycle $(\wt\gamma')^{-k_\alpha}$. Evaluating the complete lift
in homology thus yields:
\begin{align*}
	(\wt{R_\theta})_*(\wt\gamma^0) &= (\wt\gamma')^{-k_\alpha}.
\end{align*}

\paragraph{Artin generators.}
For the Artin generator $\sigma_i$, map $b_i$ and $b_{i+1}$
to $-1$ and $1$. The half-twist acts as $\theta = \pi$.
The loop $\alpha$ encircles $1$ clockwise, possessing
monodromy $k_\alpha = -k_{\pt(i)}$. 

\noindent
\begin{minipage}[c]{0.65\textwidth}
The rotation reverses the
inner loop, yielding $\gamma' = \ov\gamma_i$. Evaluating the
homology pushforward yields
\begin{align*}
	(\wt{R_\pi})_*(\wt\gamma_i^0) &= -\wt\gamma_i^{k_{\pt(i)}}.
\end{align*}

\begin{lemma}[Action of Artin generators on their own class]
	\label{lem:artin_self_action}
For $t \in \mu'(d)$ with $t\notin\mu(e_{\pt(i)})$, the Artin
generator $\sigma_i$ acts on $\omega_i$ by
	\begin{align*}
		\rho_t(\sigma_i)(\omega_i) &= -t^{k_{\pt(i)}}\omega_i.
	\end{align*}
\end{lemma}
\begin{proof}
Recall the definition of the eigenspace lift from
	\eqref{eqn:artin_atom}. The evaluation
$(\wt{R_\pi})_*(\wt\gamma_i^0) = -\wt\gamma_i^{k_{\pt(i)}}$
on the base homology class directly yields the required
action on cohomology.
\end{proof}
\end{minipage}\hfill
\begin{minipage}[c]{0.32\textwidth}
	\centering
	\includegraphics[width=\linewidth]{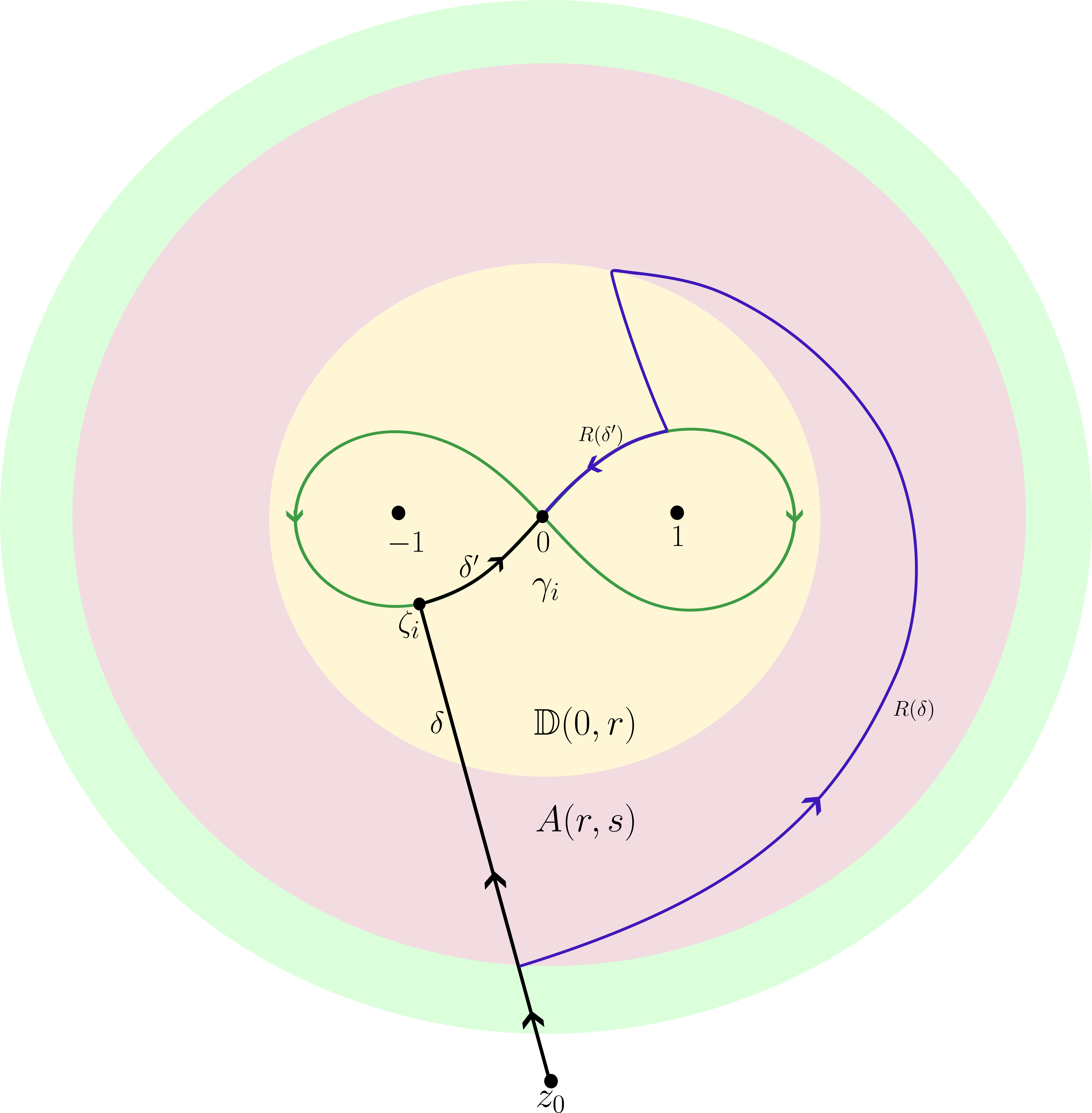}
\end{minipage}

\vspace{0.5cm}

\paragraph{Loop generators.}
For the loop generator $A_{i,j}$, map $b_{h_i}$ and
$b_{h_{j-1}+1}$ to $-1$ and $1$. The full twist
acts as $\theta=2\pi$.
The loop $\alpha$ encircles the origin clockwise, enclosing
both branch points and possessing monodromy
$k_\alpha = -(k_i+k_j)$. 

\noindent
\begin{minipage}[c]{0.65\textwidth}
Although $R_{2\pi}$ is pointwise the identity on the
inner disk, its real-valued annular angle lift is
nontrivial. It fixes the inner contour as a parametrized
curve, so
$\gamma'=\gamma_{i,j}$. The homology pushforward is
\begin{align*}
	(\wt{R_{2\pi}})_*(\wt\gamma_{i,j}^0) &= \wt\gamma_{i,j}^{k_i+k_j}.
\end{align*}

\begin{lemma}[Action of loop generators on their own class]
	\label{lem:mixed_self_action}
For $t \in \mu'(d)$ with $t\notin\mu(e_i)\cup\mu(e_j)$,
the loop generator $A_{i,j}$ acts on $\omega_{i,j}$ by
	\begin{align*}
		\rho_t(A_{i,j})(\omega_{i,j}) &= t^{k_i+k_j}\omega_{i,j}.
	\end{align*}
\end{lemma}
\begin{proof}
Recall the definition of the eigenspace lift from
	\eqref{eqn:mixed_atom}. The evaluation
$(\wt{R_{2\pi}})_*(\wt\gamma_{i,j}^0) = \wt\gamma_{i,j}^{k_i+k_j}$
on the base homology class directly yields the required
action on cohomology.
\end{proof}
\end{minipage}\hfill
\begin{minipage}[c]{0.32\textwidth}
	\centering
	\includegraphics[width=\linewidth]{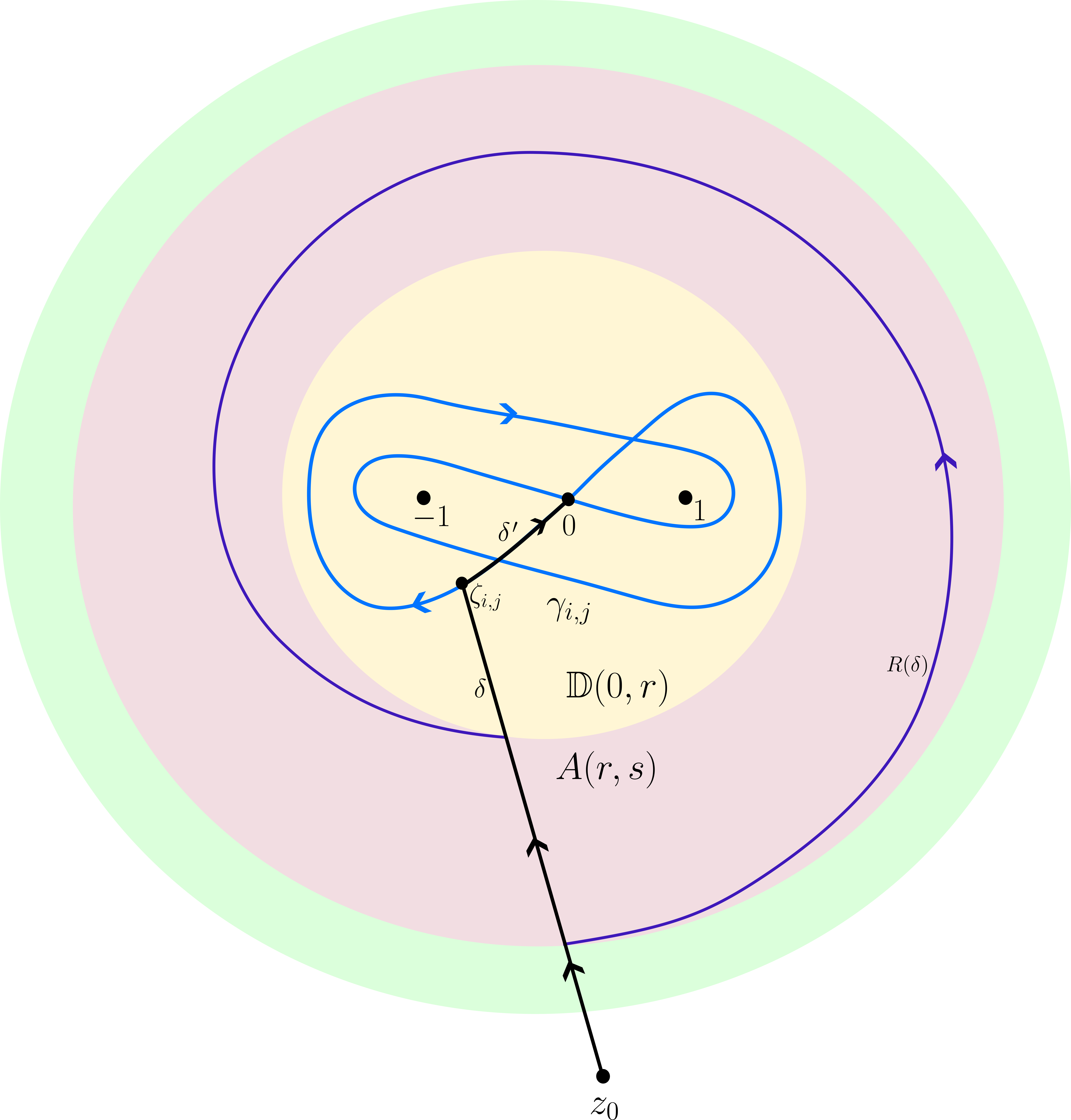}
\end{minipage}

\vspace{0.5cm}

\subsection{Action on intersecting classes}
\label{subsection:action_adjacent_classes}

Let $\sigma \in \bnp$ be a generator with
support disk $\dd_\sigma$ and
corresponding class $\omega_\sigma$, and
let $\omega$ be a spanning class whose
support disk $\dd$ satisfies
$\dd_\sigma \cap \dd \neq \emptyset$.

Choose a closed differential form
representing $\omega$ and denote it again
by $\omega$.  The canonical lift
$\wt\sigma$ restricts to the identity
outside a compact subset
$K \subset \wt\dd_\sigma$, so
$(\wt\sigma^{-1})^* \omega = \omega$ on
the complement of $K$.  Hence the
difference
$(\wt\sigma^{-1})^* \omega - \omega$
is a closed $1$-form with compact
support in $\wt\dd_\sigma$.  Its
cohomology class
$\rho_t(\sigma)(\omega) - \omega$
therefore lies in
$H^1_c(\wt\dd_\sigma)_t$. By
Proposition~\ref{Prop:Dim_Disk_2_branch_points},
this space has dimension at most $1$.

\begin{lemma}\label{lem:invisible_action}
If at least one branch point of
$\dd_\sigma$ lies in a $t$-invisible
block, then
$H^1_c(\wt\dd_\sigma)_t = 0$ by
Proposition~\ref{Prop:Dim_Disk_2_branch_points},
so $\omega_\sigma = 0$.
Consequently,
$\rho_t(\sigma)(\omega) = \omega$
for every $\omega$, i.e.\
$\rho_t(\sigma) = \id$.
\end{lemma}

The space $H^1_c(\wt\dd_\sigma)_t$ is spanned by
$\omega_\sigma$. Consequently,
\begin{align}\label{eqn:action_formula}
	\rho_t(\sigma)(\omega)
	= \omega + \nu(\omega)\omega_\sigma.
\end{align}
Assume $\dd_\sigma$ contains no branch point in a
$t$-invisible block, which guarantees 
$\omega_\sigma \neq 0$.

By
Lemmas~\ref{lem:artin_self_action}
and~\ref{lem:mixed_self_action},
$\sigma$ acts on its own class by
$\rho_t(\sigma)(\omega_\sigma)
= \lambda_\sigma \omega_\sigma$, where
$\lambda_\sigma
= -t^{k_{\pt(i)}}$ for an Artin
generator and
$\lambda_\sigma
= t^{k_i+k_j}$ for a loop generator.

The action $\rho_t(\sigma)$ preserves
the intersection form:
\begin{align*}
	\langle \omega, \omega_\sigma \rangle
	&= \langle \rho_t(\sigma)\omega,
	   \rho_t(\sigma)\omega_\sigma \rangle \\
	&= \langle \omega +
	   \nu(\omega)\omega_\sigma,
	   \lambda_\sigma\omega_\sigma \rangle \\
	&= \ov\lambda_\sigma
	   \langle \omega, \omega_\sigma \rangle
	+ \ov\lambda_\sigma \nu(\omega)
	   \langle \omega_\sigma,
	   \omega_\sigma \rangle.
\end{align*}
Multiplying by $\lambda_\sigma$ and rearranging yields
\begin{equation}\label{eqn:unitarity_relation}
	\nu(\omega) \langle \omega_\sigma, \omega_\sigma \rangle
	= (\lambda_\sigma - 1)\langle \omega, \omega_\sigma \rangle.
\end{equation}
Let
\begin{align*}
	R=\zz[\sqrt{-1}][t^{\pm1}]
	\bigl[(1-t^{k_i})^{-1}\;:\;i\notin\mathcal{I}_t\bigr].
\end{align*}
Every element of $R$ is regular at $t$, because $t\neq0$ and $t^{k_i}\neq1$
for $i\notin\mathcal{I}_t$. We compute $\nu(\omega)$ without using
\eqref{eqn:unitarity_relation}, which is vacuous when
$\langle\omega_\sigma,\omega_\sigma\rangle=0$, and show that it lies in $R$.

We use Poincar\'e duality and suppress $\eta$. Let $\gamma$ be the base
contour for $\omega$. The lift $\wt\sigma$ commutes with $T$ and fixes
$\wt z_0^{\,0}$, so $(\wt\sigma)_*\wt\gamma$ is the $t$-eigenspace cycle of
the contour $\sigma(\gamma)$, formed with the same constant $C_\gamma$ and
with lifts
$\wt{\sigma(\gamma)}^{\,g}=T^g_*\bigl((\wt\sigma)_*\wt\gamma^{\,0}\bigr)$.

Let $\omega'$ be a spanning class whose support disk meets $\dd_\sigma$, with
base contour $\gamma'$. Such a class exists when $\dim H^1(X)_t>1$. If
$\dim H^1(X)_t=1$, then $\omega$ is a multiple of $\omega_\sigma$ and
\eqref{eqn:action_formula} holds by
Lemmas~\ref{lem:artin_self_action} and~\ref{lem:mixed_self_action}. Pairing
\eqref{eqn:action_formula} with $\omega'$ gives
\begin{align}\label{eqn:nu_by_pairing}
	\nu(\omega)\,\langle\omega_\sigma,\omega'\rangle
	=\langle\rho_t(\sigma)(\omega),\omega'\rangle
	 -\langle\omega,\omega'\rangle .
\end{align}
By Lemma~\ref{lem:reduction_formula} and \eqref{eqn:crossing_sum}, applied to
$(\gamma,\gamma')$ and to $(\sigma(\gamma),\gamma')$, both terms on the right
of \eqref{eqn:nu_by_pairing} lie in $R$ and are independent of $d$.

By Propositions~\ref{prop:spanning_intersections}
and~\ref{prop:general_mixed_intersections},
$\langle\omega_\sigma,\omega'\rangle=\pm\sqrt{-1}\,(1-t^{\pm k})^{-1}$, a unit
of $R$. Dividing \eqref{eqn:nu_by_pairing},
\begin{align*}
	\nu(\omega)
	=\frac{\langle\rho_t(\sigma)(\omega),\omega'\rangle
	 -\langle\omega,\omega'\rangle}
	 {\langle\omega_\sigma,\omega'\rangle}
	\in R,
\end{align*}
again independent of $d$.

Regard $t$ as an indeterminate. Unitarity gives
\eqref{eqn:unitarity_relation} at every visible $t\in\mu'(d')$ with
$\langle\omega_\sigma,\omega_\sigma\rangle\neq0$, for each $d'$ with
$\gcd(d',k_1,\dots,k_m)=1$. Both sides are values of elements of $R$
independent of $d'$. Take $d'=dp$ with $p$ prime and $p>2\sum_r k_r$, and let
$t$ have order $p$. Let $k$ be any $k_i$ or any sum $k_i+k_j$. Then $0<2k<p$,
so $p$ divides neither $k$ nor $2k$, and therefore $t^k\neq1$ and $t^k\neq-1$.
With $k=k_i$ this gives $t^{k_i}\neq1$ for all $i$, so no block is
$t$-invisible. With $k=k_{\pt(i)}$ and $k=k_i+k_j$ it rules out both cases of
Remark~\ref{rem:isotropic}, so
$\langle\omega_\sigma,\omega_\sigma\rangle\neq0$. As $p$ varies these $t$ are
infinite in number, while a non-zero element of $R$ has finitely many zeros.
Hence \eqref{eqn:unitarity_relation} holds in $R$. The element
$\langle\omega_\sigma,\omega_\sigma\rangle$ takes a non-zero value at the $t$
above, so it is non-zero in $R$. It occurs on the left of
\eqref{eqn:unitarity_relation} only. Since $R$ is a domain, we may divide by
it in the fraction field of $R$. The resulting coefficient
\begin{align}
	\Lambda_\sigma = \frac{\lambda_\sigma - 1}{\langle \omega_\sigma, \omega_\sigma \rangle}
	\label{eqn:Lambda_sigma}
\end{align}
evaluates to
\begin{align*}
	\Lambda_\sigma = \begin{cases}
		           \sqrt{-1}\,(1-t^{k_{\pt(i)}})
		           & \text{if } \sigma = \sigma_i \text{ is an Artin generator}, \\[4pt]
		           \sqrt{-1}\,(1-t^{k_i})(1-t^{k_j})
		           & \text{if } \sigma = A_{i,j} \text{ is a loop generator}.
	           \end{cases}
\end{align*}
In both cases the quotient is a polynomial, so $\Lambda_\sigma$ lies in $R$.
In particular $\Lambda_\sigma$ is defined also where
$\langle\omega_\sigma,\omega_\sigma\rangle=0$. Hence
\begin{align*}
	\nu(\omega)
	=\Lambda_\sigma
	\langle\omega,\omega_\sigma\rangle
\end{align*}
in $R$, and therefore at every visible $t\in\mu'(d)$.

\begin{lemma}[Action on adjacent spanning classes]
	\label{lem:adjacent_spanning_class_action}
Let $\sigma$ be a generator with associated spanning class
$\omega_\sigma$, and let $\Lambda_\sigma$ be as in
	\eqref{eqn:Lambda_sigma}. For any adjacent spanning class
$\omega_i$, the action of $\sigma$ on $\omega_i$ is given by
	\begin{align*}
		\rho_t(\sigma)(\omega_i)
		= \omega_i + \Lambda_\sigma \langle \omega_i, \omega_\sigma \rangle \omega_\sigma.
	\end{align*}
\end{lemma}

\subsection{The representation theorem}
\label{subsection:representation_theorem}

We now assemble the representation $\rho_t$ from the
local computations above.  By
Lemma~\ref{lem:invisible_action}, a generator acts
as the identity whenever its support disk meets a
$t$-invisible block (justifying the term).  We may
therefore restrict to generators with no
$t$-invisible branch points in their support.  For
such a generator, the action on its own class is
given by Lemmas~\ref{lem:artin_self_action}
and~\ref{lem:mixed_self_action}, and the action on
any intersecting spanning class by
Lemma~\ref{lem:adjacent_spanning_class_action}.
Spanning classes with disjoint support are fixed.
These three cases determine $\rho_t$ on the
spanning classes.

\begin{thmA}[The Irreducible Representation of $\bnp$]\label{thm:main}
Let $t \in \mu'(d)$ be a non-trivial $d$-th root of unity, and let
$\rho_t \colon \bnp \to \Ugrp(H^1(X)_t, \langle\cdot,\cdot\rangle)$
be the monodromy representation. For a generator $\sigma$ of $\bnp$,
let $\omega_\sigma$ be the associated class supported in the support
disk of $\sigma$. Then
	\begin{align}\label{eqn:main_representation_non_isotropic}
		\rho_t(\sigma)(\omega)
		= \omega + \Lambda_\sigma \langle\omega,\omega_\sigma\rangle\,\omega_\sigma
	\end{align}
for all $\omega \in H^1(X)_t$, where the coefficient
	\begin{align*}
		\Lambda_\sigma
		=\frac{\lambda_\sigma-1}
		{\langle\omega_\sigma,\omega_\sigma\rangle}
	\end{align*}
evaluates to the polynomial
	\begin{align*}
		\Lambda_\sigma = \begin{cases}
			           \sqrt{-1}\,(1-t^{k_{\pt(i)}})
			           & \text{if } \sigma = \sigma_i \text{ is an Artin generator}, \\[4pt]
			           \sqrt{-1}\,(1-t^{k_i})(1-t^{k_j})
			           & \text{if } \sigma = A_{i,j} \text{ is a loop generator}.
		           \end{cases}
	\end{align*}
Moreover, $\rho_t$ is irreducible.
\end{thmA}

\begin{proof}
We verify \eqref{eqn:main_representation_non_isotropic} on
the spanning set of classes. For the associated class
$\omega_\sigma$, the formula yields
		\begin{align*}
			\rho_t(\sigma)(\omega_\sigma)
			=\bigl(1+\Lambda_\sigma
			\langle\omega_\sigma,\omega_\sigma\rangle\bigr)
			\omega_\sigma
			=\lambda_\sigma\omega_\sigma,
		\end{align*}
matching Lemmas~\ref{lem:artin_self_action} and
	\ref{lem:mixed_self_action}. For an adjacent class
$\omega_i$, the formula recovers
Lemma~\ref{lem:adjacent_spanning_class_action}. For a
class $\omega_i$ with disjoint support,
$\langle \omega_i, \omega_\sigma \rangle = 0$ implies
$\rho_t(\sigma)(\omega_i) = \omega_i$, matching the
identity action. This proves the formula on the
spanning set and hence on $H^1(X)_t$.
Substitution into
		\begin{align*}
			\frac{\lambda_\sigma-1}
			{\langle\omega_\sigma,\omega_\sigma\rangle}
		\end{align*}
gives the stated value of $\Lambda_\sigma$.

For irreducibility, let
$0\neq W\subseteq H^1(X)_t$ be invariant, and choose
$0\neq\omega\in W$. The Hermitian form on
$H^1(X)_t$ is non-degenerate, and the classes
$\{\omega_i\}$ span. Hence there is an index $i$
such that
	\begin{align*}
		\langle\omega,\omega_i\rangle\neq0.
	\end{align*}
Let $\sigma$ be the generator associated with
$\omega_i$. Visibility gives $\Lambda_\sigma\neq0$,
so
	\begin{align*}
		\rho_t(\sigma)(\omega)-\omega
		=\Lambda_\sigma
		\langle\omega,\omega_i\rangle\omega_i
		\in W.
	\end{align*}
Therefore $\omega_i\in W$.

The graph whose vertices are the spanning classes and
whose edges record non-zero pairings is connected. Its
consecutive off-diagonal pairings are non-zero because
all surviving endpoint blocks are $t$-visible. If
$\omega_j\in W$ and
$\langle\omega_j,\omega_k\rangle\neq0$, applying the
generator associated with $\omega_k$ gives
$\omega_k\in W$. Connectivity therefore places every
spanning class in $W$. Since these classes span
$H^1(X)_t$, we conclude $W=H^1(X)_t$.
\end{proof}

\begin{remark}[Factorization for invisible blocks]
	\label{rem:factorization_invisible}
Let $t \in \mu'(d)$. The notations
$n_t$, $\mathcal{P}_t$ and the forgetful
homomorphism
$\forgmap_t \colon \bnp \to
B_{n_t, \mathcal{P}_t}$ are defined in
Subsection~\ref{subsection:mixed_braid_group_definition}
	(see \eqref{eqn:forgmap_t} and the kernel
description there).
By \hyperref[thm:main]{Theorem~A},
$t^{k_i}=1$ forces $\Lambda_\sigma = 0$ for
every such generator, hence
$\rho_t(\sigma) = \id$.
By Subsection~\ref{subsection:mixed_braid_group_definition},
$\ker(\forgmap_t)$ is generated as a normal subgroup by these
generators.
Since $\ker(\rho_t)$ is normal,
$\ker(\forgmap_t) \subseteq \ker(\rho_t)$,
and $\rho_t$ factors uniquely through
$\forgmap_t$:
	\begin{equation*}
		\begin{tikzcd}
			\bnp \arrow[r, "\forgmap_t"]
			\arrow[rd, "\rho_t"'] &
B_{n_t, \mathcal{P}_t}
			\arrow[d, "\bar{\rho}_t"] \\
			& \Ugrp(H^1(X)_t)
		\end{tikzcd}
	\end{equation*}
The factored representation $\bar\rho_t$
acts on $H^1(X)_t$ with
$\dim H^1(X)_t = n_t - 1$
	(when $t^{k_\infty} \neq 1$).
\end{remark}

\begin{proposition}
	\label{prop:bar_rho_geometric}
Let $e_t = \operatorname{ord}(t)$.
The representation $\bar{\rho}_t$ is the
geometric monodromy representation of
$B_{n_t, \mathcal{P}_t}$ acting on
$H^1(X_{vis})_t$, where $X_{vis} = X/K$
is the cyclic cover of $\pp^1$ of degree
$e_t$ branched only over the visible
blocks. Here
$K = \langle T^{e_t} \rangle \subset G$
is the subgroup of the deck group that
acts trivially on $H^1(X)_t$.
\end{proposition}

\begin{proof}
Since $T$ acts on $H^1(X)_t$ by
multiplication by $t$, every element of
$K = \langle T^{e_t} \rangle$ acts
trivially. The quotient $X_{vis} = X/K$
has deck group
$G/K \cong \zz/e_t\zz$.

For a branch point $b_i$, the local
monodromy $T^{k_{\pt(i)}}$ maps to the
identity in $G/K$ if and only if $e_t$
divides $k_{\pt(i)}$, i.e.\
$t^{k_{\pt(i)}} = 1$. Thus
$b_i$ is unbranched in $X_{vis}$ exactly
when $t^{k_{\pt(i)}} = 1$.
The normalized equation of $X_{vis}$ is
	\begin{align*}
w^{e_t}
		= \prod_{i:\, t^{k_{\pt(i)}} \neq 1}
		(z - b_i)^{k_{\pt(i)}},
	\end{align*}
since the factors with
$e_t \mid k_{\pt(i)}$ are perfect
$e_t$-th powers and are absorbed
into the coordinate.

Because $K$ acts trivially on $H^1(X)_t$,
pullback along $X \to X_{vis}$ restricts
to an isomorphism
$H^1(X_{vis})_t
	\xrightarrow{\sim} H^1(X)_t$.
The canonical lifts of braids in $\bnp$
commute with $K$ and descend to $X_{vis}$.
Under this isomorphism, the monodromy of
$B_{n_t, \mathcal{P}_t}$ on
$H^1(X_{vis})_t$ is $\bar\rho_t$.
\end{proof}

\begin{corollary}
No block of $\mathcal{P}_t$ is $t$-invisible, so $\omega_\sigma\neq0$
for every generator $\sigma$ of $B_{n_t,\mathcal{P}_t}$. The linear
transformation $\bar\rho_t(\sigma)$ acts as a complex reflection when
the associated class $\omega_\sigma$ is non-isotropic, and degenerates
into a unitary transvection
when $\omega_\sigma$ is isotropic.
\end{corollary}

\begin{proof}
By Proposition~\ref{prop:bar_rho_geometric}, $\bar\rho_t$ is the
monodromy representation of $B_{n_t,\mathcal{P}_t}$ on
$H^1(X_{vis})_t$, so \hyperref[thm:main]{Theorem~A} applies to it.
Write $T=\bar\rho_t(\sigma)$. By
\eqref{eqn:main_representation_non_isotropic}, $T$ fixes the
orthogonal complement $\omega_\sigma^\perp$ pointwise.
When $\omega_\sigma$ is non-isotropic
	($\langle \omega_\sigma, \omega_\sigma \rangle \neq 0$),
the class $\omega_\sigma$ does not lie in its orthogonal
complement and $T(\omega_\sigma) = \lambda_\sigma \omega_\sigma$
with $\lambda_\sigma \neq 1$. The space decomposes as
$H^1(X)_t = \omega_\sigma^\perp \oplus \langle \omega_\sigma \rangle$,
proving $T$ is a diagonalizable complex reflection.

When $\omega_\sigma$ is isotropic
	($\langle \omega_\sigma, \omega_\sigma \rangle = 0$),
$\lambda_\sigma = 1$ and
	\begin{align*}
		(T-I)(\omega)
		=\Lambda_\sigma
		\langle\omega,\omega_\sigma\rangle
		\omega_\sigma.
	\end{align*}
Since $\omega_\sigma \in \omega_\sigma^\perp$,
$(T - I)^2 = 0$. Here $\Lambda_\sigma\neq0$ because no block of
$\mathcal{P}_t$ is $t$-invisible, and
$\langle\,\cdot\,,\omega_\sigma\rangle$ is not identically zero
because the form is non-degenerate and $\omega_\sigma\neq0$.
Hence $T \neq I$ is a transvection.
\end{proof}

\section{Identification with the Multivariate
Burau Representation}
\label{section:multivariate_burau_identification}

We establish the connection between the geometric
monodromy representation $\rho_t$ and the
multivariate Burau representation.
By Remark~\ref{rem:factorization_invisible} and
Proposition~\ref{prop:bar_rho_geometric},
$\rho_t$ factors through $\bar\rho_t$, which is
the geometric monodromy of the reduced mixed
braid group $B_{n_t,\mathcal{P}_t}$ acting on the
visible-only cover $X_{vis}$. Therefore, to compute
the matrices and identify them with the multivariate
Burau representation, we may assume without loss of
generality for the remainder of this section that
there are no $t$-invisible blocks
($\mathcal{I}_t=\varnothing$).
Under this assumption, we have $n_t = n$,
$\mathcal{P}_t = \mathcal{P}$, $X_{vis} = X$, and
$\bar{\rho}_t = \rho_t$. This simplifies notation
while preserving full generality.
We also assume throughout this section that
$t^{k_\infty}\neq1$, as in Theorem~B.

\subsection{Matrix form of the geometric
monodromy}
\label{subsection:matrix_form_geometric}

We present the explicit matrix form of
$\rho_t$ relative to the ordered basis
$\{\omega_1, \dots, \omega_{n-1}\}$
of $H^1(X)_t$. Indeed, these classes span by
Subsection~\ref{subsection:intersection_matrix}, and our
assumptions give $\dim H^1(X)_t=n-1$.

For an Artin generator $\sigma_i$ of
$\bnp$, the geometric
action
\eqref{eqn:main_representation_non_isotropic}
evaluates to the matrix
\begin{equation}
\label{eq:geometric_artin_matrix}
	[\rho_t(\sigma_i)] =
	\begin{pmatrix}
I_{i-2} & \mathbf{0} &
		\mathbf{0} \\
			\mathbf{0} & \begin{pmatrix}
				1 & 0 & 0 \\
				1 & -t^{k_{\pt(i)}} & t^{k_{\pt(i)}} \\
				0 & 0 & 1
			\end{pmatrix} & \mathbf{0} \\
			\mathbf{0} & \mathbf{0} &
I_{n-i-2}
		\end{pmatrix}.
\end{equation}
For $i=1$ or $i=n-1$, nonexistent rows and columns in
the displayed local block are omitted.

For the mixed loop generator $A_{i,j}$,
we distinguish two cases.

\medskip\noindent
\textbf{Adjacent loop generators
($j = i+1$ in $\mathcal{P}$).}
The associated class is
$\omega_\sigma = \omega_{h_i}$, where
$h_i$ indexes the boundary between blocks
$P_i$ and $P_{i+1}$ in the partition.
By \hyperref[thm:main]{Theorem~A} and
Proposition~\ref{prop:spanning_intersections},
the non-zero entries are confined to rows
$\{h_i-1, h_i, h_i+1\}$; the submatrix
on those rows and columns
$\{h_i-1, h_i, h_i+1\}$ is
\begin{equation}
\label{eq:geometric_loop_adjacent}
	[\rho_t(A_{i,i+1})]_{\text{sub}}
	= \begin{pmatrix}
		1 & 0 & 0 \\[4pt]
		1 - t^{k_{i+1}}
		& t^{k_i+k_{i+1}}
		& t^{k_{i+1}}(1 - t^{k_i}) \\[4pt]
		0 & 0 & 1
		\end{pmatrix}.
\end{equation}
If $h_i=1$ or $h_i=n-1$, nonexistent boundary rows and
columns are omitted.

\medskip\noindent
\textbf{General loop generators
($j > i+1$ in $\mathcal{P}$).}
For a general loop generator $A_{i,j}$, the associated 
geometric class is $\omega_{i,j}$.
By Proposition~\ref{prop:general_mixed_intersections},
the intersection pairings $\lab \omega_r, \omega_{i,j} \rab$
vanish for
all interior classes $h_i < r < h_{j-1}$.
By Proposition~\ref{prop:mixed_linear_combination},
the non-spanning mixed class decomposes as a sum
of the spanning classes:
\begin{equation}
	\omega_{i,j} = \sum_{\ell=h_i}^{h_{j-1}} \omega_\ell.
\end{equation}
Substitute this linear combination into the geometric
action formula and use
Proposition~\ref{prop:general_mixed_intersections}. The
nonzero coefficients are
\begin{equation*}
	\Lambda_{A_{i,j}} \langle \omega_r, \omega_{i,j} \rangle = \begin{cases}
		1-t^{k_j} & r = h_i-1, \\[4pt]
		-t^{k_i}(1-t^{k_j}) & r = h_i, \\[4pt]
		-(1-t^{k_i}) & r = h_{j-1}, \\[4pt]
t^{k_j}(1-t^{k_i}) & r = h_{j-1}+1, \\[4pt]
		0 & \text{otherwise}.
	\end{cases}
\end{equation*}
Expanding $\omega_{i,j}$ in adjacent classes gives the
following submatrix on rows and columns $h_i-1$ through
$h_{j-1}+1$:
\begin{equation}
\label{eq:geometric_loop_general}
	[\rho_t(A_{i,j})]_{\text{sub}}
	= \begin{pmatrix}
		1 & 0 & 0 & \dots & 0 & 0 & 0\\[6pt]
		1-t^{k_j} & 1 - t^{k_i}(1-t^{k_j}) & 0 & \dots & 0 & -(1-t^{k_i}) & t^{k_j}(1-t^{k_i})\\[6pt]
		1-t^{k_j} & -t^{k_i}(1-t^{k_j}) & 1 & \dots & 0 & -(1-t^{k_i}) & t^{k_j}(1-t^{k_i})\\[6pt]
		\vdots & \vdots & \vdots & \ddots & \vdots & \vdots & \vdots \\[6pt]
		1-t^{k_j} & -t^{k_i}(1-t^{k_j}) & 0 & \dots & 1 & -(1-t^{k_i}) & t^{k_j}(1-t^{k_i})  \\[6pt]
		1-t^{k_j} & -t^{k_i}(1-t^{k_j}) & 0 & \dots & 0 & t^{k_i} & t^{k_j}(1-t^{k_i})  \\[6pt]
		0 & 0 & 0 & \dots & 0 & 0 & 1
	\end{pmatrix}.
\end{equation}
If $h_i=1$ or $h_{j-1}+1=n$, nonexistent boundary rows
and columns are omitted.

\subsection{The multivariate Burau
representation}
\label{subsection:burau_fox}

The mixed braid group
$\bnp$ acts on the free
group
$F_n
= \langle x_1, \dots, x_n \rangle$
by left automorphisms, giving an embedding
$\bnp
\hookrightarrow \aut(F_n)$.
Under this embedding, the inverse action of
an Artin generator is
\begin{equation}
\label{eq:artin_inverse_action}
	\sigma_i^{-1}(x_i)
	= x_i x_{i+1} x_i^{-1},\qquad
	\sigma_i^{-1}(x_{i+1}) = x_i,\qquad
	\sigma_i^{-1}(x_k) = x_k
	\;(k \neq i, i+1),
\end{equation}
and the inverse of a loop generator
$A_{i,j}$ is given by
\begin{equation}
\label{eq:loop_inverse_action}
A_{i,j}^{-1}(x_k) =
	\begin{cases}
x_k
		& k < h_i \;\text{or}\;
k > h_{j-1}+1,\\[4pt]
x_{h_i} x_{h_{j-1}+1}
x_{h_i}
x_{h_{j-1}+1}^{-1}
x_{h_i}^{-1}
		& k = h_i,\\[4pt]
x_{h_i} x_{h_{j-1}+1}
x_{h_i}^{-1}
		& k = h_{j-1}+1,\\[4pt]
		[x_{h_i}, x_{h_{j-1}+1}]\,
x_k\,
		[x_{h_i}, x_{h_{j-1}+1}]^{-1}
		& h_i < k < h_{j-1}+1.
	\end{cases}
\end{equation}
(see \cite[Corollary~1.8.3]{Birman}).

For a braid $\sigma$,
the universal Fox derivative
$\partial/\partial x_j$
yields the matrix
$\partial(\sigma^{-1}(x_i))/\partial x_j$
with entries in $\cc[F_n]$.
The block specialization
$\phi_{\mathcal{P}}
\colon F_n \to
\langle y_1 \rangle
\times \cdots \times
\langle y_m \rangle$,
defined by
\begin{align*}
	\phi_{\mathcal{P}}(x_k)
	= (1,\dots,1,y_{\pt(k)},1,\dots,1),
\end{align*}
with $y_{\pt(k)}$ in the $\pt(k)$-th
coordinate, is $\bnp$-invariant
(strands in the same block share the same
parameter, and mixed braids permute only
within blocks).
Applying $\phi_{\mathcal{P}}$
coefficient-wise to the universal matrix
gives the multivariate Burau representation:
\begin{equation}
\label{eq:automorphism_matrix_def}
	[\beta(\sigma)]_{i,j}
	= \phi_{\mathcal{P}}\!\left(
	\frac{\partial(\sigma^{-1}(x_i))}
	{\partial x_j} \right)
	\in \GL(n,
	\cc[y_1^{\pm 1}, \dots,
y_m^{\pm 1}]).
\end{equation}
This is the multivariate Burau representation
of $\bnp$ with $m$ formal parameters.

To connect with the geometric representation
$\rho_t$, we evaluate each formal
parameter at the fixed $d$-th root of
unity $t$ by setting
$y_r = t^{k_r}$ for $r = 1, \dots, m$.
The matrices below are written in this
evaluated form
(see \cite[eq.~3--23 and ~3--24]{Birman}),
and we write $[\beta_t(\sigma)]$ for the
matrix of a braid $\sigma$ after this
evaluation, computed with respect to
the standard basis
$\{u_1, \dots, u_n\}$ of $\cc^n$.

For an Artin generator $\sigma_i$,
the $2 \times 2$ submatrix on rows
$\{i, i+1\}$ of $[\beta_t(\sigma_i)]$ is
\begin{equation}
\label{eq:matrix_beta_artin}
	[\beta_t(\sigma_i)]_{\text{sub}}
	= \begin{pmatrix}
		1 - t^{k_{\pt(i)}}
		& t^{k_{\pt(i)}} \\
		1 & 0
	\end{pmatrix}.
\end{equation}

The mixed loop generator $A_{i,j}$ has the
block form
\begin{equation}
\label{eq:matrix_beta_explicit}
	[\beta_t(A_{i,j})] = \begin{pmatrix}
I_{h_i-1} & \mathbf{0}
		& \mathbf{0} \\
		\mathbf{0}
		& [\beta_t(A_{i,j})]_{\text{sub}}
		& \mathbf{0} \\
		\mathbf{0} & \mathbf{0}
		& I_{n - h_{j-1}-1}
	\end{pmatrix},
\end{equation}
where the submatrix spanning indices
$h_i$ to $h_{j-1}+1$ is
\begin{equation}
\label{eq:matrix_beta_sub}
	[\beta_t(A_{i,j})]_{\text{sub}}
	= \begin{pmatrix}
		1 - t^{k_i}+ t^{k_i+k_j}
		& 0 & 0 & \dots & 0 & 0
		& t^{k_i}(1-t^{k_i}) \\
		(1-t^{k_{i+1}})(1-t^{k_j})
		& 1 & 0 & \dots & 0 & 0
		& -(1-t^{k_{i+1}})(1-t^{k_i})
		\\
		(1-t^{k_{\pt(h_{i}+2)}})(1-t^{k_j})
		& 0 & 1 &\dots & 0 & 0
		& -(1-t^{k_{\pt(h_{i}+2)}})(1-t^{k_i})
		\\
		\vdots & \vdots & \vdots
		& \ddots & \vdots \\
		(1-t^{k_{\pt(h_{j-1}-1)}})(1-t^{k_j})
		& 0 & 0 & \dots & 1 & 0
		& -(1-t^{k_{\pt(h_{j-1}-1)}})(1-t^{k_i})
		\\
		(1-t^{k_{j-1}})(1-t^{k_j})
		& 0 & 0 & \dots & 0 & 1
		& -(1-t^{k_{j-1}})(1-t^{k_i}) \\
		1 - t^{k_j}
		& 0 & 0 & \dots & 0 & 0
		& t^{k_i}
	\end{pmatrix}.
\end{equation}

Birman's explicit computation following
\cite[eq.~3--24]{Birman} contains a
typographical error.
The submatrix~\eqref{eq:matrix_beta_sub}
uses the corrected form.
The matrices displayed in
\eqref{eq:automorphism_matrix_def}--%
\eqref{eq:matrix_beta_sub} are the
unreduced $n \times n$ Burau matrices.

\begin{lemma}
\label{lem:invariant_vector}
Let $\{u_1, \dots, u_n\}$ be the
standard basis for the unreduced
Burau space $\cc^n$. The vector
$v_{\mathrm{inv}}
	= \sum_{j=1}^n
	(1 - t^{k_{\pt(j)}}) u_j$
spans a $1$-dimensional invariant
subspace under the unreduced action
$[\beta_t(\sigma)]$.
\end{lemma}

\begin{proof}
One verifies directly that
	\begin{align*}
		1 - t^{k_{\pt(i)}} 
		= \sum_{j=1}^n
		[\beta_t(\sigma)]_{ij}
		(1 - t^{k_{\pt(j)}}).
	\end{align*}
Hence $v_{\mathrm{inv}}$ is an
eigenvector of $[\beta_t(\sigma)]$
with eigenvalue $1$.
\end{proof}

For each $1 \le i \le n-1$, the
reduced basis vectors are
\begin{equation}
\label{eq:reduced_basis}
v_i = u_{i+1} - t^{k_{\pt(i)}} u_i.
\end{equation}
Put $y_j=t^{k_{\pt(j)}}$ and
$W=\operatorname{span}\{v_1,\dots,v_{n-1}\}$.
The vectors $v_1,\dots,v_{n-1}$ are linearly independent.
Define $L\colon\cc^n\to\cc$ by
\begin{align*}
	L(u_j)=\prod_{a<j}y_a.
\end{align*}
Then $L(v_i)=0$ for every $i$. Since $L(u_1)=1$, both
$W$ and $\ker L$ have dimension $n-1$, so $W=\ker L$.
Moreover, the telescoping sum gives
\begin{align*}
	L(v_{\mathrm{inv}})
	&=\sum_{j=1}^n
	  \left(\prod_{a<j}y_a\right)(1-y_j) \\
	&=1-\prod_{j=1}^n y_j
	 =1-t^{-k_\infty}\neq0.
\end{align*}
Hence
\begin{align*}
	\cc^n
	=W
	\oplus \cc v_{\mathrm{inv}}.
\end{align*}
A direct check from
\eqref{eq:matrix_beta_artin} and
\eqref{eq:matrix_beta_sub} shows that $W$ is invariant.
The restriction to $W$ is therefore naturally isomorphic
to the usual reduced quotient
$\cc^n/\cc v_{\mathrm{inv}}$; we denote it by
$\wt\beta_t$.

\subsection{The isomorphism theorem}
\label{subsection:burau_isomorphism}

\begin{thmB}[Isomorphism with the
Multivariate Burau Representation]
\label{thm:burau}
Let $t \in \mu'(d)$ with
$t^{k_\infty} \neq 1$, and let
$\bar{\rho}_t$ be the factored
representation defined in
Remark~\ref{rem:factorization_invisible}.
Then $\bar{\rho}_t$ is isomorphic
to the reduced multivariate Burau
representation $\wt\beta_t$ of
$B_{n_t, \mathcal{P}_t}$.
\end{thmB}

\begin{proof}
As established at the beginning of the section,
we may assume $\mathcal{I}_t = \varnothing$, so
$\bar\rho_t=\rho_t$ and it suffices to prove
$\rho_t\cong\wt\beta_t$ on $\bnp$.

Let $P$ be the $n \times n$ matrix
with columns
$v_1, \dots, v_{n-1},
v_{\mathrm{inv}}$ as defined in
\eqref{eq:reduced_basis}.
It suffices to verify
\begin{align*}
	[\beta_t(\sigma)]P
	=P\bigl(\rho_t(\sigma)\oplus1\bigr)
\end{align*}
for each generator $\sigma\in\bnp$. By
Lemma~\ref{lem:invariant_vector}, every generator fixes
$v_{\mathrm{inv}}$, so it remains to compare its action
on $v_r$ with the geometric action on $\omega_r$.
Terms with indices outside $\{1,\dots,n-1\}$ are omitted
below.

\medskip\noindent
\textbf{Artin generators.}
Direct substitution in \eqref{eq:matrix_beta_artin} gives
\begin{align*}
	\beta_t(\sigma_i)v_{i-1}&=v_{i-1}+v_i, \\
	\beta_t(\sigma_i)v_i&=-t^{k_{\pt(i)}}v_i, \\
	\beta_t(\sigma_i)v_{i+1}
	&=v_{i+1}+t^{k_{\pt(i)}}v_i,
\end{align*}
and all other $v_r$ are fixed. These are precisely the
columns of \eqref{eq:geometric_artin_matrix}.

\medskip\noindent
\textbf{Adjacent loop generators.}
Evaluation on the affected basis vectors gives
\begin{align*}
	\beta_t(A_{i,i+1}) v_{h_i-1} &= v_{h_i-1} + (1-t^{k_{i+1}}) v_{h_i}, \\[4pt]
	\beta_t(A_{i,i+1}) v_{h_i} &= t^{k_i+k_{i+1}} v_{h_i}, \\[4pt]
	\beta_t(A_{i,i+1}) v_{h_i+1} &= v_{h_i+1} + t^{k_{i+1}}(1-t^{k_i}) v_{h_i}.
\end{align*}
All other $v_r$ are fixed, and these formulas agree with
\eqref{eq:geometric_loop_adjacent}.

\medskip\noindent
\textbf{General loop generators.}
For $j>i+1$, set
$S_{i,j}=\sum_{\ell=h_i}^{h_{j-1}}v_\ell$. Then
\begin{align*}
	\beta_t(A_{i,j})v_{h_i-1}
	&=v_{h_i-1}+(1-t^{k_j})S_{i,j}, \\
	\beta_t(A_{i,j})v_{h_i}
	&=v_{h_i}-t^{k_i}(1-t^{k_j})S_{i,j}, \\
	\beta_t(A_{i,j})v_{h_{j-1}}
	&=v_{h_{j-1}}-(1-t^{k_i})S_{i,j}, \\
	\beta_t(A_{i,j})v_{h_{j-1}+1}
	&=v_{h_{j-1}+1}+t^{k_j}(1-t^{k_i})S_{i,j}.
\end{align*}
All remaining $v_r$ are fixed. These formulas agree with
\eqref{eq:geometric_loop_general}.

Thus $P$ intertwines the Burau and geometric matrices for
every generator, and restriction to
$\operatorname{span}\{v_1,\dots,v_{n-1}\}$ gives
$\wt\beta_t\cong\rho_t$ in the visible-only case. By the
reduction at the beginning of the section, this is precisely
$\wt\beta_t\cong\bar\rho_t$ for the original cover.
\end{proof}

When no block is $t$-invisible,
$\mathcal{I}_t = \varnothing$,
$\forgmap_t = \id$, and
$\bar{\rho}_t = \rho_t$.
Theorem~B then asserts
$\rho_t \cong \wt\beta_t$
on $B_{n, \mathcal{P}}$.

\begin{remark}
In the simple-cover case, McMullen~\cite{MR3020148} uses
the $q$-eigenspace of $T^*$, whereas we use the
$t$-eigenspace of $(T^{-1})^*$. Thus $q=t^{-1}$, and
McMullen's Theorem~5.5 gives
$\rho_t\cong\wt\beta_{t^{-1}}^\vee
\cong\wt\beta_t$, in agreement with Theorem~B.
\end{remark}

\subsection*{Acknowledgements}
The authors thank Chitrabhanu Chaudhuri for discussions
that contributed to this work.

\printbibliography

\end{document}